\documentclass[a4paper,11pt,reqno]{amsart}
\usepackage{amsmath,amsthm,amssymb,upgreek,tipa}
\usepackage{mathabx}
\usepackage[latin1]{inputenc}
\usepackage[T1]{fontenc}
\usepackage[all]{xy}
\usepackage[colorlinks=true,linkcolor=blue,pagebackref,breaklinks]{hyperref}
\usepackage[usenames]{color}
\usepackage{ulem}
\usepackage{url}
\usepackage{marginnote}
\usepackage{enumitem}

\newtheorem{thm}[equation]{Theorem}
\newtheorem*{thm*}{Theorem}

\newtheorem{cor}[equation]{Corollary}
\newtheorem{lem}[equation]{Lemma}
\newtheorem{prop}[equation]{Proposition}
\newtheorem{conj}[equation]{Conjecture}
\newtheorem{hyp}[equation]{Hypothesis}

\theoremstyle{definition}

\newtheorem{rem}[equation]{Remark}
\newtheorem{defn}[equation]{Definition}
\newtheorem{defn-lem}[equation]{Definition-Lemma}
\newtheorem{rmk}[equation]{Remark}
\newtheorem*{rem*}{Remark}
\newtheorem{obsv}[equation]{Observation}
\newtheorem{sub}[equation]{ }

\DeclareMathOperator{\h}{H}

\def\R{\mathbb R}
\def\N{\mathbb N}
\def\Z{\mathbb Z}
\def\A{\mathbb A}
\def\Q{\mathbb Q}
\def\C{\mathbb C}
\def\CC{\mathbb CC}
\def\RR{\mathbb R}
\def\QQ{\mathbb Q}
\def\Qbar{\bar{\QQ}}
\def\CC{\mathbb C}
\def\Gm{{\mathbb G}_m}

\def\EE{\mathcal E}
\def\cF{\mathcal F}

\def\ira{\stackrel{\sim}{\longrightarrow}}
\def\hra{\hookrightarrow}
\def\ra{\rightarrow}

\def\g{\mathfrak g}

\def\q{\mathfrak q}

\def\O{\mathcal O}
\def\h{\mathfrak h}

\def\k{\mathfrak k}
\def\p{\mathfrak p}

\def\<{\langle}
\def\>{\rangle}

\def\W{\mathcal W}

\def\Acm{\mathbf{A}}

\newcommand{\Ss}{{\mathbb S}}

\def\Hom{{\rm Hom}}
\def\GL{{\rm GL}}

\def\CL{\mathcal{L}}
\def\CA{\mathcal{A}}

\def\CP{\mathcal{P}}
\def\CO{\mathcal{O}}
\def\cm{F}
\def\tr{F^{+}}
\def\Acm{\mathbb{A}_{F}}

\def\Atr{\mathbb{A}_{F^{+}}}
\def\Atrf{\mathbb{A}_{F^{+},f}}
\newcommand{\isoarrow}{{~\overset\sim\longrightarrow}}

\def\HC{A} %harish-chandra or infinity type

\begin{document}

\title[]{Deligne's conjecture for automorphic motives over CM-fields}
\author{Harald Grobner, Michael Harris \& Jie Lin}
\thanks{H.G. is supported by START-prize Y-966 of the Austrian Science Fund (FWF) and the FWF Stand-alone research project P32333. M.H.'s research received funding from the European Research Council under the European Community's Seventh Framework Program (FP7/2007-2013) / ERC Grant agreement no. 290766 (AAMOT). M.H. was partially supported by NSF Grants DMS-1404769 and DMS-1701651.  This work was also supported by the National Science Foundation under Grant No. DMS-1440140 while M.H. was in residence at the Mathematical Sciences Research Institute in Berkeley, California, during the Spring 2019 semester. J.L. was supported by the European Research Council under the European Community's Seventh Framework Programme (FP7/2007-2013) / ERC Grant agreement no. 290766 (AAMOT)}

\maketitle

\begin{abstract}
The present paper is devoted to the relations between Deligne's conjecture on critical values of motivic $L$-functions and the multiplicative relations between periods of arithmetically normalized automorphic forms on unitary groups.  In the first place, we combine the Ichino--Ikeda--Neal-Harris (IINH) formula -- which is now a theorem -- with an analysis of cup products of coherent cohomological automorphic forms on Shimura varieties to establish relations between certain automorphic periods and critical values of Rankin-Selberg and Asai $L$-functions of $\GL(n)\times\GL(m)$ over CM fields.  By reinterpreting these critical values in terms of automorphic periods of holomorphic automorphic forms on unitary groups, we show that the automorphic periods of holomorphic forms can be factored as products of coherent cohomological forms, compatibly with a motivic factorization predicted by the Tate conjecture.  All of these results are conditional on a conjecture on non-vanishing of twists of automorphic $L$-functions of $\GL(n)$ by anticyclotomic characters of finite order, and are stated under a certain regularity condition.
\end{abstract}

\setcounter{tocdepth}{1}
\tableofcontents

\section*{Introduction}\label{intro}

This  paper is devoted to the relations between two themes.  The first theme is Deligne's conjecture on critical values of motivic $L$-functions.  Our first main theorem is an expression of certain critical values -- those of the $L$-functions of tensor products of motives attached to cohomological automorphic forms on unitary groups -- in terms of  periods of arithmetically normalized automorphic forms on the Shimura varieties attached to these unitary groups.   We refer to these periods as {\it automorphic periods} for the remainder of the introduction.  The full result, like most of the results contained in this  paper, is conditional on a pair of conjectures that will be described below, as well as a relatively mild list of restrictions (see \ref{localhyp}) on the local components of the $L$-functions considered.  Let $F$ be a CM field.  Here is the statement of our {\bf automorphic} version  (Thm. \ref{automorphic Deligne general}) of Deligne's conjecture, over $F$:

\begin{thm} \label{automorphic Deligne general intro} 
Let $n,n'\geq 1$ be integers and let $\Pi$ (resp. $\Pi'$) be a cohomological conjugate self-dual cuspidal automorphic representation of $\GL(n,\Acm)$ (resp. $\GL(n',\Acm)$), which satisfies Hyp.\ \ref{descent}.
 If $n\equiv n' \mod 2$, we assume that $\Pi$ and $\Pi'$ satisfy the conditions of Thm.\ \ref{automorphic Deligne near central}, i.e., that the isobaric sum $(\Pi\eta^n)\boxplus (\Pi'^{c}\eta^n)$ is $2$-regular and that either $\Pi$ and $\Pi'$ are both $5$-regular or $\Pi$ and $\Pi'$ are both regular and satisfy Conj.\ \ref{nonvan}. Whereas if $n\nequiv n' \mod 2$, we assume that $\Pi$ and $\Pi'$ satisfy the conditions of Thm.\ \ref{automorphic Deligne central}, i.e., we assume Conj.\ \ref{nonvan} and suppose that $\Pi_\infty$ is $(n-1)$-regular and $\Pi'_\infty$ is $(n'-1)$-regular.
Then the automorphic version of Deligne's conjecture, cf.\ Conj.\ \ref{main conjecture}, is true:  If $s_0$ is a critical value of   $L(s,\Pi\times \Pi')$, then the value at $s_0$ of the partial $L$-function $L^S(s,\Pi\times \Pi')$ (for some appropriate finite set $S$), satisfies
\begin{equation*}
L^S(s_{0},\Pi\otimes \Pi') \sim_{E(\Pi)E(\Pi')} (2\pi i)^{nn's_{0}} \prod\limits_{\imath \in \Sigma}[\prod\limits_{0\leq i\leq n}P^{(i)}(\Pi,\imath)^{sp(i,\Pi;\Pi',\imath)}\prod\limits_{0\leq j\leq n'}P^{(j)}(\Pi',\imath)^{sp(j,\Pi';\Pi,\imath)}].
\end{equation*}
Here $\imath$ runs over complex embeddings of $F$ belonging to a fixed CM type $\Sigma$, $sp(i,\Pi;\Pi',\imath)$ are integers depending on the relative positions of the infinitesimal characters of $\Pi$ and $\Pi'$ at the place $\imath$, and  $P^{(i)}(\Pi,\imath)$ and  $P^{(j)}(\Pi',\imath)$ are period invariants attached to $\Pi$ and $\Pi'$ by quadratic base change from certain unitary groups (that depend on $\imath$ and the superscripts $(i), (j)$), and the symbol ``$\sim_{E(\Pi)E(\Pi')}$'' means that the left-hand side is the product of the right hand side by an element of a certain number field attached to $\Pi$ and $\Pi'$.
\end{thm}

In other words, the critical values of the Rankin-Selberg $L$-function $L(s,\Pi \times \Pi')$ can be expressed in terms of Petersson norms of certain
arithmetically normalized holomorphic automorphic forms.  We refer to the body of the paper for details and explanations.  Hyp. \ref{descent}, a mild local restriction at non-archimedean places, is only relevant when $nn'$ is even, and then amounts to the familiar fact that it is not always possible to construct even-dimensional hermitian spaces with arbitrary local invariants; it can be relaxed by a standard base change construction at the cost of introducing additional
quadratic irrationalities.\\

Deligne's conjecture, as stated in \cite{deligne}, asserts that the left-hand side of the equation in Thm.\ref{automorphic Deligne general intro}  is proportional (up to the coefficient field $E(M(\Pi))E(M(\Pi'))$) to the period invariant Deligne assigned to a motive $R_{F/\Q}(M(\Pi)\otimes M(\Pi'))$ whose $L$-function is given by $L(s,\Pi\otimes \Pi')$.   The second theme of this paper concerns the relation of the right-hand side of the equation in
 Thm. \ref{automorphic Deligne general intro} to Deligne's period invariant, denoted $c^+(s_0, R_{F/\Q}(M(\Pi)\otimes M(\Pi')))$.  Under the hypotheses of Thm.\ref{automorphic Deligne general intro}, motives $M(\Pi)$ and $M(\Pi')$ over $F$, of rank $n$ and $n'$ over their
 respective coefficient fields, can be constructed in the cohomology of Shimura varieties $Sh(V)$ and $Sh(V')$ attached to the unitary groups of hermitian spaces $V$ and $V'$ of rank $n$ and $n'$ respectively.   These Shimura varieties have the property that their connected components are arithmetic quotients
 of the unit ball in $\C^{n-1}$ and $\C^{n'-1}$, respectively.   Our second main theorem can be paraphrased as follows:
 \begin{thm}\label{cplusintro}   Let $F$, $\Pi$, and $\Pi'$ be as in Thm.\ref{automorphic Deligne general intro}.
Assume Conj.\ \ref{nonvan} (non-vanishing of certain central critical values) and Conj.\ \ref{lvarch} (rationality of certain archimedean integrals).  Then for any critical value $s_0$ of $L(s,\Pi\times \Pi')$, the Deligne period $c^+(s_0,R_{F/\Q}(M(\Pi)\otimes M(\Pi')))$ can be identified with the right-hand side of 
the equation in Thm.\ref{automorphic Deligne general intro}.
 \end{thm}
 
The content of Thm.\ref{cplusintro} is a relation between periods of automorphic forms on Shimura varieties attached to hermitian spaces with different
signatures.     Following an approach pioneered by Shimura over 40 years ago, we combine special cases of Deligne's conjecture with comparisons of distinct expressions for critical values of automorphic $L$-functions to relate automorphic periods on different groups.  These periods are attached to motives (for absolute Hodge cycles) that occur in the cohomology of the various Shimura varieties.  In view of Tate's conjecture on cycle classes in $\ell$-adic cohomology, the relations obtained are consistent with the determination of the representations of Galois groups of appropriate number fields on the $\ell$-adic  cohomology of the respective motives.  The paper \cite{linfactorization} used  arguments of this type to show how to factor automorphic periods on Shimura varieties attached to a CM field $F$ as products of automorphic periods of holomorphic modular forms, each attached to an  embedding $\imath = \imath_{v_0}:  F \hookrightarrow \C$.  Thm. \ref{auto-facto} (see \eqref{fac} below) leads to a factorization of the latter periods in terms of periods of coherent cohomology classes on Shimura varieties attached to the unitary group $H^{(0)}$ of a hermitian space over $F$ with signature $(n-1,1)$ at $\imath_{v_0}$ and definite at embeddings that are distinct from $\imath_{v_0}$ and its complex conjugate.  This factorization -- see Thm.\ \ref{main factorization} (and the explanations in \S\ref{sect:goals}), which is the precise statement of which Thm.\ \ref{cplusintro} is a paraphrase --  also depends on the conjectures and local restrictions mentioned above.  The archimedean components of the automorphic representations we consider are tempered and are {\it cohomological}, in the sense explained in \eqref{rlag} below.  The local restrictions at archimedean places take the form of regularity hypotheses on the infinitesimal characters of these finite-dimensional representations, or equivalently on the Hodge structures of the associated motives. \\

Taken together, our two main theorems, Thm.\ \ref{automorphic Deligne general} and Thm.\ \ref{main factorization} -- always assuming the two conjectures and local restrictions already mentioned -- provide a plausible version of Deligne's conjecture for the $L$-function of the tensor products of the motives $M(\Pi)$ and $M(\Pi')$ over $\imath(F)$ attached to $\Pi$ and $\Pi'$, respectively, with the formula for Deligne's period $c^+(s_0,R_{F/\QQ}(M(\Pi)\otimes M(\Pi')))$, computed as in Prop.\ \ref{Deligne period motivic}. We refer to our Thm.\ \ref{thm:clozel}, where the motives $M(\Pi)$ and $M(\Pi')$ are in fact constructed, verifying a conjecture of Clozel for those cuspidal automorphic representations $\Pi$ and $\Pi'$, respectively. \\
    
The main results are based on two kinds of expressions for automorphic $L$-functions.  The first derives from the Rankin-Selberg method for $\GL(n)\times\GL(n-1)$.  The paper \cite{grob-harr} applied this method over $F$, when $F$ is imaginary quadratic, to prove some cases of Deligne's conjecture in the form derived in \cite{harrisANT}.   A principal innovation of that paper was to take the automorphic representation on $\GL(n-1)$ to define an Eisenstein cohomology class.  This has subsequently been extended to general CM fields in \cite{jie-thesis}, \cite{grob-app} and  \cite{grob_lin}.  The basic structure of the argument is the same in all cases.  Let $G_r$ denote the algebraic group $\GL(r)$ for any $r \geq 1$, over the base CM-field $F$; let $K_{G_r,\infty}$ denote a maximal connected subgroup of $G_{r,\infty}:=\GL_r(F \otimes_\Q \R)$ which is compact modulo the center, and and let $\g_{r,\infty}$ denote the Lie algebra of $G_{r,\infty}$. Let $\Pi$ be a cuspidal automorphic representation of $G_n(\A_F)$ and let $\Pi'$ be an isobaric automorphic representation of $G_{n-1}(\A_F)$, 
$$
\label{sum}  \Pi' = \Pi'_1 \boxplus \Pi'_2 \boxplus \dots \boxplus \Pi'_r
$$
where each $\Pi'_i$ is a cuspidal automorphic representation of $\GL_{n_i}(\A_F)$ and $\sum_i n_i = n-1$. We assume that both $\Pi$ and $\Pi'$ are cohomological, in the sense that there exist finite-dimensional, algebraic representations $\EE$ and $\EE'$ of $G_{n,\infty}$ and $G_{n-1,\infty}$, respectively, such that the relative Lie algebra cohomology spaces
\begin{equation}\label{rlag}
H^*(\g_{n,\infty},K_{G_n,\infty},\Pi_\infty\otimes \EE) \neq 0,\quad\quad  H^*(\g_{n-1,\infty},K_{G_{n-1},\infty},\Pi'_\infty\otimes \EE') \neq 0.
\end{equation}
Here $\Pi_\infty$ and $\Pi'_\infty$ denote the archimedean components of $\Pi$ and $\Pi'$, respectively.  Although this is not strictly necessary at this stage 
%(see \cite{?} [Raghuram?]) ({\bf COMM by Jie: we can cite also \cite{grob-app} and \cite{grob_lin}}), 
we also assume 
\begin{hyp}\label{duality}  The cuspidal automorphic representations $\Pi'_i$ and $\Pi$ are all conjugate self-dual.  
\end{hyp}
Then it is known that $\Pi$ and all summands $\Pi'_i$ are tempered locally everywhere \cite{HT,Shin, car} because each $\Pi'_i$ and $\Pi$ is (up to a twist by a half-integral power of the norm, which we ignore for the purposes of this introduction) a cuspidal {\it cohomological} representation. 
Suppose there is a non-trivial $G_{n-1}(F\otimes_\Q\C)$-invariant pairing
\begin{equation}\label{piano} \EE \otimes \EE' \ra \C. \end{equation}
Then the central critical value of the Rankin-Selberg $L$-function $L(s,\Pi \times \Pi')$ can be expressed as a  cup product in the cohomology, with twisted coefficients, of the locally symmetric space attached to $G_{n-1}$. These cohomology spaces have natural rational structures over number fields, and the cup product preserves rationality.  From this observation we obtain the following relation for the {\it central} critical value $s = s_0$:   
\begin{equation}\label{whitt}  
L^S(s_0,\Pi \times \Pi') \sim p(s_{0},\Pi_{\infty} \ \Pi'_{\infty}) \ p(\Pi) \ p(\Pi')  
\end{equation}
where $p(\Pi)$ and $p(\Pi')$ are the {\it Whittaker periods} of $\Pi$ and $\Pi'$, respectively, and $p(s_{0},\Pi_{\infty},\Pi'_{\infty})$ is an archimedean factor depending only on $s_{0}$, $\Pi_{\infty}$ and $\Pi'_{\infty}$.  
The notation  $\sim$, here and below, means ``equal up to specified algebraic factors''; we will have more to say about this in  section \ref{notationsim} at the end of this introduction.  We point out that this archimedean period $p(s_{0},\Pi_{\infty},\Pi'_{\infty})$ can in fact be computed, up to algebraic factors, as a precise integral power of $2\pi i$, see \cite{grob_lin} Cor.\ 4.30. It turns out that this power is precisely the one predicted by Deligne's conjecture. Furthermore, Thm.\  2.5 of \cite{grob_lin} provides the expression
\begin{equation}\label{piprime} 
p(\Pi') \sim \prod_{i = 1}^r p({\Pi'_i}) \cdot \prod_{i < j}L^S(1,\Pi'_i \times (\Pi'_j)^{{\sf v}}).
\end{equation} 
 Finally, the cuspidal factors $p(\Pi)$ and $p(\Pi'_i)$ can be related to the critical $L$-values of the Asai $L$-functions 
 $L^S(1,\Pi,{\rm As}^{(-1)^{n}})$, $L^S(1,\Pi'_i,{\rm As}^{(-1)^{n_i}})$, as in \cite{grob-harr}, \cite{grob_harris_lapid}, \cite{jie-thesis}, and most generally in \cite{grob_lin}.
 
Under Hyp.\ \ref{duality}, it is shown in \cite{grob-harr} and \cite{jie-thesis} that all the terms in \eqref{piprime} can be expressed in terms of automorphic periods of arithmetically normalized holomorphic modular forms on Shimura varieties attached to unitary groups of various signatures.  
So far we have only considered the central critical value $s_0$, but variants of \eqref{piano}  allow us to treat other critical values of $L(s,\Pi \times \Pi')$ in the same way.  
Non-central critical values, when they exist, do not vanish, and in this way the expressions \eqref{whitt} and \eqref{piprime} give rise to non-trivial relations among these automorphic periods, including the factorizations proved in \cite{linfactorization}.    \\

Of course
$$L(s,\Pi \times \Pi') = \prod_{i = 1}^r L(s,\Pi \times \Pi'_i).$$
When $n_i = 1$, the critical values of $L(s,\Pi \times \Pi'_i)$ were studied in \cite{H97} and subsequent papers, especially \cite{guer, guer-lin}.  Thus, provided $n_1 = m \leq n-1$, it is possible to analyze the critical values of $L(s,\Pi \times \Pi'_1)$ using \eqref{whitt} and \eqref{piprime}, provided $\Pi'_1$ can be completed to an isobaric sum as above, with $n_i = 1$ for $i = 2, \dots, r$, such that \eqref{piano} is satisfied.   This argument is carried out in detail in \cite{lin15} and \cite{jie-thesis}. See also \cite{gro-sach}.\\

In the above discussion, we  need to assume that abelian twists $L(s,\Pi \times \Pi'_i)$, $i > 1$, have non-vanishing critical values; this can be arranged 
automatically under appropriate regularity hypotheses but requires a serious {\it non-vanishing hypothesis} in general -- we return to this point later.  For 
the moment, we still have to address the restriction on the method imposed by the requirement \eqref{piano}.  For this, it is convenient to divide critical 
values of $L(s,\Pi \times \Pi'_1)$, with $n_1 = m$ as above, into two cases.  We say the weight of the $L$-function $L(s,\Pi \times \Pi'_1)$ is {\it odd} 
(resp. {\it even}) if the integers $n$ and $m$ have  opposite (resp. equal) parity.   
In the case of even parity, the left-most critical value -- corresponding to $s = 1$ in the unitary normalization of the $L$-function -- was treated completely in Lin's thesis \cite{jie-thesis}.  
Here we treat the remaining critical values in the even parity case by applying a method introduced long ago by Harder, and extended recently by Harder and Raghuram \cite{harder-ragh} for totally real fields and Raghuram in \cite{ragh19} for totally imaginary fields, to compare successive critical values of a Rankin-Selberg $L$-function for $\GL(n) \times \GL(m)$. Under different assumptions, even more refined results for successive critical values have been established in the odd parity case in \cite{lin15}, \cite{jie-thesis}, \cite{grob_lin} and \cite{gro-sach}, extending \cite{harder-ragh} to CM-fields. This reduces the analysis of critical values in the odd case to the central critical value -- provided the latter does not vanish, which we  now assume.\\

In order to treat the central critical value when we cannot directly complete $\Pi'_1$ to satisfy \eqref{piano}, we need a second expression for automorphic $L$-functions:   the Ichino-Ikeda-N. Harris formula (henceforward:  the IINH formula) for central values of automorphic $L$-functions of $U(N)\times U(N-1)$, stated below as Thm.\ \ref{conjecture II}.  
Here the novelty is that we complete both $\Pi$ on $G_n(\A_F)$ and $\Pi'_1$ on $G_m(\A_F)$ to isobaric cohomological representations $\tilde{\Pi}$ and $\tilde{\Pi}'$ of $G_N(\A_F)$ and $G_{N-1}(\A_F)$, respectively, for sufficiently large $N$, adding $1$-dimensional representations $\chi_i$ and $\chi'_j$ in each case, so that the pair 
$(\tilde{\Pi}, \tilde{\Pi}')$  satisfies \eqref{piano}.   At present we have no way of interpreting the critical values of $L(s,\tilde{\Pi} \times \tilde{\Pi}')$ as cohomological cup products, for the simple reason that both $\tilde{\Pi}$ and $\tilde{\Pi}'$ are Eisenstein representations and the integral of a product of Eisenstein series is divergent.  However, we can replace the Rankin-Selberg integral by the IINH formula, provided we assume
\begin{hyp}\label{non-van-global}  For all $i, j$ we have
$$L(s_0,\Pi \times \chi'_j) \neq 0; ~ L(s_0,\chi_i \times \Pi'_1) \neq 0;  ~~ L(s_0, \chi_i \cdot \chi'_j) \neq 0.$$
Here $s_0$ denotes the central value in each case ($s_0 = \tfrac{1}{2}$ in the unitary normalization).  
\end{hyp}
We have already assumed that the central value of interest, namely $L(s_0,\Pi \times \Pi'_1)$, does not equal zero.   
Assuming Hyp.\ \ref{non-van-global}, an argument developed in \cite{harrisANT,grob-harr}, based on the IINH formula, allows us to express the latter central value in terms of automorphic periods of arithmetically normalized holomorphic automorphic forms on unitary groups.  
In order to relate the values in Hyp.\ \ref{non-van-global} to the IINH formula, which is a relation between periods and central values of $L$-functions of pairs of unitary groups, we apply Hyp.\ \ref{duality} and the theory of stable base change for unitary groups, as developed in sufficient generality in \cite{KMSW,shin}, to identify the $L$-functions in the IINH formula with automorphic $L$-functions on general linear groups.\\

This argument, which is  carried out completely in \S\ref{central value}, is one of the keys to the factorization of periods of a single arithmetically normalized holomorphic automorphic form $\omega_{(r_{\imath},s_{\imath})}(\Pi)$ on the Shimura variety attached to the unitary group $H$ of an $n$-dimensional hermitian space over $F$ with signature $(r_{\imath},s_{\imath})$ at the place 
$\imath = \imath_{v_0}$ mentioned above, and definite at embeddings that are distinct from $\imath$ and its complex conjugate.    
The notation indicates that $\omega_{(r_{\imath},s_{\imath})}(\Pi)$ belongs to an automorphic representation of $H$ whose base change to $G_n$ is our original cuspidal automorphic representation $\Pi$.  
The period in question, denoted $P^{(s_{\imath})}(\Pi,\imath)$, is essentially the normalized Petersson inner product of $\omega_{(r_{\imath},s_{\imath})}(\Pi)$ with itself. It was already explained in \cite{H97} that Tate's conjecture implies a relation of the following form:
\begin{equation}\label{fac}
P^{(s_{\imath})}(\Pi,\imath) \sim  \prod\limits_{0\leq i\leq s_{\imath}} P_{i}(\Pi,\imath)
\end{equation}
where $P_{i}(\Pi,\imath)$ is a normalized version of the Petersson norm of a form on the Shimura variety attached to the specific unitary group $H^{(0)}$. Our main result on factorization (Thm.\ \ref{main factorization}) is a version of \eqref{fac} and allows us to relate the local arithmetic automorphic periods $P^{(s_{\imath})}(\Pi,\imath)$ to the motivic periods $Q^{(s_{\imath})}(M(\Pi),\imath)$ appearing in the factorization of Deligne's periods.  Like most of the other theorems already mentioned, this one is conditional on two conjectures and local conditions, to which we now turn. 

\subsection*{Conjectures assumed in the proofs of the main theorems}  

Thm.\ \ref{auto-facto}, which is the basis for Thm.\ \ref{main factorization}, is conditional on the following two conjectures:

\begin{itemize}

%\item[(a)]  Conj.\ \ref{conjecture II} (the IINH-Conj.\) for pairs of cohomological representations of totally definite unitary groups and of unitary groups, which are indefinite at precisely one place (and there of real rank $1$), cf.\ Ass.\ \ref{signature}. In fact, in the latter case, we only need Conj.\ \ref{conjecture II} for pairs of cohomological representations, each of which admits a base change to a {\it cuspidal} automorphic representation of the corresponding general linear group over $F$.  Under these hypotheses Conj.\ \ref{conjecture II} seems now to be completely proved.  In contrast, the results in Section \ref{central value} require knowledge of Conj.\ \ref{conjecture II} for pairs of endoscopic representations, which is still some years away.  

\item[(a)] Conj.\ \ref{nonvan} (non-vanishing of certain twisted central critical values).
\item[(b)] Conj.\ \ref{lvarch} (rationality of certain archimedean integrals).
\end{itemize}

Conj.\ \ref{lvarch} can only be settled by a computation of the integrals in question.  The conjecture is natural because its failure would contradict  the Tate conjecture; it is also known to be true in the few cases where it can be checked.  Methods are known for computing these integrals but they are not simple.   In the absence of this conjecture, the methods of this paper provide a weaker statement:  the period relation in Thm.\ \ref{auto-facto} is true up to a product of factors  that depend only on the archimedean component of $\Pi$. Such a statement had already been proved in \cite{H07} using the theta correspondence, but the proof is much more complicated.\\

Everyone seems to believe Conj.\ \ref{nonvan}, but it is clearly very difficult.  In fact, the proof of general non-vanishing theorems for character twists of $L$-functions of $\GL(n)$, with $n > 2$, seemed completely out of reach until recently.  In the last few years, however, there has been significant progress in the cases $n = 3$ and $n = 4$, by two very different methods \cite{jz,blm}, and one can hope that there will be more progress in the future.\\

%In addition to the conjectures above, we repeatedly refer to the manuscript \cite{KMSW} of Kaletha, M\'inguez, Shin, and White, which  contains a complete classification of the discrete spectrum of unitary groups over number fields in terms of the automorphic spectra of general linear groups.  In particular, it contains statements of all the results we need regarding stable base change from unitary groups to general linear groups, and descent in the other direction.  Some of the proofs given in \cite{KMSW} are not complete, however; the authors are in the process of preparing two additional instalments, which will complete the missing steps.  Some readers may therefore prefer to view Thm. \ref{auto-facto} as conditional on the completion of the project of \cite{KMSW}.\\

Although the main theorems are conditional on these conjectures, %as well as on results that have been announced but whose complete proofs have not yet appeared, 
we still believe that the methods of this paper are of interest:  they establish clear relations between important directions in current research on automorphic forms and a  version of Deligne's conjecture in the most important cases accessible by automorphic methods.  Moreover, the most serious condition is the non-vanishing Conj. \ref{nonvan} above.  The proofs, however, remain valid whenever the non-vanishing can be verified for a given automorphic representation $\Pi$ and all the automorphic representations $\Pi'$ that intervene in the successive induction steps as in \S\ref{sect:proofThm61}.

\subsection*{About the proofs}  The first main theorem relates special values of $L$-functions to automorphic periods, and relies on the methods described above:  the analysis of Rankin-Selberg $L$-functions using cohomological cup products, in particular Eisenstein cohomology, and the results of \cite{grob-harr, jie-thesis, harder-ragh, grob_lin} on the one hand, and the IINH conjecture on the other.  
The second main theorem obtains the factorization of periods \eqref{fac} by applying the IINH conjecture to the results on special values, and by using a result on non-vanishing of cup products of coherent cohomology proved in \cite{H14}.  In fact, the case used here had already been treated in \cite{HLi}, assuming properties of stable base change from unitary groups to general linear groups that were recently proved in \cite{KMSW}: Some of the results of  \cite{KMSW} are still conditional, but what we need for our purposes can be found therein in sufficient generality in unconditional form. \\

The results of \cite{H14} are applied by induction on $n$, and each stage of the induction imposes an additional regularity condition.
This explains the regularity hypothesis in the statement of Thm.\ \ref{auto-facto}.  The factorization in the theorem must be true in general, but it is not clear to us whether the method based in the IINH conjecture can be adapted in the absence of the regularity hypothesis.

\subsection*{On using the IINH conjecture to solve for unknowns}
Although we have no sympathy with the general outlook of the politician Donald Rumsfeld, and we consider his role in recent history to be largely deleterious, in the formulation of the strategy for proving our main results we did find it helpful to meditate on his thoughts on knowledge, as expressed in the following quotation \cite{EM}:
\\
\\
{\it...as we know, there are known knowns; there are things we know we know. We also know there are known unknowns; that is to say we know there are some things we do not know. But there are also unknown unknowns -- the ones we don't know we don't know.   }
\\
\\
Rumsfeld neglected the {\bf unknown knowns}, such as the period invariants and critical values that are the main subject of this paper.
The formula of Ichino--Ikeda--Neal-Harris, in the inhomogenous form in which it is presented in Thm. \ref{conjecture II}, can be viewed as an identity involving three kinds of transcendental quantities:  critical values of Rankin-Selberg and Asai $L$-functions, Petersson norms of algebraically normalized coherent cohomology classes, and cup products between two such classes.   Here is a simplified version of the conjecture, which is now a theorem, with elementary terms indicated by $(*)$:
\begin{equation}\label{IIsimple}
\cfrac{|I^{can}(f,f')|^{2}}{\<f,f\>\ \<f',f'\>} = (*)\frac{L(\tfrac{1}{2},\Pi\otimes \Pi')}{L(1,\Pi,{\rm As}^{(-1)^n})L(1,\Pi',{\rm As}^{(-1)^{n-1}})}.
\end{equation}
From the Rumsfeld perspective, the denominator of the right-hand side of \eqref{IIsimple}, which is independent of the relative position of $\Pi$ and $\Pi'$,
was an {\bf unknown known} that became a {\bf known known} thanks to \cite{grob-harr} and subsequent generalizations.   The same paper,
as well as \cite{harrisANT}, turn the numerator of the right-hand side into a {\bf known known}, as long as the coefficients of the 
cohomology classes defined by $\Pi$ and $\Pi'$ satisfy the relation \eqref{piano}.  
Thm. \ref{automorphic Deligne general intro} leverages the result under \eqref{piano} and Conj. \ref{nonvan} to turn the numerator of the right-hand side into a {\bf known known} even when $\Pi$ and $\Pi'$ themselves do not satisfy \eqref{piano}.  

Thus the entire right-hand side of the formula can be considered a {\bf known known}.  As for the left-hand side, the periods in the denominator should at best be
viewed as {\bf known unknowns}, and then only when $f$ and $f'$ are holomorphic automorphic forms -- because the only thing we know about Petersson norms of (algebraically normalized) holomorphic automorphic forms is that they are uniquely determined real numbers that are probably transcendental.  That leaves the numerator of the left-hand side, and here we use the result of \cite{H14}, when it applies, to choose $f$ and $f'$ so that the numerator, as a cup product in coherent cohomology, belongs to a fixed algebraic number field.  In fact, the numerator can be taken to be $1$, which is a {\bf  known known}, if anything is.  

Finally, as the unitary groups vary most of the periods that appear in the numerator of the left-hand side of \eqref{IIsimple} have no cohomological interpretation. Thus these have to be viewed as {\bf unknown unknowns} in Rumsfeld's sense -- precisely because the identity \eqref{IIsimple} relates these periods to {\bf known knowns} and {\bf unknown knowns} (the latter when the periods in the denominator are attached to higher coherent cohomology classes on Shimura varieties for unitary groups with mixed signature, which we have not studied).
%\end{rmk}

\begin{sub}{{\bf Local hypotheses.}}\label{localhyp}

The conclusion of Thm.\ \ref{auto-facto} is asserted for $\Pi$ that satisfy a list of conditions. One of them is our hypothesis that
\medskip

{\it $\Pi$ is a $(n-1)$-regular, cohomological conjugate self-dual cuspidal automorphic representation of $G_n(\A_F)$ such that for each possible $[F^+:\Q]$-tuple of signatures $I$, $\Pi$ descends to a tempered cohomological cuspidal representation of a unitary group $U_I(\A_{F^+})$ of signature $I$ at the archimedean places.}
\medskip

The conditions that $\Pi$ be conjugate self-dual and cohomological are necessary in order to make sense of the periods that appear in the statement of the theorem.  Assuming these two conditions, the theorem for cuspidal $\Pi$ implies the analogous statement for more general $\Pi$.  The other two conditions boil down to local hypotheses: The $(n-1)$-regularity condition is a hypothesis on the archimedean component of $\Pi$ that is used in some of the results used in the proof, notably in the repeated use of the results of \cite{jie-thesis}. 
The final condition about descent is automatic if $n$ is odd but requires only a local hypothesis at some non-archimedean place if $n$ is even.  
Using quadratic base change, as in work of Yoshida and others, one can probably obtain a weaker version of Thm.\ \ref{auto-facto} in the absence of this assumption, but we have not checked the details.
\end{sub}

\begin{sub}{{\bf About Galois equivariance and the notation ``$\sim$''.}}\label{notationsim}
Deligne's conjecture concerns two quantities $\alpha, \beta$ that are naturally elements of the algebra $E \otimes_\Q \CC$, where $E$ is a number field
(the {\it coefficient field}) and $\beta$ is invertible.  The assertion $\alpha \sim_E \beta$ means that there exists $\gamma \in E$, considered as an element of $E \otimes_\Q \CC$ through its  embedding in
the first factor, such that 
$$\alpha = \gamma\cdot \beta.$$
Suppose $L \subset E$ is naturally a subfield of $\CC$.  We write
$$\alpha \sim_{E\otimes_{\Q}L} \beta$$
to mean the weaker condition that there exists $\gamma \in E \otimes_\Q L \subset E \otimes_\Q \CC$ such that
$$\alpha = \gamma\cdot\beta$$ (in \cite[p. 82]{H97}, this is written as $\alpha \sim_{E;L} \beta$).
We consider the number field $E$ as a subfield of $\Qbar\subset\C$. We can naturally extend an element $\alpha\in E \otimes_\Q \CC\cong \prod\limits_{E\hookrightarrow \C}\C$ to a family $\underline{\alpha}=\{\alpha(\sigma)\}_{\sigma\in {\rm Gal}(\Qbar/\QQ)}$, putting $\alpha(\sigma)=\alpha_{\sigma|_{E}}$. If we look at relations of $ {\rm Gal}(\Qbar/\QQ)$-families, the relation ``$\sim_{E}$'' means equivariancy under the full Galois group ${\rm Gal}(\Qbar/\QQ)$, and the weaker relation ``$\sim_{E\otimes_{\Q}L}$'' means equivariancy only under ${\rm Gal}(\Qbar/L)$ (see Def.\ \ref{definition algebraic relation} and Rem.\ \ref{rem:Ealg} for details).\\\\
We will be working over a CM field $F$, and our coefficient field $E$ will always contain the Galois closure $F^{Gal}$ of $F$ in $\Qbar$, which
is canonically a subfield of $\CC$.   We had hoped to be able to state our main results on Deligne's conjecture and factorization of periods
using the notation $\sim_E$, but some of the intermediate results on which our theorems are based on the main theorems of \cite{guer-lin},
are stated in the weaker form $\sim_{E\otimes_{\Q}F^{Gal}}$. \\\\  
It is useful to review the rationality properties on which our main results are based, in order to explain why at the present stage we need to settle for the weaker relation $\sim_{E\otimes_{\Q}F^{Gal}}$.  There are two kinds of properties:  those derived from the topological cohomology of the locally symmetric spaces for $\GL(n)$, which are completely Galois-equivariant, and those based on the coherent cohomology of Shimura varieties, where the Galois equivariance depends on the formula for conjugation of Shimura varieties that was conjectured by Langlands \cite{langlands} and proved in general by Borovoi and Milne.  In greater detail, four classes of results are invoked in the course
of the rationality arguments:

\begin{itemize}
\item[(i)]   The relation between critical values of Rankin-Selberg $L$-functions for $\GL(n)\times \GL(n-1)$ and cup products in topological cohomology, as in
\cite{grob-harr,grob-app}, provided by the Jacquet-Piatetski-Shapiro-Shalika integral.  The action of ${\rm Gal}(\Qbar/\QQ)$ is on the coefficients and the results
are therefore completely Galois equivariant.  Moreover, the automorphic representations have models over their fields of rationality (see \ref{sect:EPi}) so there is no need to account for a Brauer obstruction.  However, the method of the papers cited leaves an archimedean Euler factor undetermined.  This Euler factor is
then identified, up to algebraic factors, in \cite{lin15} and \cite{grob_lin}, using a comparison with the method of (iii) below, which is not completely Galois equivariant unless we obtain a complete Galois equivariant version of \cite{guer-lin}.
\item[(ii)] 	The Ichino-Ikeda-N. Harris (IINH) conjecture for definite unitary groups, as in \cite{harrisANT, grob_lin}; see Proposition \ref{lem:piano} below.  This is again purely topological and ${\rm Gal}(\Qbar/\QQ)$-equivariant.  
\item[(iii)]   The doubling method to relate critical values of standard $L$-functions of unitary groups to cup products
of holomorphic and anti-holomorphic classes in coherent cohomology of Shimura varieties,  as in \cite{H97,guer-lin,H21}.  We obtain relations of rationality over the reflex field of the Shimura variety, and since we work with all $U(V)$ with $V$ of dimension $n$ over $F$,  in the applications we only obtain relations over the composite of these reflex fields, which is the source of the relation $\sim_{E\otimes_{\Q}F^{Gal}}$.  (In addition, the automorphic representations of unitary groups do not generally have models over their fields of rationality, which introduces an additional complication.)  
\item[(iv)] The IINH conjecture when the unitary groups are definite at all but one place, which allows us to relate central values of certain $L$-functions to cup products in coherent cohomology in arbitrary degree, and to make use of the result of \cite{H14}.  Again, relations are only obtained over the reflex fields of Shimura varieties.
\end{itemize}

Point (i) involves a number of separate steps, most of which make use of results on Eisenstein cohomology \cite{linfactorization, harder-ragh, gro-sach}, and in particular on Shahidi's formula for Whittaker coefficients of generic Eisenstein series.  Each of these steps is also ${\rm Gal}(\Qbar/\QQ)$-equivariant.  

To replace the relations $\sim_{E\otimes_{\Q}F^{Gal}}$ in points (iii) and (iv) with the more precise relation $\sim_{E}$, one needs to appeal to the results on ${\rm Gal}(\Qbar/\QQ)$-conjugation of Shimura varieties.  The period invariants introduced in \cite{H13} behave well with respect to Galois conjugation; this should make it possible to prove refined versions of our
main results, in the form predicted by Deligne.  It seems that the principal difficulty remaining is to determine the behavior of the archimedean $L$-factors in step (iii) under conjugation of Shimura varieties.   Even if this is resolved, however, the method of \cite{H13} is based on purely formal considerations
about conjugation of Shimura varieties, and the periods introduced there probably conceal some deeper arithmetic information.

\end{sub}

\subsubsection*{Acknowledgements}\small
We thank Sug Woo Shin for several very useful conversations.  We also thank Dipendra Prasad for help with the references for the proof of Prop.\ \ref{mult1}.\\
HG thanks also the late Ferdinand Johannes G\"odde $\dagger$ (called ``Jan Loh'') for memorable discussions (``zero times zero'') in Bonn.
\normalsize

\numberwithin{equation}{section}
\section{Preliminaries} 

\subsection{Number fields and associate characters} \label{sect:fields}

We let $\Qbar$ be the algebraic closure of $\Q$ in $\C$. All number fields are considered as subfields of $\Qbar$.
For $k$ a number field, we let $J_k$ be its set of complex field-embeddings $\imath: k\hra\C$. We will write $S_\infty(k)$ for its set of archimedean places, $\O_k$ for its ring of integers, $\A_k$ for its ring of adeles, and use $k^{Gal}$ for a fixed choice of a Galois closure of $k/\Q$ in $\overline\Q\subset\C$. If $\pi$ is an abstract representation of a non-archimedean group, we will write $\Q(\pi)$ for the field of rationality of $\pi$, as defined in \cite{waldsp}, I.1. In this paper, every rationality field will turn out to be a number field.\\\\
Throughout our paper, $\cm$ will be reserved in order to denote a CM-field of dimension $2d = \dim_\Q \cm$. The set of archimedean places of $F$ is abbreviated $S_\infty=S(\cm)_\infty$. We will chose a section $S_\infty \ra J_F$ and may hence identify a place $v\in S_\infty$ with an ordered pair of conjugate complex embeddings $(\imath_v,\bar\imath_v)$ of $\cm$, where we will drop the subscript ``$v$'' if it is clear from the context. This order in turn fixes a choice of a CM-type $\Sigma:=\{\imath_v : v\in S_\infty\}$. The maximal totally real subfield of $\cm$ is denoted $\tr$. Its set of archimedean places will be identified with $S_\infty$, identifying a place $v$ with its first component embedding $\imath_v\in\Sigma$ and we let Gal$(\cm/\tr)=\{1,c\}$. \\\\
We extend the quadratic Hecke character $\varepsilon=\varepsilon_{\cm/\tr}: (\tr)^\times\backslash \A^\times_{\tr}\ra\C^\times$, associated to $\cm/\tr$ via class field theory, to a conjugate self-dual Hecke character $\eta: \cm^\times\backslash\A_{\cm}^\times\ra\C^\times$. At $v\in S_\infty$, $z\in F_v\cong\C$, we have $\eta_v(z)=z^t \bar z^{-t}$, where $t\in\tfrac12+\Z$. For the scope of this paper, we may assume without loss of generality that $t=\tfrac12$, \cite[\S 6.9.2]{bel-chen}. We define $\psi:=\eta\|\cdot\|^{1/2}$, which is an algebraic Hecke character.\\\\
If $\chi$ is a Hecke character of $\cm$, we denote by $\widecheck{\chi}$ its conjugate inverse $(\chi^{c})^{-1}$.

\subsection{Algebraic groups and real Lie groups} \label{sect:alggrp}
We abbreviate $G_n:=\GL_n/{\cm}$. Let $(V_n,\<\cdot,\cdot\>)$ be an $n$-dimensional non-degenerate $c$-hermitian space over $\cm$, $n\geq 1$, we denote the corresponding unitary group over $F^+$ by $H:=H_n:=U(V_n)$. For each $v \in S_\infty$ we let $(r_v,s_v)$ denote the signature of the hermitian form induced by $\<\cdot,\cdot\>$ on the complex vector space $V_v := V\otimes_{F,\imath_v} \CC.$    \\\\ 
Whenever one has fixed an embedding $V_k\subseteq V_n$, we may view the attached unitary group $U(V_k)$ as a natural $\tr$-subgroup of $U(V_n)$. If $n=1$, the algebraic group $U(V_1)$ is isomorphic to the kernel of the norm map $N_{F/F^+}: R_{F/F^+}((\Gm)_F) \to (\Gm)_{F^+}$, where $R_{F/F^+}$ stands for the Weil-restriction of scalars from $F/F^+$, and is thus independent of $V_1$.\\\\ 
Let $\sigma\in{\rm Aut}(\C)$ and let $V_n$ be as above. Then there is a unique $c$-Hermitian space $^\sigma V_n$ over $F$, whose local invariants at the non-archimedean places of $F$ are the same as of $V_n$ and whose signatures satisfy $({}^\sigma r_v, {}^\sigma s_v)=(r_{\sigma^{-1}\circ \imath_v},s_{\sigma^{-1}\circ \imath_v})$ at all $v\in S_\infty$, cf.\ \cite{Lan}. We let ${}^\sigma H:= U({}^\sigma V_n)$ be the attached unitary group over $F^+$. By definition, ${}^\sigma H(\A_{f})\cong H(\A_{f})$ and ${}^\sigma H_\infty\cong \prod_{v\in S_\infty} H(F_{\sigma^{-1}\circ v})$.
\\\\
If $G$ is any reductive algebraic group over a number field $k$, we write $Z_G/k$ for its center, $G_\infty:=R_{k/\Q}(G)(\R)$ for the real Lie group of $\R$-points of the Weil-restriction of scalars from $k/\Q$ and denote by $K_{G,\infty}\subseteq G_\infty$ the product of $(Z_G)_\infty$ and a fixed choice of a maximal compact subgroup of $G_\infty$. Hence, we have $K_{G_{n,\infty}}\cong \prod_{v\in S_\infty}K_{G_n,v}$, each factor being isomorphic to $K_{G_n,v}\cong \R_+ U(n)$; $K_{H,\infty}\cong \prod_{v\in S_\infty}K_{H,v}$, with $K_{H,v}\cong U(r_v)\times U(s_v)$; and $K_{{}^\sigma H,\infty}\cong \prod_{v\in S_\infty}K_{H, \sigma^{-1}\circ v}$. Here, for any $m$, we denote by $U(m)$ the compact real unitary group of rank $m$.\\\\
Lower case gothic letters denote the Lie algebra of the corresponding real Lie group (e.g., $\g_{n,v}=Lie(G_n(\cm_v))$, $\k_{H,v}=Lie(K_{H,v})$, $\h_v=Lie(H(\tr_v))$, etc. ...). 

\subsection{Highest weight modules and cohomological automorphic representations}

\subsubsection{Finite-dimensional representations}  \label{sect:finitereps}

We let $\EE_\mu$ be an irreducible finite-dimensional representation of the real Lie group $G_{n,\infty}$ on a complex vector-space, given by its highest weight $\mu=(\mu_v)_{v\in S_\infty}$. 
Throughout this paper such a representation will be assumed to be algebraic: In terms of the standard choice of a maximal torus and a basis of its complexified Lie algebra, consisting of the functionals which extract the diagonal entries, this means that the highest weight of $\EE_\mu$ has integer coordinates, $\mu_v=(\mu_{\imath_v},\mu_{\bar\imath_v})\in\Z^n\times \Z^n$ for all $v\in S_\infty$. 
We say that $\EE_\mu$ is {\it m-regular}, if $\mu_{\imath_v,i}-\mu_{\imath_v,i+1}\geq m$ and $\mu_{\bar\imath_v,i}-\mu_{\bar\imath_v,i+1}\geq m$ for all $v\in S_\infty$ and $1\leq i\leq n-1$. 
Hence, $\mu$ is regular in the usual sense (i.e., inside the open positive Weyl chamber) if and only if it is $1$-regular.\\\\
Similarly, given a unitary group $H=U(V_n)$ we let $\cF_\lambda$ be an irreducible finite-dimensional representation of the real Lie group $H_\infty$ on a complex vector-space, given by its highest weight $\lambda=(\lambda_v)_{v\in S_\infty}$, $\lambda_v\in\Z^n$. Any such $\lambda$ may also be interpreted as the highest weight of an irreducible representation of $K_{H,\infty}$. In general, we will denote by $\Lambda=(\Lambda_v)_{v\in S_\infty}$ a highest weight for $K_{H,\infty}$ and we will write $\W_\Lambda$ for the corresponding irreducible representation.

\subsubsection{Cohomological representations} \label{sect:coh}

A representation $\Pi_\infty$ of $G_{n,\infty}$ is said to be {\it cohomological} if there is a highest weight module $\EE_\mu$ as above such that $H^*(\g_{n,\infty},K_{G_n,\infty},\Pi_\infty\otimes \EE_\mu)\neq 0$. In this case, $\EE_\mu$ is uniquely determined by this property and we say $\Pi_{\infty}$ is $m$-regular if $\EE_\mu$ is. \\\\
Analogously, a representation $\pi_\infty$ of $H_\infty$ is said to be {\it cohomological} if there is a highest weight module $\cF_\lambda$ as above such that $H^*(\h_\infty,K_{H,\infty},\pi_\infty\otimes \cF_\lambda)$ is non-zero. See \cite{bowa},\S I, for details.\\\\
It can be shown that an irreducible unitary generic representation $\Pi_\infty$ of $G_{n,\infty}$ is cohomological with respect to $\EE_\mu$ if and only if at each $v\in S_\infty$ it is of the form
\begin{equation}\label{eq:ind}
\Pi_v\cong {\rm Ind}_{B(\C)}^{G(\C)}[z_1^{a_{v,1}}\bar z^{-a_{v,1}}_1\otimes ...\otimes z_n^{a_{v,n}}\bar z^{-a_{v,n}}_n],
\end{equation}
where
\begin{equation}\label{hw and it}
a_{v,j}:=a(\mu_{\imath_v},j):=-\mu_{\imath_v,n-j+1}+\tfrac{n+1}{2}-j
\end{equation}
and induction from the standard Borel subgroup $B=TN$ is unitary, cf.\ \cite[Thm.\ 6.1]{enright} (See also \cite[\S 5.5]{grob-ragh} for a detailed exposition). The set $\{z^{a_{v,i}} \bar{z}^{-a_{v,i}}\}_{1\leq i\leq n}$ is called the {\it infinity type} of $\Pi_v$. For each $v$, the numbers $a_{v,i}\in \Z+\frac{n-1}{2}$ are all different and may be assumed to be in a strictly decreasing order, i.e. $a_{v,1}>a_{v,2}>\cdots>a_{v,n}$.\\\\
If $\pi_\infty$ is an irreducible tempered representation of $H_\infty$, which is cohomological with respect to $\cF_\lambda^{\sf v}$ (the presence of the contragredient will become clear in \S \ref{construction of motive}), then each of its archimedean component-representations $\pi_v$ of $H_v\cong U(r_v,s_v)$ is isomorphic to one of the $d_v:={n \choose r_v}$ inequivalent discrete series representations denoted $\pi_{\lambda,q}$, $0\leq q< d_v$, having infinitesimal character $\chi_{\lambda_v+\rho_v}$, \cite{vozu}. As it is well-known, \cite{bowa}, II Thm.\ 5.4, the cohomology of each $\pi_{\lambda,q}$ is centered in the middle-degree
$$H^p(\h_v,K_{H,v},\pi_{\lambda,q}\otimes\cF^{\sf v}_{\lambda_v})\cong\left\{\begin{array}{ll}
 \C & \textrm{if $p=r_vs_v$} \\
 0 & \textrm{else}
\end{array}
\right.$$
We thus obtain an $S_\infty$-tuple of Harish-Chandra parameters $(\HC_v)_{v \in S_\infty}$, and $\pi_\infty \cong \otimes_{v\in S_\infty} \pi_{\HC_v}$ where $\pi_{\HC_v}$ denotes the discrete series representation of $H_v$ with parameter $\HC_v$.

\subsubsection{Global base change and $L$-packets}\label{bc}

Let $\pi$ be a cohomological square-integrable automorphic\footnote{As usual, we will for convenience not distinguish between a square-integrable automorphic representation, its smooth limit-Fr\'echet-space completion or its (non-smooth) Hilbert space completion in the $L^2$-spectrum, cf.\ \cite{grob_zun} and \cite{grob_book} for a detailed account. Moreover, unless otherwise stated, an automorphic representation is always assumed to be irreducible.} representation of $H(\A_{F^+})$. It was first  proved by Labesse \cite{lab} (see also \cite{harris-labesse, kim-krish04, kim-krish05, morel,shin}) that $\pi$ admits a base change\footnote{Referring to \cite{shin}, the very careful reader may want to assume in addition to our standing assumptions on the field $F$ that $F=\mathcal K F^+$, where $\mathcal K$ is an imaginary quadratic field. This assumption, however, will become superfluous, once the results of \cite{KMSW} are completed, i.e., established also for non-generic global Arthur parameters. } $BC(\pi) = \Pi$ to $G_n(\A_F)$: The resulting representation $\Pi$ is an isobaric sum $\Pi=\Pi_1\boxplus...\boxplus\Pi_k$ of conjugate self-dual square-integrable automorphic representations $\Pi_i$ of $G_{n_i}(\A_F)$, uniquely determined by the following: for every non-archimedean place $v$ of $F^+$, which splits in $F$ and where $\pi_v$ is unramified, the Satake parameter of $\Pi_v$ is obtained from that of $\pi_v$ by the formula for local base change, see for example \cite{ming}. 

It is then easy to see that at such places $v$, the local base change $\Pi_v$ is tempered if and only if $\pi_v$ is. The assumption that $\pi_{\infty}$ is cohomological implies moreover that $\Pi_{\infty}$ is cohomological: This was proved in \cite{lab} \S 5.1 for discrete series representations $\pi_\infty$ but follows in complete generality recalling that $\Pi_{\infty}$ has regular dominant integral infinitesimal character and hence is necessarily cohomological  by combining \cite{enright}, Thm.\ 6.1 and \cite{bowa}, III.3.3 %the existnece of the $w\in W^Q$ being shown in Delorme's thesis, see Enright Thm. 7.1 or Cozel Lem. 3.14. 
It is then a consequence of the just mentioned \cite{bowa}, III.3.3 and the results in \cite{clozel, HT, Shin} -- here in particular \cite{car}, Thm.\ 1.2 -- that, if all isobaric summands $\Pi_i$ of $\Pi=BC(\pi)$ are cuspidal, all of their local components $\Pi_{i,v}$ are tempered. Here we also used the well-known fact that as the $\Pi_i$ are unitary, $\Pi$ is fully induced from its isobaric summands.\\\\
We define the global $L$-packet $\prod(H,\Pi)$ attached to such a representation $\Pi$ to be the set of cohomological tempered square-integrable automorphic representations $\pi$ of $H(\A_{F^+})$ such that $BC(\pi) = \Pi$. This is consistent with the formalism in \cite{mok, KMSW}, in which (as in Arthur's earlier work \cite{arthur}) the representation $\Pi$ plays the role of the global Arthur-parameter for the square-integrable automorphic representation $\pi$ of $H(\A_{F^+})$. We recall that temperedness together with square-integrability imply that $\pi$ is necessarily cuspidal, \cite{clozel2}, Prop.\ 4.10, \cite{wallach}, Thm.\ 4.3. Moreover, for each each $\pi\in \prod(H_I,\Pi)$, $\pi_\infty$ is in the discrete series, cf.\ \cite{vozu}. See also \cite{clozelihes}, Lem.\ 3.8 and Lem.\ 3.9. 

\begin{rmk}\label{rmk:temp}
It should be noted that for any cohomological cuspidal automorphic representation $\pi$ of $H(\A_{F^+})$, such that $\Pi=BC(\pi)$ is an isobaric sum $\Pi=\Pi_1\boxplus...\boxplus\Pi_k$ of conjugate self-dual cuspidal automorphic representations, $\pi_v$ is tempered at every place $v$ of $F^+$, i.e., in $\prod(H,\Pi)$. Indeed, in order to see this, recall that $\Pi$ serves as a generic, elliptic global Arthur-parameter $\phi$  in the sense of \cite{KMSW}, \S 1.3.4 (Observe that as $\Pi$ is cohomological, the isobaric summands must be all different.) Its localization $\phi_v$ at any place $v$ of $F^+$ (cf.\ \cite{KMSW}, Prop.\ 1.3.3), is bounded, because so is the local Langlands-parameter attached to the tempered representation $\Pi_{v}$ by the LLC, \cite{HT, henniart}. Hence, (the unconditional item (5) of) Thm.\ 1.6.1 of \cite{KMSW} implies that each square-integrable automorphic representation $\pi$ of $H(\A_{F^+})$ attached to $\phi$ by \cite{KMSW}, Thm.\ 5.0.5 (see also the paragraph preceding this result in {\it loc.\ cit.}, making this assignment unconditional) is tempered at all places. In particular, so is $\pi$.
\end{rmk}

\subsubsection{$\sigma$-twisted representations}
Let $\sigma\in{\rm Aut}(\C)$ and let $\Pi$ be a cohomological cuspidal automorphic representation of $G_n(\A_F)$. Then it is well-known that there exists a unique cohomological cuspidal automorphic representation ${}^\sigma\Pi$ of $G_n(\A_F)$, with the property that $({}^\sigma\Pi)_f\cong {}^\sigma(\Pi_f):=\Pi_f\otimes_{\sigma^{-1}}\C$, cf.\ \cite{clozel}, Thm.\ 3.13. Likewise, if $\pi$ is a cohomological cuspidal automorphic representation of $H(\A_{F^+})$, then there is a square-integrable automorphic representation ${}^\sigma\!\pi$ of ${}^\sigma H(\A_{F^+})$, such that $({}^\sigma\!\pi)_f\cong {}^\sigma(\pi_f):=\pi_f\otimes_{\sigma^{-1}}\C$: Recalling, \cite{gro-seb}, Thm.\ A.1 and \cite{milne-suh}, Thm.\ 1.3, this can be argued as in the second paragraph of \cite{BHR}, p.\ 665. In Lem.\ \ref{lem:unique} below we will provide conditions under which ${}^\sigma\!\pi$ is cuspidal and unique.

\subsection{Critical automorphic $L$-values and relations of rationality} \label{Galois equivariance}
 
 \subsubsection{Critical points of Rankin--Selberg $L$-functions}\label{sect:critRS}
Let $\Pi=\Pi_n\otimes\Pi_{n'}$ be the tensor product of two automorphic representations of $\GL_{n}(\A_F)\times\GL_{n'}(\A_F)$. We recall that a complex number $s_0\in \tfrac{n-{n'}}{2}+\Z$ is called {\it critical} for $L(s,\Pi_{n}\times\Pi_{n'})$ if both $L(s,\Pi_{{n},\infty}\times\Pi_{{n'},\infty})$ and $L^{S}(1-s,\Pi_{{n},\infty}^{\sf v}\times\Pi_{{n'},\infty}^{\sf v})$ are holomorphic at $s=s_0$. In particular, this defines the notion of critical points for standard $L$-functions $L(s,\Pi)$ and hence Hecke $L$-functions $L(s,\chi)$.\\\\
Let now $\Pi$ (resp. $\Pi'$) be a generic cohomological conjugate self-dual automorphic representation of $G_{n}(\Acm)$ (resp.\ $G_{n'}(\Acm)$) with infinity type $\{z^{a_{v,i}} \bar{z}^{-a_{v,i}}\}_{1\leq i\leq n}$ (resp.\ $\{z^{b_{v,j}} \bar{z}^{-b_{v,j}}\}_{1\leq j\leq n'}$) at $v\in S_\infty$. Then, the $L$-function $L(s,\Pi\times \Pi')$ has critical points if and only if $a_{v,i}+b_{v,j}\neq 0$ for all $v$, $i$ and $j$, cf.\ \S5.2 of \cite{jie-thesis}. In this case, the set of critical points of $L(s,\Pi\times\Pi')$ can be described explicitly as the set of numbers $s_0\in \tfrac{n-{n'}}{2}+\Z$ which satisfy
\begin{equation}
-\min |a_{v,i}+b_{v,j}| <s_{0}\leq \min |a_{v,i}+b_{v,j}|,
\end{equation}
the minimum being taken over all $1\leq i\leq n, 1\leq j\leq n'$, and $v\in S_{\infty}$. In particular, if $n\nequiv n' \mod 2$ then $s_0=\tfrac{1}{2}$ is always among these numbers. %See \cite{jie-thesis}, \S5.2 or \cite{ragh16}, Cor.\ 2.35.

 \subsubsection{Relations of rationality and Galois equivariance}
 
 \begin{defn}[i]\label{definition algebraic relation}
Let  $E,L\subset\C$ be subfields and let $x, y\in E\otimes_\Q\C$. We write 
$$x\sim_{E\otimes_{\Q} L} y,$$ if either $y=0$, or, if $y$ is invertible and there is an $\ell\in E\otimes_\Q L$ such that $x=\ell y$ (multiplication being in terms of $\Q$-algebras). If the field $L$ equals $\Q$, it will be omitted in notation. \\
(ii) Let $E,L\subset\C$ be again subfields. Let $\underline x=\{x(\sigma)\}_{\sigma\in {\rm Aut}(\C)}$ and $\underline y=\{y(\sigma)\}_{\sigma\in {\rm Aut}(\C)}$ be two families of complex numbers. We write
$$\underline x\sim_{E} \underline y$$ 
and say that this relation {\it is equivariant under} Aut$(\C/L)$, if either $y(\sigma)=0$ for all $\sigma\in {\rm Aut}(\C)$, or if $y(\sigma)$ is invertible for all $\sigma\in {\rm Aut}(\C)$ and the following two conditions are verified:
\begin{enumerate}
\item $\cfrac{x(\sigma)}{y(\sigma)}\in \sigma(E)$ for all $\sigma$.
\item $\varrho\left(\cfrac{x(\sigma)}{y(\sigma)}\right)=\cfrac{x(\varrho\sigma)}{y(\varrho\sigma)} $ for all $\varrho\in {\rm Aut}(\C/L)$ and all $\sigma\in {\rm Aut}(\C)$.
\end{enumerate}
\end{defn}
Obviously, one may replace the first condition by requiring it only for all $\varrho$ running through representatives of ${\rm Aut}(\C)/{\rm Aut}(\C/L)$. In particular, if $L=\Q$, one only needs to verify it for the identity $id\in {\rm Aut}(\C)$. If $E$ and $L$ are furthermore number fields, one can define analogous relations for ${\rm Gal}(\Qbar/\QQ)$-families by replacing ${\rm Aut}(\C)$ by ${\rm Gal}(\Qbar/\QQ)$ and ${\rm Aut}(\C/L)$ by ${\rm Gal}(\Qbar/L)$. Note that a ${\rm Gal}(\Qbar/\QQ)$-family can be lifted to an ${\rm Aut}(\C)$-family via the natural projection ${\rm Aut}(\C)\rightarrow {\rm Gal}(\Qbar/\QQ)$, and two ${\rm Gal}(\Qbar/\QQ)$-families are equivalent if and only if their liftings are equivalent.

\begin{rem}[Aut$(\C)$-families vs. $\C^{|J_{E}|}$-tuples]\label{rem:Ealg}
Let $\underline x=\{x(\sigma)\}_{\sigma\in {\rm Aut}(\C)}$ and $\underline y=\{y(\sigma)\}_{\sigma\in {\rm Aut}(\C)}$ be two ${\rm Aut}(\C)$-families and assume we are given two number fields $E,L\subset \C$. If the individual numbers $x(\sigma)$, $y(\sigma)$ only depend on the restriction of $\sigma$ to $E$, then we may identify $\underline x$ and $\underline y$ canonically with elements $x,y\in\C^{|J_{E}|}\cong E\otimes_\Q\C$. The assertion that $\underline x\sim_{E} \underline y$, equivariant under Aut$(\C/L)$ implies that $x\sim_{E\otimes_{\Q} L} y$. 

Conversely, any element $x\in E\otimes_\Q\C \cong \C^{|J_{E}|}$ can be extended to a ${\rm Aut}(\C)$-family $\underline x=\{x(\sigma)\}_{\sigma\in {\rm Aut}(\C)}$, putting $x(\sigma):=x_{\sigma|_E}$. If we assume moreover that $E$ contains $L^{Gal}$, then for $x,y\in  E\otimes_\Q\C \cong \C^{|J_{E}|}$, the assertion $x\sim_{E\otimes_{\Q} L} y$ implies that $\underline x\sim_{E} \underline y$, equivariant under Aut$(\C/L)$.

In this paper, it will be convenient to have both points of view at hand. In fact, we prove assertions of the second type, which is generally a little bit stronger than the first one. But as we are always in the situation that $E$ contains $L^{Gal}$, the two assertions are equivalent and we will jump between them without further mention.  
\end{rem}

\subsection{Interlude: A brief review of motives and Deligne's conjecture}\label{general Deligne}

\subsubsection{Motives, periods over $\Q$ and Deligne's conjecture} We now quickly recall Deligne's conjecture about motivic $L$-functions, in order to put our main results into a precisely formulated framework and to fix notation. We follow Deligne, \cite{deligne}, \S 0.12, in adopting the following (common) pragmatic point of view through realizations:

\begin{defn}\label{definitionmotive}
A {\it motive} $M$ over a number field $k$ with {\it coefficients} in a number field $E(M)$ is a tuple
$$M=(M_{B,\imath}, M_{dR}, M_{\textrm{{\it \'et}}}; F_{B,\imath}, I_{\infty,\imath}, I_{\textrm{{\it \'et}},\imath}),$$
where $\imath\in J_k$ runs through the embeddings $k\hra\C$ and such that there exists an $n\geq 1$, where
\begin{enumerate}

\item[(B)] $M_{B,\imath}$ is an $n$-dimensional $E(M)$-vector space, together with a Hodge-bigraduation 
$$M_{B,\imath}\otimes_\Q \C = \bigoplus_{p,q} M^{p,q}_{B,\imath}$$
as a module over $E(M)\otimes_\Q\C$. %We assume that the dimension of each $M^{p,q}_{B,\imath}$ is independent of $\imath$. 

\item[(dR)] $M_{dR}$ is a free $E(M)\otimes_\Q k$-module of rank $n$, equipped with a decreasing filtration $\{F^i_{dR}(M)\}_{i\in\Z}$ of $E(M)\otimes_\Q k$-submodules.

\item[(\'et)] $M_{\textrm{{\it \'et}}} = \{M_\ell\}_\ell$ is a strictly compatible system, cf.\ \cite{serrec} p.\ 11, of $\ell$-adic Gal$(\overline k/k)$-representations 
$$\rho_{M,\ell}: {Gal}(\overline k/k) \ra \GL(M_\ell)$$
on $n$-dimensional $E(M)_\ell$-vector spaces $M_\ell$, $\ell$ running through the set of finite places of $E(M)$,
\end{enumerate}
to be called ``realizations of $M$'', together with
\begin{enumerate}
\item[(i)] an $E(M)$-linear isomorphism 
$$F_{B,\imath}: M_{B,\imath} \ira M_{B,\overline\imath},$$
which satisfies $F_{B,\imath}^{-1} = F_{B,\overline\imath}$ and commutes with complex conjugation on the Hodge-bigraduation from (B), i.e., $\overline{F_{B,\imath}(M^{p,q}_{B,\imath})} \subseteq M^{p,q}_{B,\overline\imath}$,
\item[(ii)] an isomorphism of $E(M)\otimes_\Q \C$-modules
$$I_{\infty,\imath}: M_{B,\imath}\otimes_\Q\C \ira M_{dR}\otimes_{k,\imath}\C,$$
compatible with the Hodge-bigraduation from (B) and the decreasing filtration from (dR) above, i.e., $I_{\infty,\imath}(\bigoplus_{p\geq i} M^{p,q}_{B,\imath})=F^i_{dR}(M)\otimes_{k,\imath}\C$, and also compatible with $F_{B,\imath}$ and complex conjugation, i.e., $\overline{I_{\infty,\imath}} = I_{\infty,\overline\imath}\circ\overline{F_{B,\imath}}$, and  
\item[(iii)] a family $I_{\textrm{{\it \'et}},\imath} = \{I_{\imath,\ell}\}_\ell$ of isomorphisms of $E(M)_\ell$-vector spaces
$$I_{\imath,\ell}: M_{B,\imath}\otimes_{E(M)} E(M)_\ell \ira M_\ell,$$
$\ell$ running through the set of finite places of $E(M)$, where, if $\imath\in J_k$ is real, then $I_{\imath,\ell}\circ F_{B,\imath} = \rho_{M,\ell}(\gamma_\imath)\cdot I_{\imath,\ell}$, where $\gamma_\imath$ denotes complex conjugation of $\C$ attached to any extension to $\overline k$ of the embedding $\imath: k\hra \C$.
\end{enumerate}
to be called ``comparison isomorphisms''. The common rank $n$ of each realization as a free module is called the {\it rank} of $M$. If $n\geq 1$ and if there is an integer $w$ sucht that $M^{p,q}_{B,\imath}=\{0\}$ whenever $p+q\neq w$, then $M$ is called {\it pure of weight} $w$.
\end{defn}

The \'etale realization allows one to define the $E(M)\otimes_\Q\C$-valued $L$-{\it function} $L(s,M)$ of $M$ as the usual Euler product over the prime ideals $\p\lhd\O_k$,
$$L(s,M):=\left(\prod_{\p} L_\p(s,M)^\jmath\right)_{\jmath\in J_{E(M)}},$$
where $L_\p(s,M):=\det(id - N(\p)^{-s}\cdot \rho_{M,\ell}(Fr^{-1}_\p)| M^{I_\p}_{\ell})^{-1}$, and $Fr_\p$ denotes the geometric Frobenius locally at $\p$ (modulo conjugation) and $I_\p$ is the inertia subgroup in the decomposition group of an(y) extension of $\p$ to $\overline k$. Consequently, viewing $L_\p(s,M)$ as a rational function in the variable $X=N(\p)^{-s}$, the action of $\jmath\in J_{E(M)}$ on $L_\p(s,M)$ is defined by application to its coefficients: Here, we have to adopt the usual hypothesis, cf.\ \cite{deligne}, \S 1.2.1 \& \S 2.2, that at the finitely many ideals $\p$, where $\rho_{M,\ell}$ ramifies, the coefficients of $L_\p(s,M)$, viewed as a rational function in this way, belong to $E(M)$ and that they are independent of $\ell$ not dividing $N(\p)$, in order to obtain a well-defined element of $E(M)\otimes_\Q\C\cong \prod_{\jmath}\C$ (i.e., to make sense of the action of $\jmath$). It is well-known that $L(s,M)$ is absolutely convergent for $Re(s)\gg 0$ and it is tacitly assumed that $L(s,M)$ admits a meromorphic continuation to all $s\in\C$ as well as the usual functional equation with respect to the dual motive $M^{\sf v}$ (whose system of $\ell$-adic representations is contragredient to that of $M$), cf.\ \cite{deligne}, \S 2.2. An integer $m$ is then called {\it critical} for $L(s,M)$, if the archimedean $L$-functions on both sides of the functional equation are holomorphic at $s=m$. We refer to \cite{deligne}, \S 5.2 for the construction of the archimedean $L$-functions attached to $M$ and its dual. \\\\
Let now be $M$ a pure motive of weight $w$. By considering the motive $R_{k/\Q}(M)$, which is obtained from $M$ by applying restriction of scalars (i.e., whose system of $\ell$-adic representations is obtained by inducing the one attached to $M$ from Gal$(\overline k/k)$ to Gal$(\overline \Q/\Q)$) we may always reduce ourselves to the case, where $M$ is defined over $\Q$, which is the framework in which Deligne's conjecture is stated. As we are then left with only one embedding $\imath=id$, we will drag it along in order to lighten the burden of notation. \\\\
So, let $F_\infty=F_{B,id}: M_{B} \ira M_{B}$ be the only infinite Frobenius. If $w=2p$ is even, we suppose that $F_\infty$ acts by multiplication by $\pm 1$ on $M^{p,p}_B$. We then denote by $n^\pm=n^\pm(M)$ the dimension of the $+1$- (resp.\ $-1$-eigenspace) $M^\pm_B$ of $M_B$ of the involution $F_\infty$. Let $F_{dR}^\pm$ be $E(M)$-subspaces of $M_{dR}$, given by the filtration $\{F^i_{dR}(M)\}_{i\in\Z}$, such that the rank of $M^\pm_{dR}:=(M_{dR}/F^\mp_{dR})$ equals $n^\pm$ and such that $I_\infty$ induces isomorphisms of $E(M)\otimes_\Q\C$-modules
$$I^\pm_\infty: M^\pm_B\otimes_\Q\C \ira M^\pm_{dR}\otimes_\Q\C.$$
Following Delgine, we define two {\it periods}
$$c^\pm(M):=(\det(I^\pm_\infty)_\jmath)_{\jmath\in J_{E(M)}}\in (E(M)\otimes_\Q\C)^\times,$$
and
$$\delta(M):= (\det(I_\infty)_\jmath)_{\jmath\in J_{E(M)}}\in (E(M)\otimes_\Q\C)^\times.$$
Here, each determinant is computed with respect to a fixed choice of $E(M)$-rational bases of source and target spaces. Up to multiplication by an invertible element in the $\Q$-algebra $E(M)$, both periods hence depend only on $M$. 

\begin{conj}[Deligne, \cite{deligne}, Conj.\ 2.8]\label{conj:Deligne}
Let $M$ be a pure motive of weight $w$ over $\Q$ and let $m$ be a critical point for $L(s,M)$. Then 
$$L^{S}(m,M)\sim_{E(M)} (2\pi i)^{n^{(-1)^m}\cdot m} \ c^{(-1)^m}(M)$$
\end{conj}

\subsubsection{Factorizing periods} 
Switching back to our general number field $k$, we choose and fix a section $S_\infty(k)\ra J_k$ and let $\Sigma_k$ be its image in $J_k$. If $w=2p$ is even, we assume, similar to the case $k=\Q$, that $R_{k/\Q}(\bigoplus_{\imath\in\Sigma_k} F_{B,\imath})$ acts by a scalar on $R_{k/\Q}(M)^{p,p}_B$. For $\imath\in \Sigma_k$ complex, this implies that $M^{p,p}_{B,\imath}=\{0\}$. Next, one may analogously define $\pm 1$--eigenspaces of $F_{B,\imath}$, which now obviously have to depend of the nature of the embedding $\imath\in J_k$: If $\imath$ is real, then our definition of $M_{B,\imath}^\pm$ is verbatim the one of the case $k=\Q$ from above, whereas if $\imath$ is complex, then we obtain eigenspaces $(M_{B,\imath}\oplus M_{B,\overline\imath})^\pm$ of the direct sum $M_{B,\imath}\oplus M_{B,\overline\imath}$. We also may analogously  define spaces $F^\pm_{dR}$ attached to the Hodge-filtration $\{F^i_{dR}(M)\}_{i\in\Z}$, cf.\ \cite{yoshida}, pp.\ 149--150, and we set $M^\pm_{dR}:=(M_{dR}/F^\mp_{dR})$. For $\imath\in\Sigma_k$, the maps $I_{\infty,\imath}$ induce canonical isomorphisms of $E(M)\otimes_\Q\C$-modules
$$I^\pm_{\infty,\imath}:M^\pm_{B,\imath}\otimes_\Q\C \ira M^\pm_{dR}\otimes_{k,\imath}\C,$$
if $\imath$ is real and 
$$I^\pm_{\infty,\imath}:(M_{B,\imath}\oplus M_{B,\overline\imath})^\pm\otimes_\Q\C \ira (M^\pm_{dR}\otimes_{k,\imath}\C)\oplus(M^\pm_{dR}\otimes_{k,\overline\imath}\C)$$
if $\imath$ is complex. 
We define
$$c^\pm(M,\imath):=(\det(I^\pm_{\infty,\imath})_\jmath)_{\jmath\in J_{E(M)}}\in (E(M)\otimes_\Q\C)^\times$$
and
$$\delta(M,\imath):= (\det(I_{\infty,\imath})_\jmath)_{\jmath\in J_{E(M)}} \in (E(M)\otimes_\Q\C)^\times.$$
Up to multiplication by an invertible element in $E(M)\otimes_\Q\imath(k)$, they only depend on $M$. Finally, let $n^\pm$ be the rank of the free $E(M)\otimes_\Q k$-module $M^\pm_{dR}$, if $k$ has a real place (respectively, if $k$ is totally imaginary, let $2n^\pm$ be the rank of the free $E(M)\otimes_\Q k$-module $M^\pm_{dR}\oplus \overline{M^\pm_{dR}}$, where $\overline{M^\pm_{dR}}$ is the $E(M)\otimes_\Q k$-module $M^\pm_{dR}$, but with complex conjugated scalar-mulitplication by $k$: $x\star v:=\overline x \cdot v$, $x\in k$, $v\in M^\pm_{dR}$.) Then, the two perspectives of Deligne's periods are linked by the following relations as elements of $\Q$-algebras:
$$c^\pm(R_{k/\Q}(M)) \sim_{E(M) K} D_k^{n^\pm/2}\prod_{\imath\in \Sigma_k} c^\pm(M,\imath)$$
$$\delta(R_{k/\Q}(M)) \sim_{E(M) K} D_k^{n/2} \prod_{\imath\in J_k} \delta(M,\imath),$$
where $K$ (resp.\ $D_k$) denotes the normal closure (resp.\ discriminant) of $k/\Q$, the latter identified with $1\otimes D_k$ in $E(M)\otimes_\Q \C$, cf.\ \cite{yoshida}, Prop.\ 2.2. We also refer to Prop.\ 2.11 of \cite{harris-lin} for a finer decomposition over $E(M)$.\\\\
Finally, we also recall the notion of regularity: To this end, we assume that we are given a motive $M$ with coefficients in $E(M)$. Since $E(M)\otimes_\Q\C\cong\C^{|J_{E(M)|}}$, for each $\imath\in J_k$,  there is a decomposition of $\C$-vector spaces
$$M^{p,q}_{B,\imath}=\bigoplus_{\jmath\in J_{E(M)}} M^{p,q}_{B,\imath}(\jmath).$$ 
We say that $M$ is {\it regular}, if $\dim M^{p,q}_{B,\imath}(\jmath)\leq 1$ for all $p,q\in\Z$, $\imath\in J_k$ and $\jmath\in J_{E(M)}$. For a fixed pair $(\imath,\jmath)$ as above, the set of pairs $(p,q)$, such that $M^{p,q}_{B,\imath}(\jmath)\neq 0$ is then called the {\it Hodge-type of $M$ at $(\imath,\jmath)$} and $(p,q)$ a {\it Hodge weight}. The Hodge-type is particularly useful, to give an explicit description of the critical points of $L(s,M)$. Indeed, if $M$ is regular and pure of even weight $w$, assume that $(\frac{w}{2},\frac{w}{2})$ is not a Hodge weight at any pair of embeddings. Then an integer $m$ is critical for $L(s,M)$ if and only if 
\begin{equation}
-\min_{(p,q)}\{ |p-\tfrac{w}{2}|\}+\tfrac{w}{2}< m\leq \min_{(p,q)}\{ |p-\tfrac{w}{2}|\}+\tfrac{w}{2}
\end{equation}
where $(p,q)$ runs over the Hodge weights for all pair $(\imath,\jmath)$.

\subsection{Motivic split indices}
Let $M$ and $M'$ be regular pure motives over $F$ with coefficients in a number field $E(M)=E(M')$ of weight $w$ and $w'$, respectively. We write $n$ for the rank of $M$ and $n'$ for the rank of $M'$. We write the Hodge-type of $M$ (resp.\ $M'$) at $(\imath,\jmath)$ as $(p_i,w-p_i)_{1\leq i\leq n}$, with $p_1>...>p_n$ (resp.\  $(q_j,w'-q_j)_{1\leq j\leq n'}$, with $q_1>...>q_{n'}$). Consider the tensor product $M\otimes M'$ (over $F$), whose system of $\ell$-adic representations is simply the system of tensor products $M_\ell\otimes M'_\ell$. We assume that $(M\otimes M')^{p,q}_{B,\imath}$ vanishes at $p=q=\tfrac{w+w'}{2}$, i.e., that $(\tfrac{w+w'}{2},\tfrac{w+w'}{2})$ is not a Hodge weight, i.e., $p_i+q_j\neq \tfrac{w+w'}{2}$ for all $i,j$. We put $p_{0}:=+\infty$ and $p_{n+1}:=-\infty$, and define:

$$sp(i,M;M',\imath,\jmath):=\#\{1\leq j\leq n'\mid p_{i}-\tfrac{w+w'}{2}>-q_{j}>p_{i+1}-\tfrac{w+w'}{2}\}.$$
We call $sp(i,M;M',\imath,\jmath)$ a {\it (motivic) split index}, reflecting the fact that the sequence of inequalities $-q_{n'}>...>-q_1$ splits into exactly $n+1$ parts, when merged with $p_1-\tfrac{w+w'}{2}>...>p_n-\tfrac{w+w'}{2}$, where the length of the $i$-th part in this splitting is $sp(i,M;M',\imath,\jmath)$. This gives rise to the following

\begin{defn}\label{split motivic}
For $0\leq i\leq n$ and $\imath\in\Sigma$, we define the {\it (motivic) split indices} (cf.\ \cite{harris-lin}, Def.\ 3.2)
$$sp(i,M;M',\imath):=(sp(i,M;M',\imath,\jmath))_{\jmath\in J_{E(M)}}\in\N^{J_{E(M)}},$$
and, mutatis mutandis, 
$$sp(j,M';M,\imath):=(sp(j,M';M,\imath,\jmath))_{\jmath\in J_{E(M')}}\in\N^{J_{E(M')}},$$
\end{defn}

\subsection{Motivic periods}\label{motivic periods}

Let $M$ be a regular pure motive over $\cm$ of rank $n$ and weight $w$ with coefficients in a number field $E\supset F^{Gal}$. For $1\leq i\leq n$ and $\imath \in \Sigma$, we have defined {\it motivic periods} $Q_{i}(M,\imath)$ in \cite{harrisANT} (see \cite{harris-lin}, Def.\ 3.1 for details). They are elements in $E\otimes_\Q \C$, well-defined up to multiplication by elements in $E\otimes_\Q \imath(\cm)$. If $M$ is moreover polarised, i.e., if $M^{{\sf v}}\cong M^{c}$, the period $Q_{i}(M,\imath)$ is equivalent to the inner product of a vector in $M_{B,\imath}$, the Betti realisation of $M$ at $\imath$, whose image via the comparison isomorphism is inside $i$-th bottom degree of the Hodge filtration for $M$. We have furthermore defined
\begin{equation}\label{eq:Qperiods}
Q^{(i)}(M,\imath):=Q_{0}(M,\imath)Q_{1}(M,\imath)\cdots Q_{i}(M,\imath), 
\end{equation}
where $Q_{0}(M,\imath):=\delta(M,\imath)(2\pi i)^{n(n-1)/2}$. Then, Deligne's periods can be interpreted interpreted in terms of the above motivic periods:

\begin{prop}(cf.\ \cite{harris-lin}, Prop.\ 2.11 and 3.13)\label{Deligne period motivic}
Let $M$ be a regular pure motive over $\cm$ of rank $n$ and weight $w$ with coefficients in a number field $E\supset F^{Gal}$ and let $M'$ be a regular pure motive over $\cm$ of rank $n'$ and weight $w'$ with coefficients in a number field $E'\supset F^{Gal}$. We assume that $(\tfrac{w+w'}{2},\tfrac{w+w'}{2})$ is not a Hodge weight for the motive $M\otimes M'$ with coefficients in $EE'$. Then, the Deligne periods satisfy
\begin{eqnarray}
&c^{\pm}(R_{\cm/\Q}(M\otimes M'))&\\ \nonumber
 &\sim_{EE'\otimes_\Q F^{Gal}}  (2\pi i)^{-\frac{nn'd(n+n'-2)}{2}} \prod\limits_{\imath\in \Sigma} [\prod\limits_{j=0}^{n}Q^{(j)}(M,\imath)^{sp(j,M;M',\imath)}\prod\limits_{k=0}^{n'}Q^{(k)}(M',\imath)^{sp(k,M';M,\imath)}].
\end{eqnarray} 
\end{prop}

\section{Translating Deligne's conjecture into an automorphic context}\label{auto Deligne}

\subsection{CM-periods and special values of Hecke characters}\label{CM-periods}
\subsubsection{Special Simura data and CM-periods}\label{sect:CM}
Let $(T,h)$ be a Shimura datum where $T$ is a torus defined over $\Q$ and $h:R_{\C/\R}(\mathbb{G}_{m,\C})\rightarrow T_{\R}$ a homomorphism satisfying the axioms defining a Shimura variety, cf.\ \cite{milne}, II. Such pair is called a {\it special} Shimura datum. Let $Sh(T,h)$ be the associated Shimura variety and let $E(T,h)$ be its reflex field.\\\\
For $\chi$ an algebraic Hecke character of $T(\A_\Q)$, we let $E_T(\chi)$ be the number field generated by the values of $\chi_f$, $E(T,h)$ and $F^{Gal}$, i.e., the composition of the rationality field $\Q(\chi_f)$ of $\chi_f$, and $E(T,h) F^{Gal}$. If it is clear from the context, we will also omit the subscript ``$T$''. We may define a non-zero complex number $p(\chi,(T,h))$, called {\it CM-period}, as in Sect.\ $1$ of \cite{Harris93} and the appendix of \cite{harrisappendix}, to which we refer for details: It is defined as the ratio between a certain deRham-rational vector and a certain Betti-rational vector inside the cohomology of the Shimura variety with coefficients in a local system. As such, it is well-defined modulo $E_T(\chi)^{\times}$. Recall the $\sigma$-twisted Shimura datum, $({}^\sigma T,{}^\sigma h)$, $\sigma\in {\rm Aut}(\C)$, cf.\ \cite{milne}, II.4, \cite{langlands}. By taking ${\rm Aut}(\C)$-conjugates of the aforementioned rational vectors, we can define the family $\{p({}^\sigma\chi,({}^\sigma T,{}^\sigma h))\}_{\sigma\in {\rm Aut}(\C)}$, such that if $\sigma$ fixes $E_{T}(\chi)$ then $p({}^\sigma\chi,({}^\sigma T,{}^\sigma h))=p(\chi,(T,h))$, i.e., in view of Rem.\ \ref{rem:Ealg}, $\{p({}^\sigma\chi,({}^\sigma T,{}^\sigma h))\}_{\sigma\in {\rm Aut}(\C)}$ defines an element in $\C^{|J_{E_T(\chi)}|}\cong E_T(\chi)\otimes_\Q\C$. The following proposition holds ${\rm Aut}(\C)$-equivariantly as interpreted for the family $\{p({}^\sigma\chi,({}^\sigma T,{}^\sigma h))\}_{\sigma\in {\rm Aut}(\C)}$:\\

\begin{prop}\label{relation CM-period}
Let $T$ and $T'$ be two tori defined over $\Q$ both endowed with a special Shimura datum $(T,h)$ and $(T',h')$ and let $u:(T',h')\rightarrow (T,h)$ be a homomorphism between them. Let $\chi$ be an algebraic Hecke character of $T(\A_\Q)$ and put $\chi':=\chi\circ u$, which is an algebraic Hecke character of $T'(\A_\Q)$. Then we have:
\begin{equation}\nonumber
p(\chi,(T,h)) \sim_{E_T(\chi)} p(\chi',(T',h')).
\end{equation}
Interpreted as families, this relation is equivariant under the action of ${\rm Aut}(\C)$. 
\end{prop}
\begin{proof} 
This is due to the fact that both the Betti-structure and the deRham-structure commute with the pullback map on cohomology. We refer to \cite{Harris93}, in particular relation $(1.4.1)$ for details.
\end{proof}
If $\Psi$ a set of embeddings of $F$ into $\C$ such that $\Psi\cap \Psi^{c}=\emptyset$, one can define a special Shimura datum $(T_{F},h_{\Psi})$ where $T_{F}:=R_{F/\Q}(\mathbb{G}_{m})$ and $h_{\Psi}:R_{\C/\R}(\mathbb{G}_{m,\C}) \rightarrow T_{F,\R}$ is a homomorphism such that over $\imath\in J_{F}$, the Hodge structure induced by $h_{\Psi}$ is of type $(-1,0)$ if $\imath\in \Psi$, of type $(0,-1)$ if $\imath\in \Psi^{c}$, and of type $(0,0)$ otherwise. In this case, for $\chi$ an algebraic Hecke character of $F$, we write $p(\chi,\Psi)$ for $p(\chi, (T_{F},h_{\Psi}))$ and abbreviate $p(\chi, \imath):=p(\chi,\{\imath\})$. We also define the (finite) compositum of number fields $E_F(\chi):=\prod_\Psi E_{T_F}(\chi)$.

\begin{lem}\label{Lemma CM} 
Let $\imath\in \Sigma$ and let $\Psi$ and $\Psi'$ be disjoint sets of embeddings of $F$ into $\C$ such that $\Psi\cap \Psi^{c}=\emptyset= \Psi'\cap \Psi'^{c}$. Let $\chi$ and $\chi'$ be algebraic Hecke characters of $\GL_1(\A_F)$, whose archimedean component is not a power of the norm $\|\cdot\|_\infty$, and recall the algebraic Hecke character $\psi$ from Sect.\ \ref{sect:fields}. Then,
\begin{enumerate}[label=(\alph*)] 
\item $p(\chi, \Psi\sqcup \Psi') \sim_{E_F(\chi)} p(\chi,\Psi) \ p(\chi, \Psi')$
\item  $p(\chi\chi',\Psi) \sim_{E_F(\chi)E_F(\chi')} p(\chi,\Psi) \ p(\chi',\Psi) $
\item If $\chi$ is conjugate selfdual, then $p(\widecheck{\chi},\bar{\imath})\sim_{E_F(\chi)} p(\widecheck{\chi},\imath)^{-1}.$ 
\item $p(\widecheck{\psi},\bar{\imath})\sim_{E_F(\psi)} (2\pi i)p(\psi,\imath)^{-1}$.
\end{enumerate}
Interpreted as families, these relations are equivariant under the action of ${\rm Aut}(\C)$.
\end{lem}
\begin{proof}
The first two assertions are proved in \cite{grob_lin}, Prop.\ 4.4. For (c), observe that by Lem.\ $1.6$ of \cite{H97}, we have $p(\widecheck{\chi},\bar{\imath})\sim_{E_F(\chi)} p(\widecheck{\chi}^{c},\imath)$. Then Prop.\ $1.4$ of \cite{H97} and the fact that $\chi \chi^{c}$ is trivial imply
$$p(\widecheck{\chi}^{c},\imath)p(\widecheck{\chi},\imath)\sim_{E_F(\chi)}p(\widecheck{\chi \chi^{c}},\imath)\sim_{E_F(\chi)}1.$$ 
Similarly, for the last assertion, we have
$p(\widecheck{\psi},\bar{\imath})\sim_{E_F(\psi)} p(\widecheck{\psi}^{c},\imath)\sim_{E_F(\psi)} p(\widecheck{\psi}^{-1}\|\cdot \|^{-1},\imath)\sim_{E_F(\psi)}(2\pi i)p(\psi,\imath)^{-1}$ where the last step is due to the fact that $p(\|\cdot\|,\imath)\sim_{\Q}(2\pi i)^{-1}$ (cf. $1.10.9$ of \cite{H97}). 
\end{proof}

\subsubsection{Relation to critical Hecke $L$-values}
It is proved by Blasius that the CM-periods are related to Hecke $L$-values (cf.\  \cite{Blasius}, Thm.\ 7.4.1 and Thm.\ 9.2.1). We state Blasius's result in the form of Prop.\ $1.8.1$ of \cite{Harris93}, resp.\ \cite{harrisappendix}, Prop.\ A.10 (see also erratum on page $82$ of \cite{H97} and \S 3.6 therein):

\begin{thm}\label{CM}
Let $\chi$ be an algebraic Hecke character of $F$ with infinity type $z^{a_{v}}\bar{z}^{b_{v}}$ at $v\in S_\infty$, such that $a_{v}\neq b_{v}$ for all $v$. We define $\Psi_{\chi}:=\{\imath_v\mid a_{v}<b_{v}\}\cup \{\bar{\imath}_v\mid a_{v}>b_{v}\}$. Then, if $m$ is critical for $L(s,\chi)$, we have for any finite set of places $S$, containing $S_\infty$,
$$L^S(m, \chi)\sim_{E_F(\chi)} (2\pi i)^{dm}p(\widecheck{\chi},\Psi_{\chi}).$$
Interpreted as families, this relation is equivariant under the action of ${\rm Aut}(\C/F^{Gal})$. 
\end{thm}

\subsection{Arithmetic automorphic periods}\label{arithmetic automorphic period}
\subsubsection{A theorem of factorization}

In this paper we focus on cohomological conjugate self-dual, cuspidal automorphic representations $\Pi$ of $\GL_{n}(\Acm)$, which satisfy the following assumption: 

\begin{hyp}\label{descent} 
For each $I=(I_{\imath})_{\imath\in \Sigma}\in \{0,1,\cdots,n\}^{|\Sigma|}$ there is a unitary group $H_I$ over $F^{+}$ as in \S \ref{sect:alggrp} of signature $(n-I_\imath,I_\imath)$ at $v=(\imath,\bar\imath)\in S_\infty$ such that the global $L$-packet $\prod(H_I,\Pi)$ is non-empty. Moreover, if $\pi\in \prod(H_I,\Pi)$, then the packet also contains all the representations $\tau_\infty\otimes\pi_f$, $\tau_\infty$ running through the discrete series representation of $H_{I,\infty}$ of the same infinitesimal character of $\pi_\infty$.
\end{hyp}

\begin{rmk}\label{rmk:des}
This hypothesis is always satisfied, if $n$ is odd. For $n$ even it is also known to hold, if $\Pi_\infty$ is cohomological with respect to a regular representation and $\Pi_v$ is square-integrable at a non-archimedean place $v$ of $F^+$, which is split in $F$. Moreover, it is well-known that a cohomological conjugate self-dual, cuspidal automorphic representation of $\GL_{n}(\Acm)$ always descends to a cohomological cuspidal automorphic representation of the quasi-split unitary group $U^*_n$ of rank $n$ over $F^+$, cf.\ \cite{harris-labesse} and \cite{mok}, Cor.\ 2.5.9 (and the argument in \cite{grob_harris_lapid}, \S 6.1). In contrast to this positive result, there are also cohomological conjugate self-dual, cuspidal automorphic representations $\Pi$ of $\GL_{n}(\Acm)$, which do not satisfy Hyp.\ \ref{descent}: As the simplest counterexample, take an everywhere unramified Hilbert modular cusp form for a real quadratic field $F^+$ not of CM-type. The quadratic base change of the corresponding automorphic representation to a CM-quadratic extension $F$ does not descend to a unitary group of signature $(1,1)$ at one archimedean place and $(2,0)$ at the other.
\end{rmk}
If $\Pi$ satisfies Hyp.\ \ref{descent}, a family of \textit{arithmetic automorphic periods} $\{P^{({}^{\sigma}I)}({}^{\sigma}\Pi)\}_{\sigma\in {\rm Aut}(\C)}$ can then be defined as the Petersson inner products of an ${\rm Aut}(\C)$-equivariant family of arithmetic holomorphic automorphic forms as in (2.8.1) of \cite{H97}, or Definition 4.6.1 of \cite{jie-thesis}. The following result is proved in \cite{linfactorization}, Thm.\ 3.3.

\begin{thm}[Local arithmetic automorphic periods]\label{thm:artihap} 
Let $\Pi$ be a cohomological conjugate self-dual cuspidal automorphic representation of $\GL_{n}(\Acm)$, which satisfies Hyp.\ \ref{descent}. We assume that either $\Pi$ is $5$-regular, or $\Pi$ is regular and Conj.\ \ref{nonvan} below is true. Then there exists a number field $E(\Pi)\supseteq \Q(\Pi_f) F^{Gal}$ (see \S \ref{sect:EPi} below) and families of {\rm local} arithmetic automorphic periods $\{P^{(i)}({}^{\sigma}\Pi,\imath)\}_{\sigma\in {\rm Aut}(\C)}$, for $0\leq i \leq n$ and $\imath\in\Sigma$, which are unique up to multiplication by elements in $E(\Pi)^\times$ such that \begin{equation}\label{local end}
P^{(0)}(\Pi,\imath)\sim_{E(\Pi)} p(\widecheck{\xi}_{\Pi},\bar\imath) \quad\quad {\it and} \quad\quad P^{(n)}(\Pi,\imath)\sim_{E(\Pi)} p(\widecheck{\xi}_{\Pi},\imath),
\end{equation} 
where $\xi_{\Pi}$ denotes the central character of $\Pi$, and satisfy the relation
\begin{equation}\label{eq:splitting}
 P^{(I)}(\Pi) \sim_{E(\Pi)} \prod\limits_{\imath\in\Sigma}P^{(I_{\imath})}(\Pi,\imath).
 \end{equation}
 In particular, we have \begin{equation}\label{local end 2}
P^{(0)}(\Pi,\imath) P^{(n)}(\Pi,\imath)\sim_{E(\Pi)} 1.
\end{equation} 
Interpreted as families, all relations are equivariant under the action of ${\rm Aut}(\C/F^{Gal})$. 
\end{thm}

In the statement of Thm.\ \ref{thm:artihap}, if $\Pi$ is not $5$-regular, we made use of the following 

\begin{conj}\label{nonvan}
Let $S$ be a finite set of non-archimedean places of $F^+$ and for each $v\in S$ let $\alpha_v:  \GL_1(\mathcal O_{F^+,v})\ra \CC^{\times}$ be a given continuous character. Let $\alpha_{\infty}$ be a conjugate self-dual algebraic character of $\GL_{1}(F\otimes_{\Q}\R)$.
Let $\Pi$ be a cohomological conjugate self-dual cuspidal automorphic representation of $\GL_{n}(\Acm)$, which satisfies Hyp.\ \ref{descent}.  Then there exists a Hecke character $\chi$ of $ \GL_1(\A_F)$, such that $\chi_{\infty}=\alpha_\infty$, $\chi_{|_{ \GL_1(\mathcal O_{F^+,v})}} = \alpha_v$, $v\in S$, and
\begin{equation}\label{nonvan1} 
L^{S}(\tfrac{1}{2},\Pi\otimes \chi) \neq 0. 
\end{equation}
\end{conj}

\subsubsection{The field $E(\Pi)$}\label{sect:EPi} 
Let $\pi\in \prod(H_I,\Pi)$. It is easy to see that the {\it field of rationality} $\QQ(\pi_f)$ of $\pi_f$, which, as we recall, is defined as the fixed field in $\C$ of the subgroup of $\sigma \in {\rm Aut}(\C)$ such that ${}^\sigma\pi_f \cong \pi_f$, coincides with $\QQ(\Pi_f)$ if the local $L$-packets that base change to $\Pi$ are singletons, and are simple finite extensions of $\QQ(\Pi_f)$ otherwise, see Prop.\ \ref{contained}.   However, because of the presence of non-trivial Brauer obstructions it is not always possible to realize $\pi_f$ over $\QQ(\pi_f)$.  By {\it field of definition} we mean a field over which $\pi_f$ has a model. The field $E(\Pi)$ in the statement of Thm. \ref{thm:artihap} may be taken to be the compositum of $F^{Gal}$ with fields of definition of the descents $\pi\in \prod(H_I,\Pi)$, $I=(I_{\imath})_{\imath\in \Sigma}\in \{0,1,\cdots,n\}^{|\Sigma|}$, as in Hyp.\ \ref{descent}, cf.\ \cite{linfactorization}, Thm.\ 2.2.: It follows from \cite{gro-seb2}, Thm.\ A.2.4, that these (finitely many) fields of definition exist and are number fields. In fact, they can be taken to be finite abelian extension of the respective $\QQ(\pi_f)$ {\it From now on, $E(\Pi)$ will stand for (any fixed choice of) such a field.}

\subsection{Automorphic split indices}
Let $n$ and $n'$ be two integers. Let $\Pi$ (resp. $\Pi'$) be a cohomological conjugate self-dual cuspidal automorphic representation of $G_{n}(\Acm)$ (resp.\ $G_{n'}(\Acm)$) with infinity type $\{z^{a_{v,i}} \bar{z}^{-a_{v,i}}\}_{1\leq i\leq n}$ (resp.\ $\{z^{b_{v,j}} \bar{z}^{-b_{v,j}}\}_{1\leq j\leq n'}$) at $v\in S_\infty$.

\begin{defn}\label{split automorphic}
For $0\leq i\leq n$ and $\imath_v\in\Sigma$, we define the {\it automorphic split indices}, cf.\ \cite{jie-thesis,harris-lin},
$$sp(i,\Pi;\Pi',\imath_v):=\#\{1\leq j\leq n'\mid -a_{v,n+1-i}>b_{v,j}>-a_{v,n-i}\}$$
and
$$sp(i,\Pi;\Pi',\bar\imath_v):=\#\{1\leq j\leq n' \mid a_{v,i}>-b_{v,j}>a_{v,i+1}\}.$$ 
\noindent Here we put formally $a_{v,0}=+\infty$ and $a_{v,n+1}=-\infty$. It is easy to see that
\begin{equation}\label{sp relation}
sp(i,\Pi^{c};\Pi'^{c},\imath_v)=sp(i,\Pi;\Pi',\bar{\imath}_v)=sp(n-i,\Pi;\Pi',\imath_{v}).
\end{equation}
Similarly, for $0\leq j\leq n'$, we define $sp(j,\Pi';\Pi,\imath_v):=\#\{1\leq i\leq n\mid -b_{v,n'+1-j}>a_{v,i}>-b_{v,n'-j}\}$ and $sp(j,\Pi';\Pi,\bar\imath_v):=\#\{1\leq i\leq n\mid b_{v,j}>-a_{v,i}>b_{v,j+1}\}$.
\end{defn}

\subsection{Translating Deligne's conjecture for Rankin--Selberg $L$-functions}\label{tensorp}

We resume the notation and assumptions from the previous section and we suppose moreover that $a_{v,i}+b_{v,j}\neq 0$ for any $v\in S_\infty$, $1\leq i\leq n$ and $1\leq j\leq n'$. \\\\
Conjecturally, there are motives $M=M(\Pi)$ (resp. $M'=M(\Pi')$) over $F$ with coefficients in a finite extension $E$ of $\Q(\Pi_f)$ (resp. $E'$ of $\Q(\Pi'_f)$), satisfying $L(s,R_{\cm/\Q}(M\otimes M'))=L(s-\tfrac{n+n'-2}{2},\Pi_f\times \Pi'_f)$, which is a variant of \cite{clozel}, Conj.\ 4.5. To make sense of this statement, the right hand side of the equation must first be interpreted as a function with values in $EE'\otimes_\Q\C\cong \C^{|J_{EE'}|}$: Arguing as in \cite{grob_harris_lapid}, \S 4.3, or in \cite{clozel}, Lem.\ 4.6, one shows that at $v\notin S_\infty$, the local $L$-factor $L(s-\frac{n+n'-2}{2},\Pi_v\times \Pi'_v)=P_v(q^{-s})^{-1}$ for a polynomial $P_v(X)\in EE'[X]$, satisfying $P(0)=1$, and one deduces that $L(s-\frac{n+n'-2}{2},{}^\sigma\Pi_v\times {}^\sigma\Pi'_v)={}^\sigma P_v(q^{-s})^{-1}$, where $\sigma$ acts on $P_v$ by application to its coefficients in $EE'$. In particular, for any finite set $S$ of places of $F$ containing $S_\infty$, the family $\{\prod_{v\notin S} L(s-\frac{n+n'-2}{2},{}^\sigma\Pi_v\times {}^\sigma\Pi'_v)\}_{\sigma\in {\rm Aut}(\C)}$ only depends on the restriction of the individual $\sigma$ to $EE'$, whence we may apply Rem.\ \ref{rem:Ealg}, in order to view it as an element of $\C^{|J_{EE'}|}\cong EE'\otimes_\Q\C$. It is this way, in which we will interpret $L^S(s-\tfrac{n+n'-2}{2},\Pi\times \Pi')$ as a $|J_{EE'}|$-tuple.\\\\
In \S \ref{construction of motive}, we shall indeed construct such motives $M=M(\Pi)$ and $M'=M(\Pi')$ attached to a large family of representations $\Pi$ and $\Pi'$ (and hence verify Clozel's conjecture, \cite{clozel}, Conj.\ 4.5, for these representations). It will turn out that $M$ (resp. $M'$) is regular, pure of rank $n$ (resp. $n'$) and weight $w=n-1$ (resp. $w'=n'-1$) whose field of coefficients may be chosen to be a suitable finite extension $E$ of $E(\Pi)$, resp.\ $E'$ of $E(\Pi')$. Moreover, the above condition on the infinity type is equivalent to the condition that the $(\tfrac{w+w'}{2},\tfrac{w+w'}{2})$ Hodge component of $M\otimes M'$ is trivial. We can hence apply Prop.\ \ref{Deligne period motivic} and obtain a relation between Deligne's periods $c^\pm(R_{F/\Q}(M\otimes M'))$ and our motivic periods $Q^{(i)}(M,\imath)$ and $Q^{(j)}(M',\imath)$.\\\\
It is predicted by the Tate conjecture (see Conj.\ 2.8.3 and Cor.\ 2.8.5 of \cite{H97} and Sect.\ $4.4$ of \cite{harris-lin}), that one has the fundamental {\it Tate-relation}: 
\begin{equation}\label{eq:fundrel}
P^{(i)}(\Pi,\imath) \sim_{E} Q^{(i)}(M(\Pi),\imath).
\end{equation}
Again, as for the comparison of motivic and automorphic $L$-functions above, the left hand side of this relation should be read as an element of $E\otimes_\Q\C$ as explained in Rem.\ \ref{rem:Ealg}. Let us assume for a moment that \eqref{eq:fundrel} is valid (and that its left hand side is defined).\\\\
One easily checks that our so constructed motive $M$ (resp. $M'$) has Hodge type $(-a_{n-i}+w/2,a_{n-i}+w/2)_{1\leq i\leq n}$ (resp. $(-b_{n'-j}+w'/2,b_{n'-j}+w'/2)_{1\leq j\leq n'}$) at $\imath$. We now see immediately from Def.\ \ref{split motivic} and Def.\ \ref{split automorphic} that $sp(i,\Pi;\Pi',\imath)=sp(i,M;M',\imath)$. Whence, recalling that for any critical point $s_0\in \tfrac{n+n'}{2}+\Z$ of $L(s,\Pi\times \Pi')$, and any $v\notin S_\infty$, $L(s_0,\Pi_v\times \Pi'_v)$ is the inverse of a polynomial expression $P_v(q^{-s})\in EE'[q^{-s}]$ of an integral power $s_0-\tfrac{n+n'}{2}$ of $q$, and hence in $EE'$, and recollecting all of our previous observations, we finally deduce that Deligne's conjecture, Conj.\ \ref{conj:Deligne}, for $R_{\cm/\Q}(M\otimes M')$ may be rewritten in purely automorphic terms as follows:

\begin{conj}\label{main conjecture}
Let $\Pi$ (resp. $\Pi'$) be a cohomological conjugate self-dual cuspidal automorphic representation of $G_{n}(\Acm)$ (resp. $G_{n'}(\Acm)$), which satisfies Hyp.\ \ref{descent}. Let $s_{0}\in \Z+\tfrac{n+n'}{2}$ be a critical point of $L(s,\Pi\times \Pi')$, and let $S$ be a fixed finite set of places of $F$, containing $S_\infty$. Then, the arithmetic automorphic periods $P^{(I)}(\Pi)$ and $P^{(I)}(\Pi')$ admit a factorization as in \eqref{eq:splitting} and
\begin{equation}
L^S(s_{0},\Pi\otimes \Pi') \sim_{E(\Pi)E(\Pi')} (2\pi i)^{nn's_{0}} \prod\limits_{\imath \in \Sigma}[\prod\limits_{0\leq i\leq n}P^{(i)}(\Pi,\imath)^{sp(i,\Pi;\Pi',\imath)}\prod\limits_{0\leq j\leq n'}P^{(j)}(\Pi',\imath)^{sp(j,\Pi';\Pi,\imath)}].
\end{equation}
Interpreted as families, this relation is equivariant under action of ${\rm Aut}(\C/F^{Gal})$.
\end{conj}

\subsection{About the main goals of this paper and a remark on the strategy of proof}\label{sect:goals}

In this paper we will establish Conj.\ \ref{conj:Deligne} for the tensor products of motives $M(\Pi)\otimes M(\Pi')$ attached via Thm.\ \ref{thm:clozel} to a large family of cohomological conjugate self-dual cuspidal automorphic representations $\Pi$ and $\Pi'$ of $G_{n}(\Acm)$, resp.\ $G_{n'}(\Acm)$. To this end we will first prove Conj.\ \ref{main conjecture} for those $\Pi$ and $\Pi'$. In view of Prop.\ \ref{Deligne period motivic} this result will reduce a complete proof of Deligne's original conjecture, Conj.\ \ref{conj:Deligne}, for the motives attached to such $\Pi$ and $\Pi'$ (and with coefficients in a number field containing $F^{Gal}$) to a proof of the Tate relation \eqref{eq:fundrel}.\\\\ 
In fact, as our second main result, we will prove \eqref{eq:fundrel} by showing a refined decomposition of the local arithmetic automorphic periods $P^{(i)}(\Pi,\imath)$, which mirrors \eqref{eq:Qperiods}: Recall that the motivic periods on the right-hand-side of the Tate relation were defined as a product
$$Q^{(i)}(M(\Pi),\imath)=Q_{0}(M(\Pi),\imath)Q_{1}(M(\Pi),\imath)\cdots Q_{i}(M(\Pi),\imath).$$
As our second main result we will define factors $P_i(\Pi,\imath)$ in \S\ref{sect:thm}, which are attached to a certain (canonical) descent $\pi(i)$ of $\Pi$ to a (non-canonical) unitary group and show that, up to a scalar contained in an extension $E\supset F^{Gal}$ explicitly attached to $\Pi$ and the $\pi(i)$'s, we have
\begin{equation}\label{eq:factor}
P^{(i)}(\Pi,\imath) \sim_{E} P_{0}(\Pi,\imath)P_{1}(\Pi,\imath)\cdots P_{i}(\Pi,\imath).
\end{equation}
The $P_i$ is essentially (but not quite) the {\it automorphic $Q$-period} of  $\pi(i)$ introduced in \S \ref{Qperiod}, and the Tate relation then comes down to a rather simple comparison, established in and recorded as Thm.\ \ref{main factorization}.\\\\
It is important to notice that, when $n$ and $n'$ are of the same parity, Conj.\ \ref{main conjecture} is in fact known at the ``near central'' critical point $s_0=1$ by third named author's thesis, see Thm.\ 9.1.1.(2).(i) of \cite{jie-thesis}:

\begin{thm}\label{automorphic Deligne near central}
Let $\Pi$ (resp. $\Pi'$) be a cohomological conjugate self-dual cuspidal automorphic representation of $G_{n}(\Acm)$ (resp. $G_{n'}(\Acm)$), which satisfies Hyp.\ \ref{descent}. If $n$ and $n'$ have the same parity and if $s_{0}=1$, Conj.\ \ref{main conjecture} is true, provided that {\rm (i)} the isobaric sum $(\Pi\eta^n)\boxplus (\Pi'^{c}\eta^n)$ is $2$-regular and that {\rm (ii)} either $\Pi$ and $\Pi'$ are both $5$-regular or $\Pi$ and $\Pi'$ are both regular and satisfy Conj.\ \ref{nonvan}.
\end{thm}
Obviously, in view of Thm.\ \ref{automorphic Deligne near central}, a major obstacle that remains to be overcome in this regard, is (i) to extend Thm.\ \ref{automorphic Deligne near central} to all critical values $s_0$ and (ii) to treat the case of $n$ and $n'$ with different parity. We will solve this problem in \S\ref{central value}, see in particular Thm.\ \ref{automorphic Deligne central} and -- as a final summary -- Thm.\  \ref{automorphic Deligne general}. % \\\\
%In \S \ref{construction of motive} we shall construct motives $M=M(\Pi)$ and $M'=M(\Pi')$ attached to $\Pi$ and $\Pi'$ and so verify (our variant of) Clozel's conjecture, \cite{clozel}, Conj.\ 4.5, for these representations. See Thm.\ \ref{thm:clozel}. In course of this construction we will also encounter the representations $\pi(i)$ from above. Their automorphic $Q$-periods, see \S\ref{Qperiod}, together with the CM-periods $p(\widecheck{\xi_\Pi},\Sigma)$ attached to the central character $\xi_\Pi$ of $\Pi$ will be the final ingredient in the definition of the factors $P_i(\Pi,\imath)$. 

\section{Shimura varieties, coherent cohomology and a motive}

\subsection{Shimura varieties for unitary groups}\label{Sdata}
Let $V=V_n$ and $H=U(V)$ be as defined in \S \ref{sect:alggrp}. Let $\Ss = R_{\CC/\RR} \mathbb{G}_{m,\CC}$, so that $\Ss(\RR) = \CC^\times$, canonically.  In this paper we will use period invariants, attached to a Shimura datum $(H,Y_V)$, as in \cite[\S 2.2]{H21}. Explicitly, the base point $y_V  \in Y_V$ is given by 

\begin{equation}\label{yV} y_{V,v}(z) = \begin{pmatrix}  (z/\bar{z})I_{r_v} & 0 \\ 0 & I_{s_v} \end{pmatrix} \end{equation}
The following lemma is then obvious: We record it here in order to define parameters for automorphic vector bundles in the next sections.  

\begin{lem}\label{stab}  
Let $y \in Y_V$.   Its stabilizer $K_y=:K_{H,\infty}$ in $H_\infty$ is isomorphic to $\prod_{v \in S_\infty} U(r_v)\times U(s_v)$.
\end{lem}
%Later we will fix a base point $y \in Y_V$ and let $U_v:= K_y\cap U(V\otimes_{F,v}\CC) \cong U(r_v) \times U(s_v)$ with respect to this base point.\\\\
Unlike the Shimura varieties attached to unitary similitude groups, the Shimura variety $Sh(H,Y_V)$ attached to $(U(V),Y_V)$ parametrizes Hodge structures of weight $0$ -- the homomorphisms $y \in Y_V$ are trivial on the subgroup $\RR^\times \subset \CC^\times$ -- and are thus of abelian type but not of Hodge type. The reflex field $E(H,Y_V)$ is the subfield of $F^{Gal}$ determined as the stabilizer of the cocharacter $\kappa_V$ with $v$-component  $\kappa_{V,v}(z) = \begin{pmatrix}  zI_{r_v} & 0 \\ 0 & I_{s_v} \end{pmatrix}$.  In particular, if there is $v_0 \in S_\infty$ such that $s_{v_0} > 0$ but $s_v = 0$ for $v \in S_\infty \setminus \{v_0\}$ -- the type of unitary groups that we will be mainly interested in later -- then $E(H,Y_V)$ is the subfield $\imath_{v_0}(F) \subset \CC$. \\\\ 
We will also fix the following notation: Let $V'\subset V$ be a non-degenerate subspace of $V$ of codimension 1. We write $V$ as the orthogonal direct sum $V' \oplus V'_1$ and consider the unitary groups $H':=U(V')$ and $H'':=U(V')\times U(V'_{1})$ over $F^{+}$. Obviously, there are natural inclusions $H'\subset H''\subset H$, and a homomorphism of Shimura data
\begin{equation}\label{inclusionmap} 
(H'',Y_{V'}\times Y_{V'_1}) \hookrightarrow (H,Y_V).
\end{equation}
It is not necessarily the case that $V'_1$, as introduced above, and $V_1$ from \S \ref{sect:alggrp} are isomorphic as hermitian spaces, but the attached unitary groups $U(V'_1)$ and $U(V_1)$ are isomorphic.

 \subsection{Rational structures and (cute) coherent cohomology}\label{sect:RatCute}
\subsubsection{A characterization of cute coherent cohomology}
At each $v\in S_\infty$, we write as usual $\h_{v,\C}= \k_{H,v,\C} \oplus \p^-_v \oplus \p^+_v$ for the Harish-Chandra decomposition of the complex reductive Lie algebra $\h_{v,\C}$, and let
$$\p^+:= \oplus_v  \p^+_v, \quad\quad\p^- := \oplus_v  \p^-_v  \quad\quad{\rm and}\quad\quad \q:=\k_{H,\infty,\C}\oplus \p^-$$
so that
$$\h_{\infty,\C} = \k_{H,\infty,\C} \oplus \p^-  \oplus \p^+=\q\oplus \p^+.$$
Here $\p^+$ and $\p^-$ identify naturally with the holomorphic and anti-holomorphic tangent spaces to $Y_V$ at the point $y$, chosen in order to fix our choice of $K_y=K_{H,\infty}$. The Lie algebra $\q=\q_y$ is a complex parabolic subalgebra of $\h_{\infty,\C}$ with Levi subalgebra $\k_{H,\infty,\C}$. We let $W^\q$ be the set of attached Kostant representatives in the Weyl group of $H_\infty$, cf.\ \cite{bowa}, III.1.4. \\\\
Let  $\lambda = (\lambda_v)_{v \in S_\infty}$ be the highest weight of an irreducible finite-dimensional representation of $H_\infty$ as in \S \ref{sect:coh}. For a $w\in W^\q$, we may form the highest weight $\Lambda(w,\lambda):= w(\lambda+\rho_n)-\rho_n$, $\rho_n$ the half-sum of positive absolute roots of $H_\infty$, of a uniquely determined irreducible, finite-dimensional representation $\W_{\Lambda(w,\lambda)}$ of $K_{H,\infty}$ and we recall that its contragredient $\W^{\sf v}_{\Lambda(w,\lambda)}\cong \W_{\Lambda(w',\lambda^{\sf v})}$ is again of the above form for a uniquely determined Kostant representative $w'\in W^\q$, cf.\ \cite{bowa}, V.1.4. We will henceforth suppress the dependence of $\Lambda$ on $w$ and $\lambda$ in notation.\\\\
Recall from \cite{H90}, \S 2.1, that the representation $\W^{\sf v}_{\Lambda}$ defines an automorphic vector bundle $[\W^{\sf v}_{\Lambda}]$ on the Shimura variety $Sh(H,Y_V)$. Algebraicity of $\lambda$ implies that the canonical and sub-canonical extensions of the $H(\A_{F^+,f})$-homogeneous vector bundle $[\W^{\sf v}_{\Lambda}]$ give rise to coherent cohomology theories which are both defined over a finite extension of the reflex field, see cf.\ \cite{H90}, Prop.\ 2.8. We let $E(\Lambda)$ denote a number field over which there is such a rational structure.  (In general, there is a Brauer obstruction to realizing $\W^{\sf v}_{\Lambda}$ over the fixed field of its stabilizer in Gal$(\Qbar/\QQ)$, and we can take and fix $E(\Lambda)$ to be some finite, even abelian extension of the latter.)\\\\
Following the notation of \cite{guer-lin}, we denote by $H_!^*([\W^{\sf v}_{\Lambda}])$ the interior cohomology of $[\W^{\sf v}_{\Lambda}]$, cf.\ \cite{H90} \S (3.5.6). This is in contrast to \cite{H14,H90}, where the notation $\bar{H}$ was used.  Interior cohomology, being the image of a rational map, has a natural rational structure over  $E(\Lambda)$. It is well-known that every class in $H_!^*([\W^{\sf v}_{\Lambda}])$ is representable by square-integrable automorphic forms (\cite{H90}, Thm.\ 5.3) and that the $(\q, K_{H,\infty})$-cohomology of the space of cuspidal automorphic forms injects into $H_!^*([\W^{\sf v}_{\Lambda}])$ (\cite{H90}, Prop.\ 3.6). Let $H^*_{cute}([\W^{\sf v}_{\Lambda}]) \subseteq H_!^*([\W^{\sf v}_{\Lambda}])$ denote the subspace of classes, represented by cuspidal automorphic forms, contained in {\bf cu}spidal representations that are {\bf te}mpered at all places of $F^+$, where $H$ is unramified. Similarly, let $\CA_{cute}(H)$ be the corresponding space of cuspidal automorphic forms on $H(\A_{F^+})$, which give rise to representations which are tempered at all places of $F^+$, where $H$ is unramified. So, $H^*(\q,K_{ H,\infty},\CA_{cute}(H)\otimes \W^{\sf v}_\Lambda))\cong H^*_{cute}([\W^{\sf v}_{\Lambda}])$.

\begin{prop}\label{h,K-coh}
For a cuspidal automorphic representation $\pi$ of $H(\A_{F^+})$ the following assertions are equivalent:
\begin{enumerate}
\item $\pi \subset \CA_{cute}(H)$ and contributes non-trivially to $H^{*}_{cute}([\W^{\sf v}_{\Lambda}])$ for some $\Lambda=\Lambda(w)$, $w\in W^\q$.
\item $\pi$ is cohomological and its base change $BC(\pi)$ is an isobaric sum $\Pi=\Pi_1\boxplus...\boxplus\Pi_r$ of conjugate self-dual cuspidal automorphic representations $\Pi_i$.
\item $\pi$ is cohomological and tempered.
\end{enumerate}
If $\pi$ satisifies any of the above equivalent conditions, then $\pi_\infty$ is in the discrete series and $\pi$ occurs with multiplicity one in $L^2(H(F^+)\R_+\backslash H(\A_{F^+}))$.
\end{prop}
\begin{proof}
$(1)\Rightarrow (2)$: Let $\pi \subset \CA_{cute}(H)$ denote a cuspidal automorphic representation of $ H(\A_{F^+})$ that contributes to $H^*_{cute}([\W^{\sf v}_{\Lambda}])$. As $\Lambda^{\sf v}=\Lambda(w',\lambda^{\sf v})$ for a (unique) Kostant representative $w'\in W^\q$, it follows from reading the proof of \cite{gro-seb}, Thm.\ A.1 backwards, that there is an isomorphism of vector spaces
$$H^*(\h_\infty, K_{H,\infty}, \pi_\infty \otimes \cF^{\sf v}_\lambda) \cong H^*(\q,K_{ H,\infty},\pi_\infty\otimes \W^{\sf v}_\Lambda).$$
Therefore, $\pi_\infty$ is cohomological. Its base change $\Pi=BC(\pi)$ hence exists, cf.\ \S \ref{bc} for our conventions, and is a cohomological isobaric sum $\Pi=\Pi_1\boxplus...\boxplus\Pi_r$ of conjugate self-dual square-integrable automorphic representations $\Pi_i$ of some $G_{n_i}(\A_F)$. As $\pi \subset \CA_{cute}(H)$, $\pi_v$ is unramified and tempered outside a finite set of places $S$ of $F^+$ and hence so is $\Pi$ outside a finite set of places of $F$: Indeed, if $v\notin S$ is split in $F$, then $\Pi_v$ is tempered as noted in \S \ref{bc}. If, however, $v\notin S$ is not split, then $H(F^+_v)\cong U^*_n(F^+_v)$ is the quasisplit unitary group of rank $n$ over $F^+_v$ and so $\pi_v$ has a bounded local Arthur-parameter in the sense of \cite{mok}, Thm.\ 2.5.1. It follows that the unramfied representation $\Pi_v\cong BC(\pi)_v\cong BC(\pi_v)$ of $G_n(F_v)$ has a bounded local Langlands-parameter, whence $\Pi_v$ is tempered. Now, the argument of the proof of \cite{clozel}, Lem.\ 1.5, carries over verbatim, showing that the automorphic representation $\Pi$ must be isomorphic to an isobaric sum of unitary cuspidal automorphic representations. By the classification of isobaric sums, cf.\ \cite{jacshal2}, Thm.\ 4.4, these are nothing else than the isobaric summands $\Pi_i$ from above.

$(2)\Rightarrow (3)$: This is the contents of Rem.\ \ref{rmk:temp}.

$(3)\Rightarrow (1)$: We refer again to \cite{gro-seb}, Thm.\ A.1, which shows that a cohomological tempered cuspidal automorphic representation has non-trivial $(\q,K_{H,\infty})$-cohomology with respect to a suitable coefficient module $\W^{\sf v}_{\Lambda}$, $\Lambda=\Lambda(w)$, $w\in W^\q$, from which the assertion is obvious. 

In order to prove the last assertions of the proposition, recall from \cite{vozu}, p.\ 58 that a tempered cohomological representation of $H_\infty$, must be in the discrete series. Finally, it follows from Rem.\ 1.7.2 and Thm.\ 5.0.5 in \cite{KMSW} (and the fact that the continuous $L^2$-spectrum does not contain any automorphic forms) that every $\pi$, which satisfies the equivalent conditions of the proposition, occurs with multiplicity one in $L^2(H(F^+)\R_+\backslash H(\A_{F^+}))$.
 \end{proof}
 
 This result has several consequences. Firstly, we note
 
 \begin{prop}\label{prop:rat}
 The subspace $H^*_{cute}([\W^{\sf v}_{\Lambda}])$ of $H_!^*([\W^{\sf v}_{\Lambda}])$  is rational over $E(\Lambda)$.
\end{prop}
\begin{proof}
Let $H^*_{t}([\W^{\sf v}_{\Lambda}]) \subset H_!^*([\W^{\sf v}_{\Lambda}])$ denote the subspace of interior cohomology, which is represented by forms that are tempered at all non-archimedean places, where the ambient representation is unramified. The condition of temperedness at such a place is equivalent to the condition that the eigenvalues of Frobenius all be $q$-numbers of the same weight, hence is equivariant under Aut$(\C)$. Therefore, $H^*_{t}([\W^{\sf v}_{\Lambda}])$ is an $E(\Lambda)$-rational subspace, and it suffices to show that it coincides with $H^*_{cute}([\W^{\sf v}_{\Lambda}])$. Obviously, by the third item of Prop.\ \ref{h,K-coh},  $H^*_{cute}([\W^{\sf v}_{\Lambda}])\subseteq H^*_{t}([\W^{\sf v}_{\Lambda}])$, so we may complete the proof by showing that any square-integrable automorphic representation $\pi$ of $H(\A_{F^+})$ that contributes to $H^*_{t}([\W^{\sf v}_{\Lambda}])$ contributes to $H^*_{cute}([\W^{\sf v}_{\Lambda}])$. By \cite{clozel2}, Prop.\ 4.10, any such $\pi$ must be cuspidal. Now, the argument of the step ``$(1)\Rightarrow (2)$'' of the proof of Prop.\ \ref{h,K-coh} transfers verbatim, and we obtain that any such $\pi$ satisfies condition (2) of Prop.\ \ref{h,K-coh}. Hence, again by Prop.\ \ref{h,K-coh}, $\pi \subset \CA_{cute}(H)$, which shows the claim.
\end{proof}

\begin{rem} 
As far as we know, it has not been proved in general that the cuspidal subspace $H^*(\q,K_{ H,\infty},\CA_{cusp}(H)\otimes \W^{\sf v}_\Lambda))\cong H^*_{cusp}([\W^{\sf v}_{\Lambda}])$ of $H_!^*([\W^{\sf v}_{\Lambda}])$ is rational over $E(\Lambda)$,
but it is known for that for sufficiently general $\Lambda$ the interior cohomology is entirely cuspidal. In particular, this holds under the regularity assumptions of our main results.
\end{rem}

As another consequence of Prop.\ \ref{h,K-coh} we obtain

\begin{cor}\label{cor:uniquedegree}
For each $\W^{\sf v}_{\Lambda}$ as in Prop.\ \ref{h,K-coh}, there is a single degree $q = q(\Lambda)=\sum_{v\in S_\infty}q(\Lambda_v)$, the $q(\Lambda_v)$ being uniquely determined, such that $H^{q(\Lambda)}_{cute}([\W^{\sf v}_{\Lambda}]) \neq 0$. 
\end{cor}
\begin{proof}
Let $v\in S_\infty$. By \cite{H13}, Thm.\ 2.10, there is a unique discrete series representation $\pi_{\Lambda_v}$ of $H(F_v)$ and a unique degree $q(\Lambda_v)$ such that $H^{q(\Lambda_v)}(\q_v,K_{H,v},\pi_v\otimes \W^{\sf v}_{\Lambda_v})\neq 0$. Moreover, the latter $(\q_v,K_{H,v})$-cohomology is one-dimensional. Hence, by Prop.\ \ref{h,K-coh}, there are the following isomorphisms for the graded vector space
\begin{equation}\label{dbarcoh}
\begin{aligned}
H^{*}_{cute}([\W^{\sf v}_{\Lambda}]) &\cong \bigoplus_{\pi \subset \CA_{cute}(H)} H^{*}(\q,K_{ H,\infty},\pi_\infty\otimes \W^{\sf v}_\Lambda)\otimes \pi_f \\
&\cong \bigoplus_{\substack{\pi \subset \CA_{cute}( H)\\ \\ \pi_v \simeq \pi_{\Lambda_v}}} \bigotimes_{v\in S_\infty} H^{q(\Lambda_v)}(\q_v,K_{H,v},\pi_v\otimes \W^{\sf v}_{\Lambda_v})\otimes \pi_f \\
&\cong \bigoplus_{\substack{\pi \subset \CA_{cute}( H)\\  \pi_v \simeq \pi_{\Lambda_v}}} \pi_f.
\end{aligned}
\end{equation}
for the unique degrees $q(\Lambda_v)$. So, $H^{q}_{cute}([\W^{\sf v}_{\Lambda}]) =0$ unless $q= q(\Lambda)=\sum_{v\in S_\infty}q(\Lambda_v)$, in which case $H^{q(\Lambda)}_{cute}([\W^{\sf v}_{\Lambda}]) $ is described by \eqref{dbarcoh}.
\end{proof}

\subsubsection{The field $E(\pi)$}\label{sect:E(pi)}
Let $\pi \subset \CA_{cute}(H)$ be as in the statement of Prop.\ \ref{h,K-coh}. Recall that the $(\h_\infty, K_{H,\infty},H(\A_{F^+,f}))$-module of smooth and $K_{H,\infty}$-finite vectors in $\pi$ may be defined over a number field $E(\pi)\supseteq\Q(\pi_f)$, cf.\ \cite{H13} Cor.\ 2.13 \& Prop.\ 3.17. (Here we use that $\pi$ has multiplicity one in the $L^2$-spectrum, cf.\ Prop.\ \ref{h,K-coh}, in order to verify the assumption of \cite{H13} Prop.\ 3.17. See also the erratum to \cite{H13}.) We choose $E(\pi)$ to contain the compositum $F^{Gal}E(\Lambda)$, and refer to these rational structures as the {\it deRham-rational structures} on $\pi$. A function inside this deRham-rational structure is said to be {\it deRham-rational}. We obtain

\begin{lem}\label{contained}
$\Q(BC(\pi)_f^{\sf v})=\Q(BC(\pi)_f)\subseteq E(\pi)$.
\end{lem}
\begin{proof}
Strong multiplicity one implies that $\Q(BC(\pi)_f)=\Q(BC(\pi)^{S})$, where $S$ is any finite set of places containing $S_\infty$ and the places where $BC(\pi)$ ramifies. Hence, $\Q(BC(\pi)_f)=\Q(BC(\pi)^S)=\Q(BC(\pi)^{S,{\sf v}})\subseteq E(\pi^S)$, where the last inclusion is due to \cite{gan-ragh}, Lem.\ 9.2, the definition of base change and the definition of $E(\pi)$. Invoking strong multiplicity one once more, $\Q(BC(\pi)_f^{\sf v})=\Q(BC(\pi)_f)\subseteq E(\pi)$.
\end{proof}

\begin{lem}\label{lem:unique}
Let $\pi \subset \CA_{cute}(H)$ be an irreducible representation, which contributes non-trivially to $H^{*}_{cute}([\W^{\sf v}_{\Lambda}])$. Then, for each $\sigma\in {\rm Aut}(\C/E(\Lambda))$, there is a unique cohomological tempered cuspidal automorphic representation ${}^\sigma\!\pi$ of $H(\A_{F^+})$, such that $({}^\sigma\!\pi)_f\cong {}^\sigma\!(\pi_f)$ and which contributes non-trivially to $H^{*}_{cute}([\W^{\sf v}_{\Lambda}])$.
\end{lem}
\begin{proof}
Existence follows from Prop.\ \ref{prop:rat}, Prop.\ \ref{h,K-coh} and \eqref{dbarcoh}, while uniqueness follows from \cite{H13}, Thm.\ 2.10, in combination with multiplicity one, see again Prop.\ \ref{h,K-coh}.
\end{proof}
Our definition of $E(\pi)$ leaves us some freedom to include in it any other appropriate choice of a number field. We will specify such an additional choice right before Conj.\ \ref{lvarch}, by adding a suitable number field, constructed and denoted $E_Y(\eta)$ in \cite{H13}, p.\ 2023, to $E(\pi)$. So far, any choice (subject to the above conditions) works.

\subsection{Construction of automorphic motives}\label{construction of motive}

Let $\Pi$ be a cohomological, conjugate self-dual, cuspidal automorphic representation of $\GL_{n}(\Acm)$, which satisfies Hyp.\ \ref{descent}. Choose  $I_0=(I_{\imath})_{\imath\in \Sigma} \in \{0,1,\cdots,n\}^{|\Sigma|}$ so that $I_{v_0} = 1$, for some fixed place $v_0$, and so that $I_v = 0$ for $v \neq v_0$. Then, the unitary group $H=H_{I_0}$ has local archimedean signature $(r_{v_0},s_{v_0}) = (n-1,1)$, and the attached group $H'$ from \S \ref{Sdata} above has signature $(r'_{v_0},s'_{v_0}) = (n-2,1)$, while for $v \neq v_0$ in $S_\infty$, the signatures are $(n,0)$ (resp. $(n-1,0)$). \\\\
By Hyp.\ \ref{descent}, $\Pi^{\sf v}$ descends to a cohomological tempered cuspidal automorphic representation $\pi$ of $H(\A_{F^+})$. Morover, for any $\pi = \pi_\infty \otimes \pi_f \in \prod(H,\Pi^{\sf v})$ the representation $\tau_\infty\otimes \pi_f$ belongs to $\prod(H,\Pi^{\sf v})$, whenever $\tau_\infty$ is a discrete series representation of $H_\infty$ with the same infinitesimal character as $\pi_\infty$. By Prop.\ \ref{h,K-coh}, each such $\tau_\infty\otimes \pi_f \in \prod(H,\Pi^{\sf v})$ has multiplicity one in $L^2(H(F^+)\R_+\backslash H(\A_{F^+}))$. (The duality is not a misprint; with the usual normalization it is needed in order to obtain the Galois representation attached to the original $\Pi$, rather than $\Pi^{{\sf v}}$.)\\\\
For a given $\pi_f$ that descends $\Pi_f$, the set of $\tau_\infty$ such that $\tau_\infty\otimes \pi_f \in \prod(H,\Pi^{\sf v})$ has cardinality $n$, cf.\ \S \ref{sect:finitereps}. In other words, let $(H,Y_V)$ be the Shimura datum defined in \S \ref{Sdata} and let $Sh(H,Y_V)$ be the corresponding Shimura variety.  There is a unique irreducible finite-dimensional representation $\cF_{\lambda}=\otimes_{v\in S_\infty}\cF_{\lambda_v}$ of $H_\infty$, as in \S \ref{sect:finitereps}, such that, for all $\tau_\infty$ as above, 
\begin{equation}\label{hn-1} 
\dim H^{n-1}(\h_\infty,K_{H,\infty},\tau_\infty\otimes  \cF^{\sf v}_{\lambda}) =1
\end{equation}
Combining this with our previous observations, this implies that 
\begin{equation}\label{motivedimension}  \dim \Hom_{H(\Atrf)}(\pi_f,H^{n-1}(Sh(H,Y_V),\tilde{\cF}^{{\sf v}}_\lambda)) = n
\end{equation}
where $\tilde{\cF}^{{\sf v}}_\lambda$ is the local system on $Sh(H,Y_V)$ attached to the representation $\cF^{{\sf v}}_\lambda$. 
\\\\
The representation $\cF_\lambda$ of $H(F^+)$ is defined over a number field $E(\lambda)$, which we may assume contains the reflex field $E(H,Y_V)=\imath_{v_0}(F)$ of the Shimura variety.  Thus, the cohomology space $H^{n-1}(Sh(H,Y_V),\tilde{\cF}^{{\sf v}}_\lambda)$ has a natural $E(\lambda)$-structure, the {\it Betti} cohomological structure. Letting $\CO(\lambda)$ denote the ring of integers of $E(\lambda)$, we can find a free $\CO(\lambda)$-submodule
$\mathcal M_\lambda \subset \cF_{\lambda}$ that generates the representation, and thus we have a local system in free $\CO(\lambda)$-modules
$$\tilde{\mathcal M}_\lambda^{{\sf v}} \subset \tilde{\cF}^{{\sf v}}_\lambda$$
over $Sh(H,Y_V)$.  
For any prime number $\ell$ and any divisor $\mathfrak{l}$ of $\ell$ in $\CO(\lambda)$ we let 
$$\tilde{\cF}^{{\sf v}}_{\lambda,\mathfrak{l}}:= \tilde{\mathcal M}^{{\sf v}}_\lambda\otimes_{\CO(\lambda)} E(\lambda)_\mathfrak{l} \cong \tilde{\mathcal M}^{{\sf v}}_{\lambda,\mathfrak{l}}\otimes_{\CO(\lambda)_\mathfrak{l}} E(\lambda)_\mathfrak{l}$$
denote the corresponding $\ell$-adic \'etale sheaf.
Then we have the \'etale comparison map
\begin{equation}\label{ladic}
H^{n-1}(Sh(H,Y_V),\tilde{\cF}^{{\sf v}}_{\lambda})\otimes_{E(\lambda)} E(\lambda)_\mathfrak{l} 
 \isoarrow H^{n-1}_{\textrm{{\it \'et}}}(Sh(H,Y_V),\tilde{\cF}^{{\sf v}}_{\lambda,\mathfrak{l}}).
\end{equation}
On the other hand, for any embedding $\imath:  E(\lambda) \hookrightarrow \C$, we have the de Rham comparison 
\begin{equation}\label{dR}
H^{n-1}(Sh(H,Y_V),\tilde{\cF}^{{\sf v}}_{\lambda})\otimes_{E(\lambda),\imath} \C 
 \isoarrow H^{n-1}_{dR}(Sh(H,Y_V),\tilde{\cF}^{{\sf v}}_{\lambda,dR})\otimes_{E(\lambda),\imath} \C.
\end{equation}
Here we let $\tilde{\cF}^{{\sf v}}_{\lambda,dR}$ denote the flat vector bundle over $Sh(H,Y_V)$ attached to the local system $\tilde{\cF}^{{\sf v}}_{\lambda}$ by the Riemann-Hilbert correspondence; the $E(\lambda)$ structure on $H^{n-1}_{dR}$ is derived from the canonical model of $Sh(H,Y_V)$ over $E(H,Y_V) \subset E(\lambda)$, and the rational structure on the flat vector bundle $\tilde{\cF}^{{\sf v}}_{\lambda,dR}$.\\\\  
It is well-known (cf., \cite{H97}, Prop.\ 2.2.7) that, for any $\lambda$, the Hodge filtration on the right hand side of \eqref{dR} has an associated graded composed of $n$ spaces of interior cohomology:
\begin{equation}\label{Hodge}
gr_F^{\bullet} H^{n-1}_{dR}(Sh(H,Y_V),\tilde{\cF}^{{\sf v}}_{\lambda,dR}) \cong\bigoplus_{q = 0}^{n-1} H_!^q([\W^{\sf v}_{\Lambda(q)}]),
\end{equation}
Here $\Lambda(q) = \HC(q) - \rho_n$, where $\HC(q)$ is the Harish-Chandra parameter of $\pi_{\lambda,q}$, cf.\ \S \ref{sect:coh}. For $q = 0, \dots, n-1$, we define $i(q) \in \Z$ by
$$H_!^q([\W^{\sf v}_{\Lambda(q)}])= gr_F^{i(q)}H^{n-1}_{dR}(Sh(H,Y_V),\tilde{\cF}^{{\sf v}}_{\lambda,dR}).$$
We now recall that the group $H(\Atrf)$ acts on the spaces in \eqref{ladic}, \eqref{dR}, and \eqref{Hodge} compatibly with the comparison isomorphisms.  Let $\pi_f$ be as before. It is defined over the number field $E(\pi) \supset E(\lambda)$, as introduced in \S \ref{sect:E(pi)}, and we define
$$M_{dR}(\pi_f) := \Hom_{H(\Atrf)}(\pi_f,H_{dR}^{n-1}(Sh(H,Y_V),\tilde{\cF}^{{\sf v}}_{\lambda,dR})\otimes E(\pi)).$$
$$M_B(\pi_f):= \Hom_{H(\Atrf)}(\pi_f,H^{n-1}(Sh(H,Y_V),\tilde{\cF}^{{\sf v}}_{\lambda})\otimes E(\pi)),$$
$$M_{\mathfrak{l}}(\pi_f) := \Hom_{H(\Atrf)}(\pi_f,H^{n-1}(Sh(H,Y_V),\tilde{\cF}^{{\sf v}}_{\lambda})\otimes E(\pi)_\mathfrak{l}),$$
Here we are abusing notation: The $\pi_f$ in each $\Hom$ space above is viewed as a vector space over the appropriate coefficient field by extension of scalars, namely $E(\pi)$, in the first two, and $E(\pi)_{\mathfrak{l}}$, in the third line. Clearly all three of the spaces $M_?(\pi_f)$ have the same dimension over their respective coefficient fields.  In fact, one may check (e.g., by the results recalled in part D of the proof of Prop.\ \ref{mult1} below), that $\dim_{E(\pi)} M_B(\pi_f) = n$.  More precisely, the Hodge filtration on $H^{n-1}_{dR}(Sh(H,Y_V),\tilde{\cF}^{{\sf v}}_{\lambda,dR})$ induces a decreasing filtration $F^{i}M_{dR}(\pi_f)$ on $M_{dR}(\pi_f)$, and the isomorphism $\eqref{Hodge}$ induces an isomorphism
$$gr_F^{\bullet}M_{dR}(\pi_f) = \bigoplus_{q=0}^{n-1}gr_F^{i(q)} M_{dR}(\pi_f)  \cong \bigoplus_{q = 0}^{n-1} \Hom_{H(\Atrf)}(\pi_f,H_!^q([\W^{\sf v}_{\Lambda(q)}])), $$
and each of the spaces $M^{i(q)}_{dR}(\pi_f):= gr_F^{i(q)} M_{dR}(\pi_f)$ is of dimension $1$ over $E(\pi)$. The following result now follows from our construction

\begin{thm}\label{thm:clozel}
The collection $(M_B(\pi_f), M_{dR}(\pi_f), \{M_{\mathfrak{l}}(\pi_f)\}_{\mathfrak l})$, together with the obvious comparison maps defines a regular, pure motive $M(\Pi)$ over $E(H,Y_V)$ with coefficients in the finite extension $E(\pi)$ of $\Q(\Pi_f)$.  More precisely, the data satisfy the conditions of Definition \ref{definitionmotive}, with the exception of (i) (the infinite Frobenius); see \ref{Finfty} below.

Moreover, if $\Pi'$ is another cohomological, conjugate self-dual, cuspidal automorphic representation of $\GL_{n'}(\Acm)$, which satisfies Hyp.\ \ref{descent}, then 
$$L(s,M(\Pi)\otimes M(\Pi'))=L(s-\tfrac{n+n'-2}{2},\Pi_f\times \Pi'_f),$$
interpreted as $E(\pi)E(\pi')\otimes_\Q\C$-valued functions as in \S\ref{tensorp}.
\end{thm}  

If we recall that $E(H,Y_V)\cong F$, putting $n'=1$ in Thm.\ \ref{thm:clozel} proves \cite{clozel}, Conj.\ 4.5, for the conjugate self-dual, cuspidal automorphic representations $\Pi$ at hand.\\\\ 
We remark, however, that the motive $M(\Pi)$ depends on the choice of the place $v_0$, as so does the Shimura variety $Sh(H,Y_V)$. In view of \S\ref{sect:alggrp} and \cite{milne-suh}, Thm.\ 1.3, replacing $v_0$ by a different choice $v_1$ means to descend to the unitary group ${}^\sigma H$ underlying the $\sigma$-twisted Shimura variety ${}^\sigma Sh(H,Y_V)$, where $\sigma$ is any complex automorphism such that $\sigma^{-1}\circ \imath_{v_0} = \imath_{v_1}$. Hence, upon applying restriction of scalars, one obtains a motive $R_{E(H,Y_V)/\Q}(M(\Pi))$ over $\Q$, which is in fact independent of the choice of $v_0$. \\\\
At least if $F^+ \neq \Q$, the motive $M(\Pi)$ can be identified with a direct summand in the cohomology of a certain abelian scheme over a locally symmetric space $S'(H)$ isomorphic over $\Qbar$ to $Sh(H,Y_V)$ -- but with algebraic structure inherited from an embedding in the PEL Shimura variety attached to a similitude group containing $H$. 

\subsubsection{The automorphic version of $F_\infty$}\label{Finfty}  It is most convenient to take complex conjugation of differential forms as a surrogate for
the operator $F_{B,\imath}$ of Definition \ref{definitionmotive}.  For the reason explained in \cite[Remark 3.5]{H21}, this is not quite right.  This is why the
automorphic $Q$-periods of \S \ref{Qperiod}, which arise naturally in the calculation of $L$-functions, do not quite correspond to the motivic periods of
\S \ref{motivic periods}.  We return to this point in \S \ref{Qperiod} and in \S \ref{local period definition}.

\section{Periods for unitary groups and the Ichino-Ikeda-Neal Harris conjecture}

\subsection{GGP-periods, pairings for unitary groups, and a recent theorem}\label{sect:GGPunit}

Let $V$, $V'$, $V'_1$, $H$, $H'$, $H''$ be as in \S \ref{Sdata}. The usual Ichino-Ikeda-N. Harris conjecture considers the inclusion $H'\subset H$. However, in view of \eqref{inclusionmap} it is sometimes more convenient to consider the inclusion $H''\subset H$ instead, see \cite{H13} and \cite{H14}, and we are going to use both points of view in this paper. In this section we take the opportunity to discuss the relations of the associated periods for the two inclusions $H'\subset H$ and $H''\subset H$. We warn the reader that our notation here differs slightly from \cite{H13} and \cite{H14}.\\\\
Let $\pi$ (resp. $\pi'$) be a cohomological tempered cuspidal automorphic representation of $H(\Atr)$ (resp.\ $H'(\Atr)$). Let $\xi$ be a Hecke character on $U(V_{1})(\Atr)$ (recall that $U(V_1)$ is independent of the hermitian structure on $V_1$, \S \ref{sect:alggrp}). We write $\pi'':=\pi' \otimes \xi$, which is a tempered cuspidal automorphic representation of $H''(\Atr)$. Moreover, we fix a Haar measure $dh:=\prod_{v} dh_v$ on $H(\Atr)$, normalized as in \cite{H13}, \S 4.2, adding the (compatible) convention that $vol_{dh}(U(V_1)(F^+)\backslash U(V_1)(\A_{F^+}))=1$. This defines measures on $H'(\Atr)$ and $H''(\Atr)$ accordingly.\\\\
For  $f_{1}, f_{2}\in \pi$ the {\it Petersson inner product} on $\pi$ is defined as usual as
$$ \<f_{1},f_{2}\>:=\int_{H(\tr)Z_{H}(\Atr)\backslash H(\Atr)} f_{1}(h)\overline{f_{2}(h)} \ dh.$$
Analogously, we may define the Petersson inner product on $\pi'$ and $\pi''$. Next, for $f\in \pi$, $f'\in \pi'$ we put
$$I^{can}(f,f'):=\int\limits_{H'(\tr)\backslash H'(\Atr)} f(h')f'(h') \ dh' .$$
Then it is easy to see that $I^{can} \in \Hom_{H'(\Atr)}(\pi\otimes \pi',\C)$. With this notation the {\it GGP-period} for the pair $(\pi,\pi')$ 
(called the Gross-Prasad period in \cite{H13}) is defined as 
$$\CP(f,f'):=\cfrac{|I^{can}(f,f')|^{2}}{\<f,f\>\ \<f',f'\>}.$$
We can similarly define a $H''(\Atr)$-invariant linear form on $\pi \otimes \pi''$, which we will also denote by $I^{can}$, as
$$I^{can}(f,f''):=\int\limits_{H''(\tr)\backslash H''(\Atr)} f(h'')f''(h'') \ dh'' \text{ for } f\in \pi, f''\in \pi'',$$
leading to a definition of the {\it GGP-period} for the pair $(\pi,\pi'')$ as 
$$\CP(f,f''):=\cfrac{|I^{can}(f,f'')|^{2}}{\<f,f\>\ \<f'',f''\>}.$$
Let $\xi_\pi$, resp. let $\xi_{\pi'}$, be the central character of $\pi$, resp.\ $\pi'$. We assume that 

\begin{equation}\label{central extra U(1)}
\xi_{\pi}^{-1}=\xi_{\pi'}\xi 
\end{equation} 
(resembling equation $(\Xi)$ on page 2039 of \cite{H13}). Then, without restriction of generality, we may write a $f'' \in \pi''$, as $f''=f'\cdot \xi$, and one verifies easily that with our normalizations
$$I^{can}(f,f'')=I^{can}(f,f''|_{H'(\Atr)})=I^{can}(f,f')$$ 
and 
$$\<f'',f''\>=\<f''|_{H'(\Atr)},f''|_{H'(\Atr)}\>= \<f',f'\>.$$ We conclude that:
\begin{lem}\label{add U1}
$$\CP(f,f'')=\CP(f,f''|_{H'(\Atr)})=\CP(f,f').$$
Moreover, one gets $\Hom_{H'(\Atr)}(\pi\otimes \pi',\C) \cong \Hom_{H''(\Atr)}(\pi\otimes \pi'',\C)$.
\end{lem}
We will also need a local version of the above pairings. To this end, choose $f \in \pi$, $f' \in \pi'$, and assume they are factorizable as
$f = \otimes f_v, f' = \otimes f'_v$ with respect to the restricted tensor product factorizations
\begin{equation}\label{422} 
\pi \cong \otimes'_v \pi_v, \quad\quad \pi' \cong \otimes'_v \pi'_v.  
\end{equation}
Outside a finite set $S\supset S_\infty$ of places of $F^+$, we assume $\pi_v$ and $\pi'_v$ are unramified, and $f_v$ and $f'_v$ are the normalized
spherical vectors, i.e., the unique spherical vector taking value $1$ at the identity element.  We choose inner products $\<\cdot,\cdot\>_{\pi_v}, \<\cdot,\cdot\>_{\pi'_v}$
on each of the unitary representations $\pi_v$ and $\pi'_v$ such that at an unramified place $v$, the local normalized spherical vector in $\pi_v$ or $\pi'_v$ has norm $1$. For each place $v$ of $F^+$, let
$$c_{f_v}(h_v) := \<\pi_v(h_v) f_v,f_v\>_{\pi_v} \quad\quad\quad c_{f'_v}(h'_v) := \<\pi'_v(h'_v) f'_v,f'_v\>_{\pi'_v},  ~h_v \in H_v, h'_v \in H'_v,$$
and define
$$I_v(f_v,f'_v) := \int_{H'_v} c_{f_v}(h'_v) c_{f'_v}(h'_v) dh'_v  \quad\quad\quad I^*_v(f_v,f'_v) := \frac{I_v(f_v,f'_v)}{c_{f_v}(1)c_{f'_v}(1)}.$$
Neal Harris proves that these integrals converge since $\pi$ and $\pi'$ are locally tempered at all places.\\\\ 
The GGP-periods and local pairings are interconnected by the Ichino-Ikeda-N.Harris conjecture, which is now a theorem: In order to state it, denote the base change of the cohomological  tempered cuspidal automorphic representation  $\pi\otimes\pi'$ of  $H(\A_{F^+})\times H'(\A_{F^+})$ to $G_n(\A_{F})\times G_{n-1}(\A_{F})$ by $\Pi\otimes\Pi'$. We define
\begin{equation}\label{421} \CL^S(\Pi,\Pi'): = \frac{L^S(\tfrac{1}{2},\Pi\otimes \Pi')}{L^S(1,\Pi,{\rm As}^{(-1)^n})L^S(1,\Pi',{\rm As}^{(-1)^{n-1}})},\end{equation}
where $L^S(1,\Pi,{\rm As}^\pm)$ denotes the partial Asai $L$-function of the appropriate sign and we let
$$\Delta_{H}:= L^{S}(1, \eta) \zeta(2) L(3,\eta) ....L(n, \eta^{n}).$$
The Ichino-Ikeda-N.Harris conjecture for unitary groups is now the following theorem

\begin{thm} \label{conjecture II} 
Let $f \in \pi$, $f' \in \pi'$ be factorizable vectors as above. Then there is an integer $\beta$ (depending on the Arthur-Vogan-packets containing $\pi$ and $\pi'$), such that
$$\CP(f,f') = 2^{\beta}\Delta_{H} \CL^S(\pi,\pi') \ \prod_{v \in S} I^*_v(f_v,f'_v) .$$
\end{thm}

\begin{rem}\label{rem:IIproved}
The conjecture has been proved in increasingly general versions in \cite{zhang, xue, BP3}, and finally the proof was completed in \cite{BLZZ19,BCZ20}.  For totally definite unitary groups it was also shown in \cite{grob_lin} up to a certain algebraic number. 
\end{rem}

\begin{rem}\label{Ivstar}  
Both sides of Thm.\ \ref{conjecture II} depend on the choice of factorizable vectors $f, f'$, but the dependence is invariant under scaling.  In particular, the statement is independent of the choice of factorizations \eqref{422}, and the assertions below on the nature of the local factors $I^*_v(f_v,f'_v)$ are meaningful.
\end{rem}

The algebraicity of local terms $I^*_v$ was proved in \cite{harrisANT} when $v$ is non-archimedean.  More precisely, we have the following:

\begin{lem}\label{localIv} 
Let $v$ be a non-archimedean place of $F^+$. Let $\pi$ and $\pi'$ be cohomological tempered cuspidal automorphic representations as above. Let $E$ be a number field over which $\pi_v$ and $\pi'_v$  both have rational models.  Then for any $E$-rational vectors $f_v \in \pi_v$, $f'_v \in \pi'_v$, we have
$$I^*_v(f_v,f'_v)  \in E.$$
\end{lem}
\begin{proof}  The algebraicity of the local zeta integrals $I_v(f_v,f'_v)$ is proved in \cite{harrisANT}, Lem.\ 4.1.9, when the local inner products $\<\cdot,\cdot\>_{\pi_v}$ and  $\<\cdot,\cdot\>_{\pi'_v}$ are taken to be rational over $E$. Since $f_v$ and $f'_v$ are $E$-rational vectors, this implies the assertion for the normalized integrals $I^*_v$ as well. 
\end{proof}

We will state an analogous result for the archimedean local factors as an expectation of ours in the next section.

\subsection{Review of the results of \cite{H14}}\label{reviewH14}

Let us fix a place $v_0\in S_\infty$ and let $H=H_{I_0}$, where $I_0$ is as in \S \ref{construction of motive}. Let $\cF_{\lambda}=\otimes_{v\in S_\infty}\cF_{\lambda_v}$ be some irreducible finite-dimensional representation of $H_\infty=H_{I_0,\infty}$ as in \S \ref{sect:finitereps}. Using that $H_\infty$ is compact at $v\neq v_0$ one sees as in \S \ref{sect:coh} that there are exactly $n$ inequivalent discrete series representations of $ H_{\infty}$, denoted $\pi_{\lambda,q}$, $0\leq q \leq n-1$, for which  $H^p(\h_\infty, K_{H,\infty},\pi_{\lambda,q}\otimes\cF^{\sf v}_{\lambda}) \neq 0$ for some degree $p$ (which necessarily equals $p=n-1$). Moreover, the representations $\pi_{\lambda,q}$, $0\leq q \leq n-1$, are distinguished by the property that,
$$\dim H^q(\q,K_{H,\infty},\pi_{\lambda,q}\otimes \W^{\sf v}_{\Lambda(q)}) = 1$$
for $\Lambda(q)=A(q)-\rho_n$, where $A(q)$ denotes the Harish-Chandra parameter of $\pi_{\lambda,q}$ and all other $H^*(\q,K_{H,\infty},\pi_{\lambda,q}\otimes \W)$ vanish as $\W$ runs over all irreducible representations of $K_{H,\infty}$. We can determine $\HC(q)$ explicitly: Let
$$\HC_{\lambda} = (\HC_{\lambda,v})_{ v \in S_\infty};  ~\HC_{\lambda,v} =(\HC_{v,1} > \dots > \HC_{v,n})$$
be the infinitesimal character of $\cF_{\lambda}$, as in \S 4.2 of \cite{H14}.  Then $\HC(q) = (\HC(q)_{v})_{ v \in S_\infty}$ where $\HC(q)_{v} = \HC_{\lambda,v}$ for $v \neq v_0$ and
$$\HC(q)_{v_0} = (\HC_{v_{0},1} > \dots > \widehat{\HC_{v_{0},q+1}}> \dots  > \HC_{v_{0},n}; \HC_{v_{0},q+1})$$
(the parameter marked by $\widehat{}$ is deleted from the list).  The following is obvious:

\begin{lem}\label{knownknown}  
For $0\leq q \leq n-2$ the parameter $\HC(q)$ satisfies Hyp.\ 4.8 of \cite{H14}. For $q=0$, the representation $\pi_{\lambda,q}$ is holomorphic.
\end{lem}

Now suppose that the highest weight $\lambda_{v_0}$ is regular. Equivalently, the Harish-Chandra parameter $\HC_{\lambda,v_0}$ satisfies the regularity condition $\HC_{v_0,i} - \HC_{v_0,i+1} \geq 2$ for $i = 1, \dots, n-1$. Then, for $0\leq q \leq n-2$ define a Harish-Chandra parameter $\HC'(q) = (\HC'(q)_v)_{ v \in S_\infty}$ by the formula (4.5) of \cite{H14}:
\begin{equation}\label{lambdaprime}  
\HC'(q)_{v_0} = (\HC_{v_{0},1} - \tfrac{1}{2} > \dots > \widehat{\HC_{v_{0},q+1}-\tfrac{1}{2}} > \dots  > \HC_{v_{0},n-1}  - \tfrac{1}{2}; \HC_{v_{0},q+1} +  \tfrac{1}{2}).
\end{equation}
For $v \neq v_0$, $\HC'(q)_v =  (\HC_{v,1} - \tfrac{1}{2}  > \dots  > \HC_{v,n-1}  - \tfrac{1}{2})$.\\\\
Since $\lambda_{v_0}$ is regular, \cite{H14}, Lem.\ 4.7, shows that $\HC'(q)$ is the Harish-Chandra parameter for a unique discrete series representation $\pi_{\HC'(q)}$ of $H'_{\infty}$. Indeed, the regularity of $\lambda_{v_0}$ is the version of Hyp.\ 4.6 of \cite{H14}, where the condition is imposed only at the place $v_0$ where the local unitary group is indefinite. Observe that there is no need for a regularity condition at the definite places: For $v \neq v_0$ the parameter $A'(q)_v$ is automatically the Harish-Chandra parameter of an irreducible representation. We can thus adapt Thm.\ 4.12 of \cite{H14} to the notation of the present paper:

\begin{thm}\label{GPknownknown}  
Suppose $\lambda_{v_0}$ is regular.  For $0\leq q \leq n-2$ let $ \pi(q)= \pi(A(q))$ and $ \pi'(q)= \pi'(A'(q))$ be a tempered cuspidal automorphic representations of $ H(\A_{F^+})$ and $ H'(\A_{F^+})$, respectively, with archimedean components $\pi_{\HC(q)}$ and $\pi^{\sf v}_{\HC'(q)}$. Let $ \xi$ be the Hecke character  $(\xi_{ \pi}\cdot\xi_{ \pi'})^{-1} $ of $U(V_{1})(\A_{F^+})$ and set $\pi''(q)=\pi'(q)\otimes\xi$. Then for any deRham-rational elements $f\in \pi(q)$, $f''\in \pi''(q)$
$$I^{can}(f,f'')\in E(\pi(q))  E(\pi''(q))=E(\pi(q))  E(\pi'(q)).$$ 
\end{thm}
The statement in \cite{H14} has two hypotheses:  the first one is the regularity of the highest weight, while the second one (Hyp.\ $4.8$ of \cite{H14}) follows as in Lem.\ \ref{knownknown} from the assumption that $q \neq n-1$. We remark that the assumption in  \textit{loc.cit} on the Gan-Gross-Prasad multiplicity one conjecture for real unitary groups has been proved by He in \cite{he}.\\\\
A cuspidal automorphic representation $ \pi$ of $ H(\A_{F^+})$ that satisfies the hypotheses of Thm.\ \ref{GPknownknown} contributes to interior cohomology of the corresponding Shimura variety $Sh(H, Y_V)$ with coefficients in the local system defined by the representation $\W^{\sf v}_{\Lambda(q)}$. This cohomology carries a (pure) Hodge structure of weight $n-1$, with Hodge types corresponding to the infinitesimal character of $\W^{\sf v}_{\Lambda(q)}$, which is given by 
$$A(q)^{{\sf v}}= (-A(q)_{v,n}>...> -A(q)_{v,1})_{v\in S_\infty}.$$
The Hodge numbers corresponding to the place $v_0$ are
\begin{equation}\label{hodge1} 
(p_i = -\HC_{v_0,n+1 -i}+ \tfrac{n-1}{2}, q_i = n-1-p_i);  \quad\quad (p_i^c = q_i, q_i^c = p_i). 
\end{equation}
Analogously, a cuspidal automorphic representation $\pi'$ of $ H'(\A_{F^+})$ as in Thm.\ \ref{GPknownknown} contributes to interior cohomology of the corresponding Shimura variety $Sh(H',Y_{V'})$  with coefficients in the local system defined by a representation $\W_{\Lambda'(q)}$, whose parameters are obtained from those of $\HC'(q)$ by placing them in decreasing order and substracting $\rho_{n-1}$.  In particular, it follows from \eqref{lambdaprime} that the infinitesimal character of $\W_{\Lambda'(q)}$ at $v_0$ is given by
$$(\HC_{v_0,1} - \tfrac{1}{2} > \dots > \HC_{v_0,q} - \tfrac{1}{2} > \HC_{v_0,q+1} + \tfrac{1}{2} > \HC_{v_0,q+2} - \tfrac{1}{2} > \dots > \HC_{v_0,n-1}  - \tfrac{1}{2}),$$
with strict inequalities due to the regularity of $A_{\lambda,v_0}$ and with corresponding Hodge numbers
\begin{equation}\label{hodge2} p_i' = \HC_{v_0,i} + \tfrac{n-3}{2}, \quad {\rm for}\quad i \neq q+1; \quad\quad p'_{q+1} = \HC_{v_0,q+1} + \tfrac{n-1}{2}
 \end{equation}
and $q'_i = n-2 - p'_i$, etc. Here is a consequence of the main result of \cite{H14}.

\begin{thm}\label{BS}  
Let ${\pi}(q)$ be as above. Then there exists a tempered cuspidal automorphic representation $\pi'(q)$ of $H'(\A_{F^+})$ with archimedean component $\pi_{\HC'(q)}^{{\sf v}}$, such that 
\begin{enumerate}
%\item $\pi'(q)$ extends to a cuspidal automorphic representation $\pi'(q)$ of $H'(\A_\Q)$ which satisfies the assumptions of Thm.\  \ref{GPknownknown} 
\item\label{BS1} $BC(\pi'(q))$ is cuspidal automorphic and supercuspidal at a non-archimedean place of $\tr$ which is split in $F$, and
\item\label{BS2} there are factorizable cuspidal automorphic forms $f \in \pi(q)$, $f' \in \pi'(q)$, so that $I^{can}(f,f') \neq 0$ with $f_v$ (resp.\ $f'_v$) in the minimal $K_{H,v}$- (resp.\ $K_{H',v}$-type) of $\pi(q)_v$  (resp.\ $\pi'(q)_v$) for all $v \in S_{\infty}$. 
\end{enumerate}
In particular, the GGP-period $\CP(f,f')$ does not vanish.
\end{thm}

\begin{proof}  
Although this is effectively the main result of \cite{H14}, it is unfortunately nowhere stated in that paper. So, let us explain why this is a consequence of the results proved there: First, we claim that the discrete series representation $\pi_{\HC'(q)}$ is isolated in the (classical) automorphic spectrum of $H'_\infty$, in the sense of \cite{BS}, see Cor.\ 1.3 of \cite{H14}. Admitting the claim, we note that Hyp.\ 4.6 of \cite{H14} is our regularity hypothesis, and Hyp.\ 4.8 is true by construction. The theorem then follows from the discussion following the proof of Thm.\ 4.12 of \cite{H14}. More precisely,  because $\pi_{\HC'(q)}$ is isolated in the automorphic spectrum, we can apply Cor.\ 1.3 (b) of \cite{H14}, which is a restatement of the main result of \cite{BS}. Finally, the isolation follows as in \cite{HLi}, Thm.\ 7.2.1, using the existence of base change from $H'(\A_{F^+})$ to $G_{n-1}(\A_F)$, as it was established for tempered cuspidal automorphic representations in \cite{KMSW}, Thm.\ 5.0.5.
\end{proof}

Recall the number field $E_Y(\eta)$ from \cite{H13}, p.\ 2023. It has been shown in Cor.\ 3.8 of \cite{H13} (and its correction in the Erratum to that paper) that the underlying Harish-Chandra modules of the discrete series representations $\pi(q)_v$ and $\pi(q)'_v$, $v\in S_\infty$, are defined over this number field $E_Y(\eta)$. {\it From now on, we will assume that the number field $E(\pi)$, defined for a cohomological tempered cuspidal representation of a unitary group over $F^+$ in \S \ref{sect:E(pi)}, contains $E_Y(\eta)$.} One sees that the cuspidal automorphic forms $f \in \pi(q)$, $f' \in \pi'(q)$ from Thm.\ \ref{BS}.(2) can be chosen so that, for all $v\in S_\infty$, $f_v$ (resp. $f'_v$) belongs to the $E(\pi(q))$- (resp. $E(\pi'(q))$-) rational subspace of the minimal $K_{H,v}$-type of $\pi(q)_v$ (resp. $K_{H',v}$-type of $\pi'(q)_v$), with respect to the $E(\pi(q))$- (resp. $E(\pi'(q))$-) deRham-rational structure defined in \S \ref{sect:E(pi)}. The following statement is then a conjectural, archimedean analog of Lem.\ \ref{localIv}:

\begin{conj}\label{lvarch}  
Let $f$ and $f'$ be as in Thm.\ \ref{BS} and assume that they are chosen so that, for all $v\in S_\infty$, $f_v$ (resp. $f'_v$) belongs to the $E(\pi(q))$- (resp. $E(\pi'(q))$-) rational subspace of the minimal $K_{H,v}$-type of $\pi(q)_v$ (resp. $K_{H',v}$-type of $\pi'(q)_v$). Then,
$$I^*_v(f_v,f'_v) \in (E(\pi(q))\cdot E(\pi'(q)))$$
for all $v\in S_\infty$.

\end{conj}

\subsection{Automorphic $Q$-periods}\label{Qperiod}
In order to prove the factorization of the local arithmetic automorphic periods, see \eqref{eq:factor}, we will need one last ingredient, namely {\it automorphic $Q$-periods}. To define them, let $\pi$ be a cohomological, tempered, cuspidal automorphic representation of $H(\A_{F^+})$ and let $\varphi_\pi\in \pi$ be deRham-rational, cf.\ \S \ref{sect:E(pi)}. For each $\sigma\in {\rm Aut}(\C/E(\Lambda))$, we choose a ${}^{\sigma}\! \varphi_\pi$, which generates the deRham-rational structure of the unique twist ${}^\sigma\!\pi$, cf.\ Lem.\ \ref{lem:unique}. We define
$$Q({}^{\sigma}\!\varphi_{\pi}):=\<{}^{\sigma}\!\varphi_{\pi},{}^{\sigma}\!\varphi_{\pi}\>.$$ 
By Lem.\ 3.19 and Lem.\ 3.20 of \cite{H13}, $Q({}^{\sigma}\!\varphi_{\pi})$ is well-defined up to multiplication by elements of the form $\sigma(t)$ with $t\in E(\pi)^{\times}$ and (within the respective quotient of algebras) independent of the choice of ${}^{\sigma}\!\varphi_{\pi}$. Therefore, the family of numbers $Q({}^{\sigma}\!\varphi_{\pi})$ gives rise to an element
$$Q(\pi)\in E(\pi)\otimes_\Q\C\cong \C^{|J_{E(\pi)}|},$$
called the {\it automorphic $Q$-period} attached to $\pi$.   

Since the Petersson inner product appears in the definition of the GGP-period it is convenient to use it to define the automorphic $Q$-periods.  However, it has already been mentioned that the complex conjugation used to define the Petersson inner product does not quite correspond to the operator $F_\infty$ used to define motivic $Q$-periods.  Thus $Q(\pi)$ differs from the $Q$-period attached to the motive $M(\Pi)$
by a factor corresponding to the central character of $\Pi$; this explains the normalization in \ref{local period definition}.

\section{Proof of the automorphic variant of Deligne's conjecture}\label{central value}

\subsection{Existence of certain endoscopic representations}\label{sect:endo}
In view of Thm.\ \ref{automorphic Deligne near central} we first concentrate on the case when $n$ and $n'$ have different parity. So, let $n$ be an odd integer and $n'$ be an even integer, and let $\Pi$ and $\Pi'$ denote cohomological conjugate self-dual cuspidal automorphic representations of $G_n(\A_F)$ and $G_{n'}(\A_F)$, respectively. Then there are non-trivial cohomological $L$-packets $\prod(U^*_n,\Pi)$ and $\prod(U^*_{n'},\Pi')$, cf.\ Rem.\ \ref{rmk:des}\\\\
We fix an odd integer $N \geq \max\{n,n'\}$. Let $\chi_i$, $1\leq i\leq N-n$, and $\chi'_j$, $1\leq j\leq N-1-n'$, denote conjugate self-dual algebraic Hecke characters of $\GL_1(\A_F)$. They are base changes from algebraic Hecke characters of $U^*_1(\A_{F^+})$. The isobaric sums
$$\tau:=\Pi \boxplus \chi_1\boxplus \cdots \boxplus \chi_{N-n}$$
on $G_N(\A_F)$, respectively,
$$\tau':=\Pi' \boxplus (\chi'_1\eta^{-1})\boxplus \cdots \boxplus (\chi'_{N-1-n'}\eta^{-1})$$ 
on $G_{N-1}(\A_F)$, are always algebraic and in fact by the unitarity of the inducing datum always fully induced. If at every $v\in S_\infty$, the exponents $a_{v,i}$ of $\Pi$, $1\leq i\leq n$, (resp. $b_{v,j}$ of $\Pi'$, $1\leq j\leq n'$) (see Sect.\ \ref{sect:coh}) and the corresponding exponents of each character $\chi_{i,v}$,  $1\leq i\leq N-n$, (resp. $\chi'_{j,v}\eta_v^{-1}$, $1\leq j\leq N-1-n'$) are different, then the isobaric sum is cohomological: In order to see this, one puts them in decreasing order and solves \eqref{hw and it} to obtain a highest weight $\mu_v$ (resp.\ $\mu'_v$).

\begin{prop}\label{mult1}  
Let $\Pi$, $\Pi'$, $\chi_i$ and $\chi'_j$ be as above and assume that $\tau$ and $\tau'$ are cohomological. Fix some finite set of places $S$ of $F$ containing $S_\infty$. We suppose the central critical value 
\begin{equation}\label{nonvanhyp} 
L^S(\tfrac{1}{2}, \tau \times \tau') \neq 0.
\end{equation}

Then, there exists a pair of non-degenerate $c$-hermitian spaces $V \supset V'$ over $F$, with $\dim V = N$, $\dim V' = N-1$, and a cohomological tempered cuspidal automorphic representation
$\sigma \otimes \sigma'$ of $U(V)(\A_{F^+})\times U(V')(\A_{F^+})$, such that
\begin{itemize}
\item[(1)]  $BC(\sigma) \otimes BC(\sigma')=\tau \otimes \tau'$ and
\item[(2)]  $\dim {\Hom}_{U(V')(\A_{F^+})}(\sigma \otimes \sigma', \CC) =1$, where $U(V')(\A_{F^+})$ is diagonally embedded into $U(V)(\A_{F^+})\times U(V')(\A_{F^+})$,
\end{itemize}
Integration over the diagonal gives a basis of ${\rm Hom}_{U(V')(\A_{F^+})}(\sigma \otimes \sigma', \CC)$. The pairs $(V,V')$ and $(\sigma,\sigma')$ are unique with respect to these properties.
\end{prop}

\begin{proof}   
We proceed in several steps:

{\bf A.} By construction $\tau$ defines a generic global Arthur-parameter $\phi$, as in \cite{KMSW}, \S 1.3.4. As the local components $\tau_v$ are tempered at all places $v$ of $F$, their attached local Langlands-parameters are all bounded, whence so are the local Arthur-parameters $\phi_v$, constructed in \cite{KMSW}, Prop.\ 1.3.3, at all $v$ of $F^+$. Following (the unconditional) \cite{KMSW}, Thm.\ 1.6.1.(5) and (6), to each local component $\tau_v$ of $\tau$ we may associate a local Vogan $L$-packet $[\tau_v]$, which consists by definition of equivalence classes of pairs $(V_v,\sigma_v)$ where $V_v$ is an $N$-dimensional hermitian space over $F_v$ and $\sigma_v$ is a tempered irreducible admissible representation of $U(V_v)$, and a component group
$$A(\tau_v) \isoarrow (\Z/2\Z)^{r_v},$$
where $r_v \geq N-n+1$ is the number of irreducible pieces in the local Langlands-parameter of the descent of $\tau_v$ to $U_N^*(F_v)$. If $v$ is non-archimedean the members of $[\tau_v]$ are in one-to-one correspondence with characters of $A(\tau_v)$; if $v$ is archimedean, the signature of $V_v$ is an additional invariant and the members of $[\tau_v]$ are cohomological, cf.\ \cite{clozelihes}, Lem.\ 3.8 and Lem.\ 3.9. We denote the character group $\hat{A}(\tau_v)$ and let  $\chi_{(V_v,\sigma_v)} \in \hat{A}(\tau_v)$ denote the character corresponding to the indicated pair. 

Analogously, attached to each local component $\tau'_v$ of $\tau'$ there is a local Vogan $L$-packet $[\tau'_v]$, consisting of equivalence classes of pairs $(V'_v,\sigma'_v)$, a component group $A(\tau'_v)$, and characters
$\chi_{(V'_v,\sigma'_v)} \in \hat{A}(\tau'_v)$.

For almost all $v$, $[\tau_v]$ contains a pair $(U^*_N(F_v),\sigma^*_v)$ where $U^*_N$ -- the quasi-split inner form, as before -- is unramified, and $\sigma^*_v$ is a spherical representation.   We call this the {\it unramified member} of $[\tau_v]$.  For such $v$, the character group $\hat{A}(\tau_v)$ is normalized so that the unramified pair $(U^*_N(F_v),\sigma^*_v)$ corresponds to the trivial character.

{\bf B.}   Any $a \in A(\tau_v)$ can be viewed as an involution in $\GL_N$, and thus has eigenvalues $+1$ and $-1$.  We let $N[a]$ denote the dimension of its fixed subspace.  As in \cite[\S 6]{GGP1}, $a$ determines an admissible irreducible tempered representation $\tau_v[a]$ of $\GL_{N[a]}(F_v)$ (see 
also the discussion around \eqref{signchar} below).  
Similarly, $a' \in A(\tau'_v)$ is an involution in $\GL_{N-1}$ and determines an admissible irreducible tempered representation $\tau'_v[a']$ of $\GL_{N[a']}(F_v)$.

{\bf C.}  Now the assumption \eqref{nonvanhyp} implies in particular that none of the factors in $L(s,\tau \times \tau')$ vanishes at $s = \tfrac{1}{2}$.  It follows in particular that the global sign of the functional equation is $+1$ for any product of factors of $L(s,\tau \times \tau')$.  In particular, for any $a \in A(\tau_v)$, $a' \in A(\tau'_v)$, we have
\begin{equation}\label{pointC}\begin{split}  \varepsilon(\tfrac{1}{2},\tau[a]\otimes \tau') =  \prod_v\varepsilon(\tfrac{1}{2},\tau_v[a]\otimes \tau'_v) = + 1; \\
 \varepsilon(\tfrac{1}{2},\tau\otimes \tau'[a']) =  \prod_v\varepsilon(\tfrac{1}{2},\tau_v\otimes \tau'_v[a']) = + 1.
\end{split}
\end{equation}

{\bf D.}  Arthur's multiplicity formula asserts the following: for each place $v$ of $F^+$ choose a pair $(\tilde{V}_v,\tilde{\sigma}_v) \in [\tau_v]$ and assume it is the unramified member for almost all $v$. Then there is a global $c$-hermitian space $V$ of dimension $N$ over $F$ and a square-integrable automorphic representation $\sigma$ of $U(V)(\A_{F^+})$, with $V_v \ira \tilde{V}_v$ and $\sigma_v \ira \tilde{\sigma}_v$ for all $v$ (hence, $\sigma$ is tempered everywhere; and so cuspidal), if and only if
$$\prod_v \chi_{(\tilde{V}_v,\tilde{\sigma}_v)} = 1.$$
Moreover, if the sign condition is satisfied, then $\sigma$ has multiplicity one in the cuspidal spectrum of $U(V)$. This formula has been proved by Mok for the quasi-split unitary group $U^*_N$ and, recalling from Rem.\ \ref{rmk:temp} that the generic global Arthur parameter $\phi=\tau$ is also elliptic, it follows from (the unconditional part of) Thm.\ 5.0.5 of \cite{KMSW} in the general case. Obviously, mutatis mutandis, there is a similar formula for $\tau'$.

{\bf E.}  The Gan-Gross-Prasad conjecture asserts that for each $v$, there exists a unique quadruple $(V_v,\sigma_v;V'_v,\sigma'_v) \in [\tau_v]\times [\tau'_v]$ such that $V'_v$ embeds in $V_v$ as a non-degenerate $c$-hermitian subspace, and such that
$$\Hom_{U(V'_v)}(\sigma_v\otimes \sigma'_v,\CC) \neq 0.$$
Moreover, $\dim \Hom_{U(V'_v)}(\sigma_v\otimes \sigma'_v,\CC) = 1$ for this quadruple.  Finally, the characters $\chi_{V_v,\sigma_v}$ and $\chi_{V'_v,\sigma'_v}$ are determined by the explicit formulas, cf.\ Thm.\ 6.2 and \S 17 of \cite{GGP1}:
\begin{equation}\label{signchar} \begin{split}
 \chi_{V_v,\sigma_v}(a) = \varepsilon(\tfrac{1}{2},\tau_v[a]\otimes \tau'_v)\cdot \xi_{\tau_v[a]}(-1)^{N-1}\cdot \xi_{\tau_v'}(-1)^{N[a]}; \\
  \chi_{V'_v,\sigma'_v}(a') = \varepsilon(\tfrac{1}{2},\tau_v\otimes \tau_v'[a'])\cdot \xi_{\tau_v}(-1)^{N[a']}\cdot \xi_{\tau_v'[a']}(-1)^{N}.
 \end{split}
\end{equation}
Here $\xi_{\tau_v[a]}$, $\xi_{\tau_v'}$, etc. denote the respective central characters.

The Gan-Gross-Prasad conjecture has been proved for unitary groups over local fields by Rapha\"el Beuzart-Plessis \cite{BP1}, including the conjecture for tempered representations of real groups. We remark that for discrete series representations, in any case, an explicit formula for the signs we need can be extracted, with some difficulty, from \S 2 of \cite{GGP2}, and this suffices for our application to the case where the unitary groups are totally definite.

{\bf F.} It follows from points D and E that the global datum $(V,\sigma)$ exists, provided
\begin{equation}\label{producta} \prod_v \varepsilon(\tfrac{1}{2},\tau_v[a]\otimes \tau'_v) = +1.\end{equation}
Here we could ignore the central characters because the global central characters are trivial on principal ad\`eles of $G_N$. However, the product in \eqref{producta} is just the global sign, which equals $+1$ by \eqref{pointC}.

{\bf G.} So, we have verified everything except that ${\rm Hom}_{U(V')(\A_{F^+})}(\sigma \otimes \sigma', \CC)$ is generated by integration over the diagonal.
But this is precisely the content of the Ichino--Ikeda--Neal-Harris conjecture, proved in \cite{BCZ20}, see also Thm.\ \ref{conjecture II}.

\end{proof}

\subsection{The ``piano position''}
Let $\tau \otimes \tau'$ be an isobaric automorphic representation of $G_N(\A_F) \times G_{N-1}(\A_F)$, as in Prop.\ \ref{mult1} and let us denote by $\mathcal{F}_{\lambda}$ (resp.\ by $\mathcal{F}_{\lambda'}$) the coefficient module, with respect to which $\sigma_\infty$ (resp.\ $\sigma'_\infty$) is cohomological. 

\begin{defn}  
We say $\tau \otimes \tau'$ is {\it in piano position}\footnote{The terminology comes from the image of the relative position of the weights of the descents $\sigma$ and $\sigma$  in the branching formula \eqref{branching law}, which remind the authors of a keyboard in which the parameters of the highest weight of $\sigma$ are the white keys and those of the highest weight of $\sigma'$ are the black keys.  This only works on a piano in which every successive pair of white keys is separated by a black key.  The authors do not know whether such a piano can be purchased, so we recommend the reader, who has a distinguished sense for conventionalism, to extrapolate from the series of keys from the notes F to B (which corresponds to the case $\GL_4 \times \GL_3$) using her/his imagination.}, if 
$$\Hom_{U(V')_{\infty}} (\mathcal{F}_{\lambda}\otimes \mathcal{F}_{\lambda'},\CC) \neq 0,$$
or, equivalent,
\begin{equation}\label{branching law}
\lambda_{v,1}\geq -\lambda'_{v,N-1}\geq \lambda_{v,2}\geq -\lambda'_{v,N-2}\cdots \geq -\lambda'_{v,1}\geq  \lambda_{v,N} 
\end{equation}
for each $v\in S_\infty$, cf.\ \cite{goodman-wall}, Thm.\ 8.1.1. 
\end{defn}

\begin{lem}\label{lem:piano}
Let $\tau \otimes \tau'$ be an isobaric automorphic representation of $G_N(\A_F) \times G_{N-1}(\A_F)$, as in Prop.\ \ref{mult1}. Then $\tau \otimes \tau'$ is in piano position if and only if the groups $U(V)$ and $U(V')$ in Prop.\ \ref{mult1} are both totally definite.
\end{lem}
\begin{proof}
It is easy to see that if the groups $U(V)$ and $U(V')$ in Prop.\ \ref{mult1} are both totally definite, $\Hom_{U(V')_{\infty}} (\sigma_\infty\otimes \sigma'_\infty,\CC) \cong \Hom_{U(V')_{\infty}} (\mathcal{F}_{\lambda}\otimes \mathcal{F}_{\lambda'},\CC)$, whence $\tau \otimes \tau'$ is in piano position by Prop.\ \ref{mult1}.(2). For the opposite direction, observe that $BC(\mathcal{F}^{\sf v}_{\lambda})\otimes BC(\mathcal{F}^{\sf v}_{\lambda'})\cong \tau_\infty\otimes\tau'_\infty$, where we view $\mathcal{F}^{\sf v}_{\lambda}$ (resp.\ $\mathcal{F}^{\sf v}_{\lambda'}$) as an irreducible tempered representation of the compact Lie group $U(N)^d$ (resp. $U(N-1)^d$). However, by construction also $BC(\sigma_\infty)\otimes BC(\sigma'_\infty)\cong \tau_\infty\otimes\tau'_\infty$ and $\Hom_{U(V')_\infty}(\sigma_\infty\otimes\sigma'_\infty,\C)\neq 0$. So, if $\tau \otimes \tau'$ is in piano position, i.e., if 
$$0\neq\Hom_{U(V')_{\infty}} (\mathcal{F}_{\lambda}\otimes \mathcal{F}_{\lambda'},\CC)\cong \Hom_{U(N-1)^d} (\mathcal{F}^{\sf v}_{\lambda}\otimes \mathcal{F}^{\sf v}_{\lambda'},\CC),$$ 
then by the archimedean uniqeness-results of \cite{BP2}, $U(V)_\infty\times U(V')_\infty\cong U(N)^d\times U(N-1)^d$ and $\sigma_\infty\otimes\sigma'_\infty\cong\mathcal{F}^{\sf v}_{\lambda}\otimes  \mathcal{F}^{\sf v}_{\lambda'}$.
\end{proof}

We may translate equation \eqref{branching law} in terms of the infinity type by equation \eqref{hw and it}. More precisely, we have

\begin{lem}[Branching condition]\label{branching condition}
Suppose $\tau$ has infinity type $\{z^{A_{v,i}} \bar{z}^{-A_{v,i}}\}_{1\leq i\leq N}$ and $\tau'$ has infinity type $\{z^{B_{v,j}} \bar{z}^{-B_{v,j}}\}_{1\leq i\leq N}$ at $v\in S_{\infty}$ with $A_{v,1}> A_{v,2} \cdots > A_{v,N}$ and $B_{v,1}> B_{v,2} \cdots > B_{v,N-1}$ respectively.  Then $\tau\otimes \tau'$ is in piano position if and only if $$A_{v,1}>-B_{v,N-1}>A_{v,2}>-B_{v,N-2}\cdots >-B_{v,1}>A_{v,N}.$$
\end{lem}

The following result is a direct generalisation of $(4.1.6.1)$ and $(4.1.6.2)$ of \cite{harrisANT} to the case of arbitrary totally real base fields $F^+$. We recall that when the unitary groups are totally definite, the period $\CP(f,f')$ is algebraic for certain choice of $f$ and $f'$ (cf.\  Cor.\ $2.5.5$ of \cite{harrisANT}).

\begin{prop}\label{good position central value}
Let $\tau \otimes \tau'$ be an isobaric automorphic representation of $G_N(\A_F) \times G_{N-1}(\A_F)$, as in Prop.\ \ref{mult1} and assume that $\tau\otimes \tau'$ is in piano position (so, in particular, $U(V)$ and $U(V')$ from Prop.\ \ref{mult1} are totally definite, see Lem.\ \ref{lem:piano} above). Then,
$$L^S(\tfrac{1}{2},\tau\times \tau') \sim_{E(\tau)E(\tau')} (2\pi i)^{-\tfrac{1}{2}dN(N+1)}L^S(1,\tau,{\rm As}^{(-1)^{N}})L^S(1,\tau',{\rm As}^{(-1)^{N-1}}).$$
Interpreted as families, this relation is equivariant under the action of ${\rm Aut}(\C/F^{Gal})$.
\end{prop}

\begin{rem}  In \cite{harrisANT} the corresponding result is proved when $F$ is imaginary quadratic, assuming the Ichino-Ikeda-N. Harris conjecture, which is now  Thm.\ \ref{conjecture II}.   When $\tau$ is cuspidal along with some mild technical conditions, Prop.\ \ref{good position central value} was proved in \cite{grob_lin} without using Thm.\ \ref{conjecture II}.
\end{rem}

\begin{prop}\label{existence of chi}  
Let $\Pi$ and $\Pi'$ cohomological conjugate self-dual cuspidal automorphic representations of $G_n(\A_F)$ and $G_{n'}(\A_F)$, respectively, and suppose that $L^S(\tfrac{1}{2}, \Pi \times \Pi') \neq 0$.  We assume that Conj.\ \ref{nonvan} holds.  Then, there exist an odd integer $N$ and characters $\chi_k$, $\chi'_l$, $1\leq k\leq N-n$, $1\leq l\leq N-1-n'$ so that $\tau$ and $\tau'$ satisfy the conditions of Prop.\ \ref{mult1}. Moreover, the product $\tau \otimes \tau'$ is in piano position. 
\end{prop}
\begin{proof}  
Let $\{z^{a_{v,i}} \bar{z}^{-a_{v,i}}\}_{1\leq i\leq n}$ (resp. $\{z^{b_{v,j}} \bar{z}^{-b_{v,j}}\}_{1\leq j\leq m}$) be the infinity type of $\Pi$ (resp. $\Pi'$) at $v\in S_{\infty}$. In particular, $a_{v,i}\in \Z$ and $b_{v,j}\in \Z+\tfrac{1}{2}$ for each $v$, $i$ and $j$. Let $C,C'\in \Z$ such that $C\leq t\leq C'$ when $t$ runs over $a_{v,i}$ and $b_{v,j}$. We may assume that $C'-C$ is even. We let $N:=C'-C+1$.

For each $v$, the set $\{C,C+1,\cdots, C' \} \backslash \{a_{v,i}\}_{1\leq i\leq n}$ has exactly $N-n$ integers and is denoted by $\{x_{k,v}\}_{1\leq k\leq N-n}$. Similarly, the set $\{C+\tfrac{1}{2},C+\frac{3}{2},\cdots, C'-\tfrac{1}{2}\}\backslash \{b_{v,j}\}_{1\leq j\leq n'}$ has exactly $N-n'-1$ half integers and is denoted by $\{y_{l,v}\}_{1\leq l\leq N-n'-1}$. 

For $1\leq k\leq N-n$ (resp. $1\leq l\leq N-1-n'$), let $\chi_{k}$ (resp. $\chi'_{l}$) be a Hecke character of $\GL_1(\Acm)$ with infinity type $z^{x_{k,v}}z^{-x_{k,v}}$ (resp. $z^{y_{l,v}-\tfrac{1}{2}}z^{-y_{l,v}+\tfrac{1}{2}}$) at $v$ (such Hecke characters exist by \cite{weil}). Let $\tau$ and $\tau'$ be as defined in \S \ref{sect:endo}. Then $\tau$ (resp. $\tau'$) has infinity type $\{z^{t}\overline{z}^{-t}\}_{C\leq t\leq C'}$ (resp. $\{z^{s+1/2}\overline{z}^{-s-1/2}\}_{C\leq s \leq C'-1}$) at each $v$. In particular, $\tau \otimes \tau'$ is cohomological. Our Conj.\ \ref{nonvan} now allows us to choose the $\chi_k$ and $\chi'_l$ so that the various central values do not vanish. Note that we also require that $L^S(\tfrac{1}{2},\chi_k\cdot \chi'_l) \neq 0$ for all $k, l$, but this is easy to arrange by varying the $\alpha_v$ in the statement of Conj.\ \ref{nonvan}. Hence, $L^S(\tfrac{1}{2}, \tau \times \tau') \neq 0$ and so $\tau$ and $\tau'$ satisfy the conditions of Prop.\ \ref{mult1}. Finally, $\tau \otimes \tau'$ is in piano position by Lem.\ \ref{branching condition}.
\end{proof}

\subsection{A proof of the automorphic variant of Deligne's conjecture at the central critical value: The case $n\nequiv n' \mod 2$}\label{sect:diffparity}
Complementary to Thm.\ \ref{automorphic Deligne near central} in this section we will finally prove Conj.\ \ref{main conjecture} for a large family for automorphic representations $\Pi$ and $\Pi'$ in the case when $n$ and $n'$ are of the same parity and when $s_0$ denotes the central critcial value, $s_0=\tfrac{1}{2}$. To this end we need to recall two more results from the literature:

\begin{thm}[\cite{H97}, Thm.\ 3.5.13; \cite{guer-lin}, Thm.\ 4.3.3, \cite{H21}, Thm.\ 7.1] \label{n*1}
Let $\Pi$ be a cohomological conjugate self-dual cuspidal automorphic representation of $G_n(\A_F)$, which satisfies Hyp.\ \ref{descent}. Let $\{ z^{a_{v,i}}\bar{z}^{-a_{v,i}}\}_{1\leq i\leq n}$ be its infinity type at $v\in S_\infty$. Let $\chi$ be an algebraic Hecke character of $\GL_1(\A_F)$ with infinity type $z^{a_{v}}\bar{z}^{b_{v}}$ at $v\in S_{\infty}$. If $2a_{v,i}+a_{v}-b_{v}\neq 0$ for all $i$ and $v$, we define $I_{\imath_{v}}=\#\{i\mid 2a_{v,i}+a_{v}-b_{v}<0\}$ and $I=I(\Pi,\chi):=(I_{\imath})_{\imath\in \Sigma}$. If $m$ is critical for $\Pi\otimes \chi$, then we have:
$$ L^{S}(m,\Pi\otimes \chi) \sim_{E(\Pi)E_F(\chi)} (2\pi i)^{mdn}P^{(I)}(\Pi)\prod\limits_{\imath \in\Sigma}p(\widecheck{\chi},\imath)^{I_{\imath}}p(\widecheck{\chi},\overline{\imath})^{n-I_{\imath}}.$$
Interpreted as families, this relation is equivariant under the action of ${\rm Aut}(\C/F^{Gal})$. 
\end{thm}

\begin{thm}\label{Asai thm}
Let $\Pi$ be a cohomological conjugate self-dual cuspidal automorphic representation of $G_n(\A_F)$, which satisfies Hyp.\ \ref{descent}. We assume that Conj.\ \ref{nonvan} holds and that $\Pi$ is $(n-1)$-regular. Then, one has
\begin{equation}\label{Asai L-value 1}
L^{S}(1,\Pi,{\rm As}^{(-1)^{n}})\sim_{E(\Pi)}(2\pi i)^{dn(n+1)/2}\prod\limits_{\imath \in \Sigma}\prod\limits_{0\leq i\leq n}P^{(i)}(\Pi,\imath).
\end{equation}
Interpreted as families, this relation is equivariant under the action of ${\rm Aut}(\C/F^{Gal})$.
\end{thm}

\begin{proof}
By \cite{grob_lin}, Thm.\ 1.42 \& Thm.\ 4.17, we know that there is a certain Whittaker-period $p(\Pi)$ attached to $\Pi$ (cf.\ \cite{grob_lin}, Cor.\ 1.22 \& \S1.5.3), such that  
\begin{equation}\label{eq:q1}
L^{S}(1,\Pi,{\rm As}^{(-1)^{n}})\ \sim_{E(\Pi)} (2\pi i)^{dn} \ p(\Pi).
\end{equation}
If $\Pi$ is not at least $2$-regular, then here we use that Conj.\ \ref{nonvan} implies Conj.\ 4.31 from \cite{grob_lin}. Combining Thm.\ \ref{thm:artihap} and \cite{jie-thesis}, Cor.\ $7.5.1$, there exists an archimedean factor $Z(\Pi_{\infty})$ such that 
\begin{equation}\label{eq:q2}
p(\Pi)\sim_{E(\Pi)} Z(\Pi_{\infty})\prod\limits_{\imath\in \Sigma}\prod\limits_{1\leq i\leq n-1}P^{(i)}(\Pi,\imath)\sim_{E(\Pi)} Z(\Pi_{\infty})\prod\limits_{\imath\in \Sigma}\prod\limits_{0\leq i\leq n}P^{(i)}(\Pi,\imath)
\end{equation}
where the last equation follows from equation \eqref{local end 2}. Now, let $\Pi^{\#}$ be any cohomological conjugate self-dual cuspidal automorphic representation of $G_{n-1}(\Acm)$, which satisfies Hyp.\ \ref{descent} and assume that the coefficient modules in cohomology attached to the pair $(\Pi,\Pi^{\#})$ satisfy the analogue of the branching-law \eqref{branching law}. If $\Pi$ is at least $2$-regular, then we may choose $\Pi^{\#}$ such that the inequalities are all strict and so $\Pi^\#$ is then also at least $2$-regular. Then, it follows from \S\ref{sect:critRS}, that there is a critical point $s_0=\tfrac{1}{2}+m$ of $L(s,\Pi\times\Pi^{\#})$ with $m\neq 0$. By Prop.\ $9.4.1$ of \cite{jie-thesis}, for any such critical point, we have
$$p(m,\Pi_{\infty},\Pi_{\infty}^{\#})Z(\Pi_\infty)Z(\Pi^{\#}_\infty)\sim_{E(\Pi)E(\Pi)^{\#}} (2\pi i)^{d(m+\tfrac{1}{2})n(n-1)}$$
where $p(m,\Pi_{\infty},\Pi_{\infty}^{\#}))$ is the bottom-degree archimedean Whittaker period attached to $s_0=\tfrac{1}{2}+m$, defined in Cor.\ 1.17 of \cite{grob_lin} and which is calculated in Cor.\ $4.30$ of the \textit{loc.cit}. More precisely, we have 
$$p(m,\Pi_{\infty},\Pi_{\infty}^{\#})\sim_{E(\Pi)E(\Pi^{\#})} (2\pi i)^{mdn(n-1)-\tfrac{1}{2}d(n-1)(n-2)}.
$$
We then deduce that $$Z(\Pi_\infty)Z(\Pi^{\#}_\infty)\sim_{E(\Pi)E(\Pi^{\#})} (2\pi i)^{\tfrac{1}{2}dn(n-1)+\tfrac{1}{2}d(n-1)(n-2)}.$$
By induction on $n$, we obtain that $Z(\Pi_\infty)\sim_{E(\Pi)} (2\pi i)^{\tfrac{1}{2}dn(n-1)}$. Recalling that every relation above is equivariant under the action of ${\rm Aut}(\C/F^{Gal})$, the result follows combining \eqref{eq:q1} and \eqref{eq:q2}. Finally, if $\Pi$ is not at least $2$-regular, then we invoke Conj.\ \ref{nonvan} under which Prop.\ $9.4.1$ of \cite{jie-thesis} still holds for $p(0,\Pi_{\infty},\Pi_{\infty}^{\#})$ and we conclude as in the case of $s_0\neq \tfrac{1}{2}$.
\end{proof}

\begin{rmk}
The above theorem in fact proves an automorphic variant of Deligne's conjecture for the Asai motive attached to $\Pi$. We refer to Cor.\ 1.3.5 of \cite{harrisANT} and Sect.\ 3 of \cite{harris-lin} for the calculation of Deligne's period in this case. We remark that the condition that $\Pi$ is $(n-1)$-regular can be replaced by a milder regularity condition by considering $\Pi$ of auxiliary type.
\end{rmk}

\begin{thm} \label{automorphic Deligne central}
Let $\Pi$ (resp. $\Pi'$) be a cohomological conjugate self-dual cuspidal automorphic representation of $G_{n}(\Acm)$ (resp. $G_{n'}(\Acm)$) which satisfies Hyp.\ \ref{descent}. We assume that Conj.\ \ref{nonvan} holds. If $n\nequiv n' \mod 2$ and if $\Pi_\infty$ is $(n-1)$-regular and $\Pi'_\infty$ is $(n'-1)$-regular, then
\begin{equation}\label{eq:mainthm1}
L^S(\tfrac{1}{2},\Pi\times \Pi') \sim_{E(\Pi)E(\Pi')} (2\pi i)^{dnn'/2} \prod\limits_{\imath \in \Sigma}\left(\prod\limits_{i=0}^{n}P^{(i)}(\Pi,\imath)^{sp(i,\Pi;\Pi',\imath)}\prod\limits_{j=0}^{n'}P^{(j)}(\Pi',\imath)^{sp(j,\Pi';\Pi,\imath)}\right).
\end{equation}
Interpreted as families, this relation is equivariant under the action of ${\rm Aut}(\C/F^{Gal})$.
\end{thm}
\begin{proof}
The assertion is trivially true, if $L^S(\tfrac{1}{2}, \Pi \times \Pi')= 0$. So, let us suppose that $L^S(\tfrac{1}{2}, \Pi \times \Pi') \neq 0$. Then, by Prop.\ \ref{existence of chi}, there exists an odd integer $N$ and algebraic Hecke characters $\chi_{i}$, $\chi'_{j}$, such that the isobaric sums $\tau$ and $\tau'$, as defined in \S \ref{sect:endo}, satisfy the conditions of Prop.\ \ref{mult1}, and such that $\tau\otimes\tau'$ is in piano position. We may hence apply Prop.\ \ref{good position central value} to this non-vanishing central value:  
\begin{equation}\label{JS1}
L^S(\tfrac{1}{2},\tau\times \tau') \sim_{E(\tau)E(\tau')} (2\pi i)^{-dN(N+1)/2}L^S(1,\tau,{\rm As}^{(-1)^{N}})L^S(1,\tau',{\rm As}^{(-1)^{N-1}}).
\end{equation}
In what follows, we will use $i$ to indicate an integer $0\leq i\leq n$, $j$ to indicate an integer $0\leq j\leq n'$, $k$ to indicate an integer $1\leq k\leq N-n$ and $l$ to indicate an integer $1\leq l\leq N-1-n'$. Recall the algebraic Hecke character $\psi$ from \S\ref{sect:fields}. Moreover, we abbreviate the compositum of all the number fields $E_{F}(\chi_{i})$, $E_{F}(\chi'_{k})$ and $E_{F}(\psi)$ as defined in \S\ref{sect:CM} by $E_{\sf char}$.\\\\
The left hand side of equation \ref{JS1}  then equals
\begin{eqnarray}\label{JS2} 
&&L^S(\tfrac{1}{2},\Pi \times \Pi')\prod\limits_{l=1}^{N-1-n'}L^S(\tfrac{1}{2},\Pi\otimes \chi_{l}^{\prime}\eta^{-1})\prod\limits_{k=1}^{N-n}L^S(\tfrac{1}{2},\Pi'\otimes \chi_{k})\prod\limits_{k=1}^{N-n}\prod\limits_{l=1}^{N-1-n'}L^S(\tfrac{1}{2},\chi_{k} \chi_{l}^{\prime}\eta^{-1})\nonumber\\
&=& L^S(\tfrac{1}{2},\Pi \times \Pi')\prod\limits_{l=1}^{N-1-n'}
L^S(1,\Pi\otimes \chi_{l}^{\prime}\psi^{-1})\prod\limits_{k=1}^{N-n}
L^S(\tfrac{1}{2},\Pi'\otimes \chi_{k})\prod\limits_{k=1}^{N-n}\prod\limits_{l=1}^{N-1-n'}L^S(1,\chi_{k} \chi_{l}^{\prime}\psi^{-1}).\nonumber
\end{eqnarray}
One can verify easily that all values, which appear above, are critical values of the respective $L$-function. By Thm.\ \ref{CM}, Thm.\ \ref{n*1} and Thm.\ \ref{thm:artihap}, we obtain that
 $L^S(\tfrac{1}{2},\tau\times \tau') $ is hence in relation to a product of the following form
\begin{equation}
(2\pi i)^{C}L^S(\tfrac{1}{2},\Pi \times \Pi')\prod\limits_{\imath\in \Sigma}  \Big[
\prod\limits_{i=0}^{n}
P^{(i)}(\Pi,\imath)^{S_{i,\imath}} \prod\limits_{j=0}^{n'}
P^{(j)}(\Pi',\imath)^{T_{j,\imath}} \label{left hand side}
\prod\limits_{k=0}^{N-n}p(\widecheck{\chi_{k}},\imath)^{D_{k,\imath}}\prod\limits_{l=0}^{N-1-n'} p(\widecheck{\chi'_{l}},\imath)^{E_{l,\imath}} p(\widecheck{\psi},\imath)^{F_{\imath}} \Big]
\end{equation}
for certain exponents $C$, $S_{i,\imath}$, $T_{j,\imath}$, $D_{k,\imath}$, $E_{l,\imath}$ and $F_{\imath}$. Here we already factored the CM-periods and replaced those at conjugate embeddings $\overline\imath$ using Lem.\ \ref{Lemma CM}. On the other hand, by Lem.\ $3.3$ of \cite{grob_lin}, 
\begin{eqnarray}
&&L^S(1,\tau,{\rm As}^{(-1)^{N}})\nonumber\\
&\sim_{E(\Pi)E_{\sf char}} & L^S(1,\Pi,{\rm As}^{(-1)^{n}}) \prod_{k=1}^{N-n}L^S(1,\chi_{k},{\rm As}^{-}) \prod_{k=1}^{N-n}L^S(1,\Pi\otimes \chi_{k}^{c})\prod\limits_{1\leq k<k'\leq N-n}L^S(1,\chi_{k}\chi_{k'}^{c})\nonumber\\
&\sim_{E(\Pi)E_{\sf char}} & (2\pi i)^{d(N-n)} L^S(1,\Pi,{\rm As}^{(-1)^{n}})  \prod_{k=1}^{N-n}L^S(1,\Pi\otimes \chi_{k}^{-1})\prod\limits_{1\leq k<k'\leq N-n}L^S(1,\chi_{k}\chi_{k'}^{-1}).\nonumber
\end{eqnarray}
The last equation is due to the fact that $L^S(1,\chi_{k},{\rm As}^-)=L^S(1,\varepsilon_{F/F^{+}})\sim_{\cm^{Gal}}(2\pi i)$ for each $k$, as it follows from the proof of Lem.\ 1.30 of \cite{grob_lin}. Similarly, 
\begin{eqnarray}
&&L^S(1,\tau',{\rm As}^{(-1)^{N-1}})\nonumber\\
&\sim_{E(\Pi')E_{\sf char}} & L^S(1,\Pi',{\rm As}^{(-1)^{n'}}) \prod_{l=1}^{N-1-n'}L^S(1,\chi'_{l},{\rm As}^{-1})
\prod_{l=1}^{N-1-n'}L^S(1,\Pi'\otimes (\chi'_{l}\eta^{-1})^{c}) \nonumber\\
& & \times\prod\limits_{1\leq l<l'\leq N-1-n'} L^S(1,\chi_{l}^{\prime}\chi_{l'}^{\prime,c}) \nonumber\\
&\sim_{E(\Pi')E_{\sf char}} & (2\pi i)^{d(N-n'-1)} L^S(1,\Pi',{\rm As}^{(-1)^{n'}})\prod_{l=1}^{N-1-n'}L^S(1,\Pi'\otimes (\chi'_{l})^{-1}\eta) \prod\limits_{1\leq l<l'\leq N-1-n'} L^S(1,\chi_{l}^{\prime}(\chi'_{l'})^{-1})\nonumber\\
&\sim_{E(\Pi')E_{\sf char}} & (2\pi i)^{d(N-n'-1)} L^S(1,\Pi',{\rm As}^{(-1)^{n'}})\prod_{l=1}^{N-1-n'}L^S(\tfrac{1}{2},\Pi'\otimes (\chi'_{l})^{-1}\psi) \prod\limits_{1\leq l<l'\leq N-1-n'} L^S(1,\chi_{l}^{\prime}(\chi'_{l'})^{-1}).\nonumber
\end{eqnarray}
By Thm.\ \ref{Asai thm}, Thm.\ \ref{CM}, Thm.\ \ref{n*1}, Thm.\ \ref{thm:artihap}, and Lem.\ \ref{Lemma CM}
$$(2\pi i)^{-dN(N+1)/2}L^S(1,\tau,{\rm As}^{(-1)^{N}})L^S(1,\tau',{\rm As}^{(-1)^{N-1}})$$ is hence in relation to a product of the following form\begin{equation}\label{right hand side}
(2\pi i)^{c}\prod\limits_{\imath\in \Sigma}  \Big[
\prod\limits_{i=0}^{n}
P^{(i)}(\Pi,\imath)^{s_{i,\imath}} \prod\limits_{j=0}^{n'}
P^{(j)}(\Pi',\imath)^{t_{j,\imath}}
\prod\limits_{k=0}^{N-n}p(\widecheck{\chi_{k}},\imath)^{d_{k,\imath}}\prod\limits_{l=0}^{N-1-n'} p(\widecheck{\chi'_{l}},\imath)^{e_{l,\imath}} p(\widecheck{\psi},\imath)^{f_{\imath}} \Big]
\end{equation}
for certain exponents $c$, $s_{i,\imath}$, $t_{j,\imath}$, $d_{k,\imath}$, $e_{l,\imath}$ and $f_{\imath}$.\\\\ 
It therefore remains to show that
\begin{eqnarray}\label{compare 1}
c-C&=&\frac{dnn'}{2} \\\label{compare 2}
s_{i,\imath}-S_{i,\imath}&=& sp(i,\Pi;\Pi',\imath) \text{ for all } i,\imath\\\label{compare 3}
t_{j,\imath}-T_{j,\imath}&=& sp(j,\Pi';\Pi,\imath) \text{ for all } j,\imath\\\label{compare 4}
d_{k,\imath}-D_{k,\imath}&=&0 \text{ for all } k,\imath\\\label{compare 5}
e_{l,\imath}-E_{l,\imath}&=&0 \text{ for all } l,\imath\\ \label{compare 6}
f_{\imath}-F_{\imath}&=&0 \text{ for all } \imath 
\end{eqnarray}
We now calculate these exponents explicitly.\\\\ 
We first deal with the exponents in (\ref{left hand side}). For each $1\leq l\leq N-1-n'$, by Thm.\ \ref{CM}, Thm.\ \ref{n*1}, Thm.\ \ref{thm:artihap}, and Lem.\ \ref{Lemma CM} we have:
\begin{eqnarray}\nonumber
L^S(1,\Pi\otimes \chi_{l}^{\prime}\psi^{-1})&\sim_{E(\Pi)E_{\sf char}}& (2\pi i)^{dn}\prod\limits_{\imath \in \Sigma}
\big[P^{(I(\Pi,\chi'_{l}\psi^{-1})_{\imath})}(\Pi,\imath)
p(\widecheck{\chi'_{l}\psi^{-1}},\imath)^{I(\Pi,\chi'_{l}\psi^{-1})_{\imath}}p(\widecheck{\chi'_{l}\psi^{-1}},\overline{\imath})^{n-I(\Pi,\chi'_{l}\psi^{-1})_{\imath}}\big]\\
&\sim_{E(\Pi)E_{\sf char}}& (2\pi i)^{dn}\prod\limits_{\imath \in \Sigma}
\big[P^{(I(\Pi,\chi'_{l}\psi^{-1})_{\imath})}(\Pi,\imath)
p(\widecheck{\chi'_{l}},\imath)^{I(\Pi,\chi'_{l}\psi^{-1})_{\imath}}p(\widecheck{\chi'_{l}},\overline{\imath})^{n-I(\Pi,\chi'_{l}\psi^{-1})_{\imath}}\times \nonumber\\
&&p(\widecheck{\psi},\imath)^{-I(\Pi,\chi'_{l}\psi^{-1})_{\imath}}p(\widecheck{\psi},\overline{\imath})^{-n+I(\Pi,\chi'_{l}\psi^{-1})_{\imath}}\big]
\nonumber\\
&\sim_{E(\Pi)E_{\sf char}}& (2\pi i)^{dn}\prod\limits_{\imath \in \Sigma}
\big[P^{(I(\Pi,\chi'_{l}\psi^{-1})_{\imath})}(\Pi,\imath)
p(\widecheck{\chi'_{l}},\imath)^{2I(\Pi,\chi'_{l}\psi^{-1})_{\imath}-n}\times \nonumber\\
&&
p(\widecheck{\psi},\imath)^{-2I(\Pi,\chi'_{l}\psi^{-1})_{\imath}+n}(2\pi i)^{-n+I(\Pi,\chi'_{l}\psi^{-1})_{\imath}}\big]
\nonumber\\
&\sim_{E(\Pi)E_{\sf char}}& (2\pi i)^{dn/2-\sum\limits_{\imath\in \Sigma}(-2I(\Pi,\chi'_{l}\psi^{-1})_{\imath}+n)/2}\prod\limits_{\imath \in \Sigma}
\big[P^{(I(\Pi,\chi'_{l}\psi^{-1})_{\imath})}(\Pi,\imath)
p(\widecheck{\chi'_{l}},\imath)^{2I(\Pi,\chi'_{l}\psi^{-1})_{\imath}-n}\times\nonumber\\
&&
p(\widecheck{\psi},\imath)^{-2I(\Pi,\chi'_{l}\psi^{-1})_{\imath}+n}\big]
\nonumber.
\end{eqnarray}
Similarly, for each $1\leq k\leq N-n$, 
\begin{eqnarray}\nonumber
L^S(\tfrac{1}{2},\Pi'\otimes \chi_{k})&\sim_{E(\Pi')E_{\sf char}}& (2\pi i)^{dn'/2}\prod\limits_{\imath \in \Sigma}
\big[P^{(I(\Pi',\chi_{k})_{\imath})}(\Pi',\imath)
p(\widecheck{\chi_{k}},\imath)^{I(\Pi',\chi_{k})_{\imath}}p(\widecheck{\chi_{k}},\overline{\imath})^{n'-I(\Pi',\chi_{k})_{\imath}}\big]\\
&\sim_{E(\Pi')E_{\sf char}}& (2\pi i)^{dn'/2}\prod\limits_{\imath \in \Sigma}
\big[P^{(I(\Pi',\chi_{k})_{\imath})}(\Pi',\imath)
p(\widecheck{\chi_{k}},\imath)^{2I(\Pi',\chi_{k})_{\imath}-n'}\big]
\end{eqnarray}
Moreover, applying Thm.\ \ref{CM} and Lem.\ \ref{Lemma CM}, for each $1\leq k\leq N-n$ and $1\leq l\leq N-1-n'$ we get:
\begin{eqnarray}
L^S(1,\chi_{k} \chi_{l}^{\prime}\psi^{-1})&\sim_{E_{\sf char}}& (2\pi i)^{d}p(\widecheck{\chi_{k}\chi_{l}^{\prime}\psi^{-1}},\Psi_{\chi_{k} \chi_{l}^{\prime}\psi^{-1}})\nonumber\\
&\sim_{E_{\sf char}}&(2\pi i)^{d}p(\widecheck{\chi_{k}}, \Psi_{\chi_{k} \chi_{l}^{\prime}\psi^{-1}})p(\widecheck{\chi'_{l}},\Psi_{\chi_{k} \chi_{l}^{\prime}\psi^{-1}})p(\widecheck{\psi}, \Psi_{\chi_{k} \chi_{l}^{\prime}\psi^{-1}})^{-1}\nonumber\\
&\sim_{E_{\sf char}}& (2\pi i)^{d}\prod\limits_{\imath\in\Sigma}\big[
p(\widecheck{\chi_{k}}, \imath)^{\epsilon_{k,l,\imath}}p(\widecheck{\chi'_{l}},\imath)^{\epsilon_{k,l,\imath}}p(\widecheck{\psi}, \imath)^{-\epsilon_{k,l,\imath}}(2\pi i)^{(\epsilon_{k,l,\imath}-1)/2}\big]\nonumber\\
&\sim_{E_{\sf char}}& (2\pi i)^{d/2-(\sum\limits_{\imath\in \Sigma}-\epsilon_{k,l,\imath}/2)}\prod\limits_{\imath\in\Sigma}\big[
p(\widecheck{\chi_{k}}, \imath)^{\epsilon_{k,l,\imath}}p(\widecheck{\chi'_{l}},\imath)^{\epsilon_{k,l,\imath}}p(\widecheck{\psi}, \imath)^{-\epsilon_{k,l,\imath}}\big]
\nonumber
\end{eqnarray}
where $\epsilon_{k,l,\imath}=1$ if $\imath\in \Psi_{\chi_{k} \chi_{l}^{\prime}\psi^{-1}}$ and $\epsilon_{k,l,\imath}=-1$ otherwise. We observe immediately that 
\begin{eqnarray}\label{U observe}
F_{\imath}&=&-\sum\limits_{l=1}^{N-1-n'}E_{l,\imath},\label{U_imath}\\
C&=&(N-1-n')\frac{dn}{2}+(N-n)\frac{dn'}{2}+(N-n)(N-1-n')\frac{d}{2}-\sum\limits_{\imath\in \Sigma}\frac{F_{\imath}}{2}\nonumber\\
&=& \frac{dN(N-1)}{2}-\frac{dnn'}{2}-\sum\limits_{\imath\in \Sigma}\frac{F_{\imath}}{2}.\label{Cconstant}
\end{eqnarray}
We also have:
\begin{eqnarray}
S_{i,\imath}&=&\#\{l\mid I(\Pi,\chi_{l}'\psi^{-1})_{\imath}=i\}\\
T_{j,\imath}&=&\#\{k\mid I(\Pi',\chi_{k})_{\imath}=j\}\\
D_{k,\imath}&=& 2I(\Pi',\chi_{k})_{\imath}-n' +\sum\limits_{1\leq l\leq N-1-n'}\epsilon_{k,l,\imath}\\
E_{l,\imath}
 &=&2I(\Pi,\chi'_{l}\psi^{-1})_{\imath}-n+\sum\limits_{1\leq k\leq N-n}\epsilon_{k,l,\imath}
\end{eqnarray}
We now deal with the exponents in (\ref{right hand side}). By Thm.\ \ref{Asai thm}, 
\begin{eqnarray}
L^S(1,\Pi,{\rm As}^{(-1)^{n}})&\sim_{E(\Pi)}&(2\pi i)^{dn(n+1)/2}\prod\limits_{\imath \in \Sigma}\prod\limits_{0\leq i\leq n}P^{(i)}(\Pi,\imath);\nonumber\\
L^S(1,\Pi',{\rm As}^{(-1)^{n'}})&\sim_{E(\Pi')}&(2\pi i)^{dn'(n'+1)/2}\prod\limits_{\imath \in \Sigma}\prod\limits_{0\leq j\leq n'}P^{(j)}(\Pi',\imath)\nonumber.
\end{eqnarray}
Again by Thm.\ \ref{CM}, Thm.\ \ref{n*1}, Thm.\ \ref{thm:artihap}, and Lem.\ \ref{Lemma CM} we have

\begin{eqnarray}
L^S(1,\Pi\otimes \chi_{k}^{-1})&\sim_{E(\Pi)E_{\sf char}}& (2\pi i)^{dn}\prod\limits_{\imath\in \Sigma}
\Big[ P^{I(\Pi,\chi_{k}^{-1})_{\imath}}(\Pi,\imath)p(\widecheck{\chi_{k}},\imath)^{-I(\Pi,\chi_{k}^{-1})_{\imath}}p(\widecheck{\chi_{k}},\bar{\imath})^{I(\Pi,\chi_{k}^{-1})_{\imath}-n} \Big]\nonumber\\
&\sim_{E(\Pi)E_{\sf char}}& (2\pi i)^{dn}\prod\limits_{\imath\in \Sigma}
\Big[ P^{I(\Pi,\chi_{k}^{-1})_{\imath}}(\Pi,\imath)p(\widecheck{\chi_{k}},\imath)^{-2I(\Pi,\chi_{k}^{-1})_{\imath}+n} \Big]\nonumber
\end{eqnarray}

\begin{eqnarray}
L^S(\tfrac{1}{2},\Pi'\otimes (\chi'_{l})^{-1}\psi) &\sim_{E(\Pi')E_{\sf char}} & (2\pi i)^{dn'/2} 
\prod\limits_{\imath\in \Sigma}
\Big[ P^{I(\Pi',(\chi'_{l})^{-1}\psi)_{\imath}}(\Pi',\imath)p(\widecheck{\chi'_{l}},\imath)^{-I(\Pi',(\chi'_{l})^{-1}\psi)_{\imath}}\times \nonumber\\
&&p(\widecheck{\chi'_{l}},\bar{\imath})^{I(\Pi',(\chi'_{l})^{-1}\psi)_{\imath}-n'}p(\widecheck{\psi},\imath)^{I(\Pi',(\chi'_{l})^{-1}\psi)_{\imath}}p(\widecheck{\psi},\bar{\imath})^{n'-I(\Pi',(\chi'_{l})^{-1}\psi)_{\imath}}
 \Big]\nonumber\\
 &\sim_{E(\Pi')E_{\sf char}} & (2\pi i)^{dn'/2} 
\prod\limits_{\imath\in \Sigma}
\Big[ P^{I(\Pi',(\chi'_{l})^{-1}\psi)_{\imath}}(\Pi',\imath)p(\widecheck{\chi'_{l}},\imath)^{-2I(\Pi',(\chi'_{l})^{-1}\psi)_{\imath}+n'}\nonumber\times \\
&&p(\widecheck{\psi},\imath)^{2I(\Pi',(\chi'_{l})^{-1}\psi)_{\imath}-n'}(2\pi i)^{n'-I(\Pi',(\chi'_{l})^{-1}\psi)_{\imath}}
 \Big]\nonumber \\
 &\sim_{E(\Pi')E_{\sf char}} & (2\pi i)^{dn'-\sum\limits_{\imath\in\Sigma}(2I(\Pi',(\chi'_{l})^{-1}\psi)_{\imath}-n')/2} 
\prod\limits_{\imath\in \Sigma}
\Big[ P^{I(\Pi',(\chi'_{l})^{-1}\psi)_{\imath}}(\Pi',\imath)\nonumber\times \\
&&p(\widecheck{\chi'_{l}},\imath)^{-2I(\Pi',(\chi'_{l})^{-1}\psi)_{\imath}+n'}p(\widecheck{\psi},\imath)^{2I(\Pi',(\chi'_{l})^{-1}\psi)_{\imath}-n'}
 \Big]\nonumber 
 \end{eqnarray}
 
\begin{eqnarray}
 L^S(1,\chi_{k}\chi_{k'}^{-1})&\sim_{E_{\sf char}}& (2\pi i)^{d}p(\widecheck{\chi_{k}},\Psi_{\chi_{k}\chi_{k'}^{-1}}) p(\widecheck{\chi_{k'}},\Psi_{\chi_{k}\chi_{k'}^{-1}}) ^{-1} \nonumber\\
 &\sim_{E_{\sf char}}& (2\pi i)^{d}\prod\limits_{\imath\in \Sigma}
 p(\widecheck{\chi_{k}},\imath)^{\eta_{k,k',\imath}} p(\widecheck{\chi_{k'}},\imath) ^{-\eta_{k,k',\imath}} \nonumber\\\nonumber
  L^S(1,\chi'_{l}(\chi'_{l})^{-1})&\sim_{E_{\sf char}}& (2\pi i)^{d}p(\widecheck{\chi'_{l}},\Psi_{\chi'_{l}(\chi'_{l'})^{-1}}) p(\widecheck{\chi_{l'}},\Psi_{\chi'_{l}(\chi'_{l'})^{-1}}) ^{-1}\\&\sim_{E_{\sf char}}& (2\pi i)^{d}\prod\limits_{\imath\in \Sigma}p(\widecheck{\chi'_{l}},\imath)^{\xi_{l,l',\imath}} p(\widecheck{\chi_{l'}},\Psi_{\chi'_{l}(\chi'_{l'})^{-1}}) ^{-\xi_{l,l',\imath}} .\nonumber
\end{eqnarray}
where $\eta_{k,k',\imath}=1$, if $\imath\in \Psi_{\chi_{k}\chi_{k'}^{-1}}$ and $\eta_{k,k',\imath}=-1$ otherwise;  $\xi_{l,l',\imath}=1$, if $\imath\in \Psi_{\chi'_{l}(\chi'_{l'})^{-1}}$ and $\xi_{l,l',\imath}=-1$ otherwise. This gives
\begin{eqnarray}\label{u observe}
f_{\imath}&=&-\sum\limits_{l=1}^{N-1-n'}e_{l,\imath},\label{fff}\\
c&=& -\frac{dN(N+1)}{2}+d(N-n)+\frac{dn(n+1)}{2}+(N-n)dn+\frac{(N-n)(N-n-1)d}{2}+\nonumber\\
&&d(N-n'-1)+\frac{dn'(n'+1)}{2}+(N-m-1)dn'+\frac{(N-1-n')(N-2-n')d}{2}-\sum\limits_{\imath\in \Sigma}\frac{f_{\imath}}{2}\nonumber\\
&=& \frac{dN(N-1)}{2}-\sum\limits_{\imath\in \Sigma}\frac{f_{\imath}}{2}\label{cconstant}
\end{eqnarray}
%where the last equation follows from:
%\begin{eqnarray}
%&&d(N-n)+\frac{dn(n+1)}{2}+(N-n)dn+\frac{(N-n)(N-n-1)d}{2}\nonumber\\
%&=& \frac{d}{2}\big[
%dn(n+1)+(N-n)(2+2n+N-n-1)
%\big]\nonumber\\
%&=& \frac{d}{2}\big[
%n(n+1)+(N-n)(N+n+1)
%\big]\nonumber\\
%&=& \frac{d}{2}\big[
%n(n+1)+(N-n)(n+1)+(N-n)N
%\big]\nonumber\\
%&=& \frac{d}{2}\big[
%N(n+1)+(N-n)N
%\big]=\frac{dN(N+1)}{2}\nonumber
%\end{eqnarray}
%and similarly $d(N-m-1)+\frac{dm(m+1)}{2}+(N-m-1)dm+\frac{(N-1-m)(N-2-m)d}{2}=\frac{dN(N-1)}{2}$.\\
We also obtain
\begin{eqnarray}
s_{i}&=& 1+\#\{k\mid I(\Pi,\chi_{k}^{-1})_{\imath}=i\} \\
t_{j}&=& 1+\#\{l\mid I(\Pi',(\chi'_{l})^{-1}\psi)_{\imath}=j\}\\
 d_{k,\imath}&=&-2I(\Pi,\chi_{k}^{-1})_{\imath}+n+\sum\limits_{k'>k}\eta_{k,k',\imath}-\sum\limits_{k'<k}\eta_{k',k,\imath}\\
 e_{l,\imath}&=&-2I(\Pi',(\chi'_{l})^{-1}\psi)_{\imath}+m+\sum\limits_{l'>l}\xi_{l,l',\imath}-\sum\limits_{l'<l}\xi_{l',l,\imath}.
\end{eqnarray}
Comparing (\ref{Cconstant}) with (\ref{cconstant}) and (\ref{fff}) with (\ref{U observe}), we see that Eq.\ (\ref{compare 5}) implies equations (\ref{compare 6}) and (\ref{compare 1}). Hence, it remains to prove the identities (\ref{compare 2}) -- (\ref{compare 5}), which are all local.\\\\
We hence fix an $\imath=\imath_{v}\in\Sigma$ and drop the subscript $\imath$ for simplicity. We write the infinity type of $\Pi$ (resp. $\Pi'$) at $v$ as $\{z^{a_{i}} \bar{z}^{-a_{i}}\}_{1\leq i\leq n}$ (resp. $\{z^{b_{j}} \bar{z}^{-b_{j}}\}_{1\leq j\leq n'}$) with $b_{j}$ strictly decreasing. For each $k$ (resp $l$), We write the infinity type of $\chi_{k}$ (resp. $\chi'_{l}$) at $v$ as $z^{x_{k}}\bar{z}^{-x_{k}}$ (resp. $z^{y_{l}}\bar{z}^{-y_{l}}$). \\\\
Then, the infinity type of $\tau$ (resp.\ $\tau'$) at $v$ is $\{z^{a_{i}} \bar{z}^{-a_{i}}\}_{1\leq i\leq n}\cup \{z^{x_{k}}\bar{z}^{-x_{k}}\}_{1\leq k\leq N-n}$ (resp.\ $\{z^{b_{j}} \bar{z}^{-b_{j}}\}_{1\leq j\leq n'}\cup \{z^{y_{l}-\tfrac{1}{2}}\bar{z}^{-y_{l}+\tfrac{1}{2}}\}_{1\leq l\leq N-1-n'}$). We define 
$$\mathcal{A}:=\{a_{i},x_{k}\mid 1\leq i\leq n, 1\leq k\leq N-n\}$$$$\text{and }\mathcal{B}:=\{b_{j},y_{l}-\tfrac{1}{2}\mid 1\leq j\leq n', 1\leq l\leq N-1-n'\}.$$
%We may assume that $x_{1}>x_{2}>\cdots >x_{N-n}$ and $y_{1}>y_{2}>\cdots >y_{N-1-m}.$\\
For each $l$, we note that $I(\Pi,\chi'_{l}\psi^{-1})=\#\{i \mid a_{i}+y_{l}-\tfrac{1}{2}<0\}=\#\{i \mid a_{i}<-y_{l}+\tfrac{1}{2}\}$. Consequently, for each $0\leq i\leq n$, 
 $I(\Pi,\chi'_{l}\psi^{-1})=n-i$ if and only if $a_{i}>-y_{l}+\tfrac{1}{2}>a_{i+1}$. Hence 
 \begin{equation}\label{S_n-j}S_{n-i}=\#\{l\mid a_{i}>-y_{l}+\tfrac{1}{2}>a_{i+1}\}.
 \end{equation}
Moreover, by definition of $\Psi_{\chi_{k} \chi_{l}^{\prime}\psi^{-1}}$ in Theorem \ref{CM} we have 

 $\epsilon_{k,l,\imath}=1$ if $x_{k}+y_{l}-\tfrac{1}{2}<0$, $=-1$ otherwise. Hence the coefficient

 \begin{eqnarray} E_{l}\nonumber
 &=&2I(\Pi,\chi'_{l}\psi^{-1})-n+\sum\limits_{1\leq k\leq N-n}\epsilon_{k,l}\\
 &=&2\#\{i \mid a_{i}<-y_{l}+\tfrac{1}{2}\}-n +\#\{k\mid x_{k}<-y_{l}+\tfrac{1}{2}\}-\#\{k\mid x_{k}>-y_{l}+\tfrac{1}{2}\}\nonumber\\
  &=&\#\{i \mid a_{i}<-y_{l}+\tfrac{1}{2}\}-\#\{i \mid a_{i}>-y_{l}+\tfrac{1}{2}\} +\#\{k\mid x_{k}<-y_{l}+\tfrac{1}{2}\}-\#\{k\mid x_{k}>-y_{l}+\tfrac{1}{2}\}\nonumber\\
  &=&\#\{A\in \mathcal{A} \mid A<-y_{l}+\tfrac{1}{2}\}-\#\{A\in \mathcal{A}  \mid A>-y_{l}+\tfrac{1}{2}\}. \label{F_l}
  \end{eqnarray}
 Similarly, we get 
 
\begin{eqnarray}
T_{n'-j}&=&\#\{k\mid b_{j}>-x_{k}>b_{j+1}\}; \label{T} \label{T_n'-j}\\
D_{k}&=&\#\{B\in \mathcal{B}\mid -B>x_{k}\}-\#\{B \in \mathcal{B} \mid -B<x_{k}\} .\label{D_k}
\label{D}
\end{eqnarray}
For each $k$, $I(\Pi,\chi_{k}^{-1})=\#\{i\mid a_{i}<x_{k}\}$. In particular, for each $0\leq i\leq n$, $I(\Pi,\chi_{k}^{-1})=n-i$ if and only if $a_{i}>x_{k}>a_{i+1}$. We get $s_{n-i}=1+\#\{k\mid a_{i}>x_{k}>a_{i+1}\}$. Hence,

\begin{eqnarray}
s_{n-i}-S_{n-i}&=&1+\#\{k\mid a_{i}>x_{k}>a_{i+1}\}-\#\{l\mid a_{i}>-y_{l}+\tfrac{1}{2}>a_{i+1}\}\nonumber\\
&=& 1+\#\{A\in \mathcal{A}\mid a_{i}>A>a_{i+1}\}-\#\{l\mid a_{i}>-y_{l}+\tfrac{1}{2}>a_{i+1}\}.
\end{eqnarray}
Since $\tau\otimes \tau'$ is in piano position, we know the pair $(\mathcal{A},\mathcal{B})$ gives rise to two strings of numbers, which satisfy the branching condition (cf.\ Lem.\ \ref{branching condition}). In particular, $1+\#\{A\in \mathcal{A}\mid a_{i}>A>a_{i+1}\}= \# \{ B\in \mathcal{B}\mid a_{i}>-B>a_{i+1}\}$. Therefore,

\begin{eqnarray}
s_{n-i}-S_{n-i}&=& \# \{ B\in \mathcal{B}\mid a_{i}>-B>a_{i+1}\}-\#\{l\mid a_{i}>-y_{l}+\tfrac{1}{2}>a_{i+1}\}\nonumber\\
&=& \#\{j\mid a_{i}>-b_{j}>a_{i+1}\}\nonumber\\
&=& sp(n-i,\Pi;\Pi')
\end{eqnarray}
by Definition \ref{split automorphic}.\\\\
The constant $d_{k}=-2I(\Pi,\chi_{k}^{-1})+n+\sum\limits_{k'>k}\eta_{k,k'}-\sum\limits_{k'<k}\eta_{k',k}$. Note that for $k< k'$, by definition $\eta_{k,k'}=1$ if and only if $\imath\in \Psi_{\chi_{k}\chi_{k'}^{-1}}$ which is equivalent to $x_{k}<x_{k'}$. In particular, 

\begin{eqnarray}&&\sum\limits_{k'>k}\eta_{k,k'}-\sum\limits_{k'<k}\eta_{k',k}\nonumber\\
&=&\#\{k'>k\mid x_{k}<x_{k'}\}-\#\{k'>k\mid x_{k}>x_{k'}\}-\#\{k'<k\mid x_{k}>x_{k'}\}+\#\{k'<k\mid x_{k}<x_{k'}\}\nonumber\\
&=&\#\{k'\neq k\mid x_{k}<x_{k'}\}-\#\{k'\neq k\mid x_{k}>x_{k'}\}\nonumber
\end{eqnarray}
Hence, 

\begin{eqnarray}d_{k}&=&-2\#\{i\mid a_{i}<x_{k}\}+n+\#\{k'\neq k\mid x_{k}<x_{k'}\}-\#\{k'\neq k\mid x_{k}>x_{k'}\}\nonumber\\
 &=&-\#\{i\mid a_{i}<x_{k}\}+\#\{i\mid a_{i}>x_{k}\}+\#\{k'\neq k\mid x_{k}<x_{k'}\}-\#\{k'\neq k\mid x_{k}>x_{k'}\}\nonumber\\
 &=&-\#\{ A\in \mathcal{A}\mid A<x_{k}\}+\#\{ A\in \mathcal{A}\mid A>x_{k}\}. \label{d_k}
 \end{eqnarray}
 Again by the branching-law, we know $\#\{ A\in \mathcal{A}\mid A<x_{k}\}=\#\{ B\in \mathcal{B}\mid -B<x_{k}\}$ and $\#\{ A\in \mathcal{A}\mid A>x_{k}\}=\#\{ B\in \mathcal{B}\mid -B>x_{k}\}$. Comparing with (\ref{D_k}) we obtain that $d_{k}=D_{k}$.\\\\
The proof that $t_{n-j}-T_{n-j}=sp(j,\Pi';\Pi,\imath)$ and $e_{l}=E_{l}$ for each $j$ and $l$ is in complete analogy to the above and left to the reader.\\\\
This shows the identities (\ref{compare 2}) -- (\ref{compare 5}) and hence, as we already observed, also (\ref{compare 6}) and (\ref{compare 1}):
\begin{equation*}
c-C=\left(\frac{dN(N-1)}{2}-\sum\limits_{\imath\in \Sigma}\frac{f_{\imath}}{2}\right)-\left(\frac{dN(N-1)}{2}-\frac{dnn'}{2}-\sum\limits_{\imath\in \Sigma}\frac{F_{\imath}}{2}\right)=\frac{dnn'}{2}.
\end{equation*}
Finally, comparing (\ref{left hand side}) and (\ref{right hand side}) we obtain that 
 \begin{equation*}
L^{S}(\tfrac{1}{2},\Pi\otimes \Pi') \sim_{E(\Pi)E(\Pi')E_{\sf char}} (2\pi i)^{dnn'/2} \prod\limits_{\imath \in \Sigma}\left(\prod\limits_{i=0}^{n}P^{(i)}(\Pi,\imath)^{sp(i,\Pi;\Pi',\imath)}\prod\limits_{j=0}^{n'}P^{(j)}(\Pi',\imath)^{sp(j,\Pi';\Pi,\imath)}\right).
\end{equation*}
Observe that, interpreted as families of complex numbers, both sides of this relation only depend on the embeddings of $E(\Pi)$ and $E(\Pi')$, so we can remove $E_{\sf char}$ from the relation by Lem.\ 1.34 in \cite{grob_lin}. This finishes the proof.

\end{proof}

\subsection{The automorphic variant of Deligne's conjecture: General $n$, $n'$ and $s_0$.}

We may now complete the proof of Conj.\ \ref{main conjecture} for the respective families of automorphic representations $\Pi$ and $\Pi'$, treated in Thm.\ \ref{automorphic Deligne near central} ($n\equiv n'$ mod $2$) and Thm.\ \ref{automorphic Deligne central} ($n\notequiv n'$ mod $2$), by extending the statements of the aforementioned two results to all critical values $s_0$ of $L(s,\Pi\times\Pi')$. This will be a direct application of the following theorem of Raghuram, cf.\ \cite{ragh19}, Thm.\ 110.(ii), which extends the results for totally real fields, dealt with in \cite{harder-ragh}, to general totally imaginary fields. We remark that the special case of $n'=n-1$ and $\Pi\otimes\Pi'$ in piano position has already been established by \cite{grob_lin}, Thm.\ 5.5, whereas for general even $n$ and odd $n'$, $n>n'$, the result may be found (under some additional hypotheses on $\Pi\otimes\Pi'$) as Cor.\ 4.4 in \cite{gro-sach}:

\begin{thm}\label{CM:ragh}
Let $n,n'\geq 1$ be integers and let $\Pi$ (resp.\ $\Pi'$) be a cohomological conjugate self-dual cuspidal automorphic representation of $GL_n(\A_F)$ (resp. $GL_{n'}(\A_F)$). The ratio of any consecutive critical values $s_0,s_0+1$ of $L(s,\Pi\times\Pi')$, such that $L^S(s_0+1,\Pi\times\Pi')\neq 0$, satisfies
$$\frac{L^S(s_0,\Pi\times\Pi')}{L^S(s_0+1,\Pi\times\Pi')}\sim_{E(\Pi) E(\Pi')} (2\pi i)^{dnn'}.$$
Interpreted as families, this relation is equivariant under the action of ${\rm Aut}(\C/F^{Gal})$. 
\end{thm}
This result enables us to extend Thm.\ \ref{automorphic Deligne near central} and Thm.\ \ref{automorphic Deligne central}, where the central (resp.\ near-central) critical point of $L(s,\Pi\times\Pi')$ was considered, to all critical points. We summarize this as our first main theorem:

\begin{thm} \label{automorphic Deligne general} 
Let $n,n'\geq 1$ be integers and let $\Pi$ (resp. $\Pi'$) be a cohomological conjugate self-dual cuspidal automorphic representation of $G_{n}(\Acm)$ (resp. $G_{n'}(\Acm)$), which satisfies Hyp.\ \ref{descent}. If $n\equiv n' \mod 2$, we assume that $\Pi$ and $\Pi'$ satisfy the conditions of Thm.\ \ref{automorphic Deligne near central}, i.e., that the isobaric sum $(\Pi\eta^n)\boxplus (\Pi'^{c}\eta^n)$ is $2$-regular and that either $\Pi$ and $\Pi'$ are both $5$-regular or $\Pi$ and $\Pi'$ are both regular and satisfy Conj.\ \ref{nonvan}. Whereas if $n\nequiv n' \mod 2$, we assume that $\Pi$ and $\Pi'$ satisfy the conditions of Thm.\ \ref{automorphic Deligne central}, i.e., we assume Conj.\ \ref{nonvan} and suppose that $\Pi_\infty$ is $(n-1)$-regular and $\Pi'_\infty$ is $(n'-1)$-regular.

Then the automorphic version of Deligne's conjecture, cf.\ Conj.\ \ref{main conjecture}, is true.
\end{thm}

\section{Proof of the factorization}\label{sect:proofThm61}

\subsection{Statement of the main theorem on factorization}\label{sect:thm}
We shall resume the notation from \S\ref{reviewH14}. In particular, we assume to have fixed a real embedding $\imath_{v_0}$ of $F^+$ and denote by $H=H_{I_0}$ the attached unitary group. Given a highest weight $\lambda$, we obtained $n$ cohomological discrete series representations $\pi_{\lambda,q}$, $0 \leq q \leq n-1$ of $H_\infty$, which were distinguished by the property that their $(\q,K_{H,\infty})$-cohomology is concentrated in degree $q$.\\\\
Now, let $\Pi$ be a cohomological conjugate self-dual cuspidal automorphic representation of $G_n(\A_F)$, which satisfies Hyp.\ \ref{descent}. For the same reason as in \S\ref{construction of motive}, we shall descend $\Pi^{{\sf v}}$ instead of $\Pi$. So, for each $q$ as above, we are given a cohomological tempered cuspidal automorphic representation $\pi(q)\in \prod(H,\Pi^{\sf v})$ with archimedean component $\pi_{\lambda,q}$. By Prop.\ \ref{h,K-coh} it has multiplicity one in the square-integrable automorphic spectrum. Finally, recall the number field $E(\pi(q))\supseteq E_Y(\eta)$ from \S\ref{sect:E(pi)} and let us abbreviate $E_q(\Pi):=E(\Pi) E(\pi(q))$. We are now ready to state our second main theorem:

\begin{thm}\label{auto-facto} 
Let $\Pi$ be a cohomological conjugate self-dual cuspidal automorphic representation of $G_n(\A_F)$, which satisfies Hyp.\ \ref{descent} and let ${\xi}_{\Pi}$ be its central character. We assume that $\Pi_\infty$ is $(n-1)$-regular. Let $\pi(q)\in \prod(H,\Pi^{\sf v})$ be a cohomological tempered cuspidal automorphic representation with archimedean component $\pi_{\lambda,q}$. Moreover we suppose that Conj.\ \ref{nonvan} (non-vanishing of twisted central critical values) and Conj.\ \ref{lvarch} (rationality of archimedean integrals) are valid. Then, for each $0\leq q\leq n-2$, 
$$Q(\pi(q)) \sim_{E_q(\Pi)}  p(\widecheck{\xi}_{\Pi},\Sigma)^{-1} \cfrac{P^{(q+1)}(\Pi,\imath_{v_0}) }{ P^{(q)}(\Pi,\imath_{v_0})}.$$
Interpreted as families, this relation is equivariant under the action of ${\rm Aut}(\C/F^{Gal})$. 
\end{thm}
Before we give a proof of Thm.\ \ref{auto-facto}, let us make several remarks and derive an important consequence.\\\ 
Firstly, this theorem establishes a version of the factorization of periods which was conjectured in \cite{H97}, see Conj.\ 2.8.3 and Cor.\ 2.8.5 {\it loc.\ cit.}. A proof of this conjecture (up to an unspecified product of archimedean factors) when $\cm=\mathcal K$ is imaginary quadratic was obtained in \cite{H07}, based on an elaborate argument involving the theta correspondence and under a certain regularity hypothesis. The more general argument, which we will give here, is much shorter and more efficient (but evidently depends on the hypotheses of Thm.\ \ref{auto-facto}).\\\\ 
Secondly, our theorem will imply the desired factorization, cf.\ \eqref{eq:factor}, of the local arithmetic automorphic periods $P^{(i)}(\Pi,\imath)$ as follows:

\begin{defn}\label{local period definition}
Let $\Pi$ and $\pi(q)$ be as in the previous theorem. We define
$$ P_{i}(\Pi,\imath):= \left\{ \begin{array}{rcl}
         P^{(0)}(\Pi,\imath)& \mbox{if} & i=0; \\
         Q(\pi(i-1)) \ p(\widecheck{\xi}_{\Pi},\Sigma) & \mbox{if}
         & 1\leq i\leq n-1; \\  
P^{(n)}(\Pi,\imath)\prod\limits_{i=0}^{n-1}P_{i}(\Pi,\imath)^{-1} & \mbox{if} & i=n.                \end{array}\right.$$
\end{defn}
Moreover, for any $0\leq i\leq n$, let $E^{(i)}(\Pi)$ be the compositum of the number fields $E_0(\Pi)$ and $E_q(\Pi)$, $q\leq i-1$. Then, 

\begin{thm}\label{main factorization}  
Under the hypotheses of Thm.\ \ref{auto-facto}, the Tate relation \eqref{eq:fundrel} is true. More previsely, we obtain the following factorization
\begin{equation}\label{Pfac}
P^{(i)}(\Pi,\imath) \sim_{E^{(i)}(\Pi)} P_{0}(\Pi,\imath)P_{1}(\Pi,\imath)\cdots P_{i}(\Pi,\imath)
\end{equation}
and in addition for each $i$ and $\imath$
\begin{equation}\label{eq:Tate}
P_{i}(\Pi,\imath)\sim_{E_{i}(\Pi)}Q_{i}(M(\Pi),\imath),
\end{equation}
for the motive $M(\Pi)$ attached to $\Pi$ as constructed in Thm.\ \ref{thm:clozel}. Interpreted as families, both relations are equivariant under the action of ${\rm Aut}(\C/F^{Gal})$. 
\end{thm}
\begin{proof}
Given Thm.\ \ref{auto-facto}, the factorization \eqref{Pfac} follows directly. We now prove \eqref{eq:Tate}: For $i=0$, by equation (\ref{local end}) we have $P^{(0)}(\Pi,\imath)\sim_{E(\Pi)} p(\widecheck{\xi}_{\Pi},\bar\imath) \sim_{E(\Pi)} p(\widecheck{\xi_{\Pi}^{c}},\imath)$. By Def.\ $3.1$ of \cite{harris-lin}, Eq.\ (2.12) of \cite{linorsay} and Eq.\ (6.13) of \cite{jie-thesis}, we know $$Q_{0}(M(\Pi),\imath) \sim_{E(\Pi)} (2\pi i)^{n(n-1)/2}\delta(M(\Pi),\imath)\sim_{E(\Pi)}\delta(M(\xi_{\Pi}),\imath)\sim_{E(\Pi)}  p(\widecheck{\xi_{\Pi}^{c}},\imath)$$ as expected.

For each $1\leq i\leq n-1$, by Rem.\ $3.5$ of \cite{H21} (see also Rem.\ \ref{Finfty}), we know 
$$Q_{i}(M(\Pi),\imath) \sim_{E_{i}(\Pi)} Q(\pi(i-1)) q(M(\Pi)),$$ 
where $q(M(\Pi))$ is the period defined in Lem.\ $4.9$ of \cite{grob-harr}. We see immediately from this lemma that $q(M(\Pi))\sim_{E(\Pi)}\prod\limits_{\imath\in \Sigma}((2\pi i)^{n(n-1)/2}\delta(M(\Pi),\imath))^{-1}$. Similarly as above we have 
$$ (2\pi i)^{n(n-1)/2}\delta(M(\Pi),\imath)\sim_{E(\Pi)}  p(\widecheck{\xi_{\Pi}^{c}},\imath).$$ 
Hence $q(M(\Pi))\sim_{E(\Pi)} \prod\limits_{\imath\in \Sigma} p(\widecheck{\xi_{\Pi}^{c}},\imath)^{-1}\sim_{E(\Pi)}  p(\widecheck{\xi_{\Pi}},\Sigma)$ and $Q_{i}(M(\Pi),\imath) \sim_{E_{i}(\Pi)} P_{i}(\Pi,\imath)$ as expected.

It remains to show that $\prod\limits_{i=0}^{n}Q_{i}(M(\Pi),\imath)\sim_{E(\Pi)} P^{(n)}(\Pi,\imath)$.  By Lem.\ $1.2.7$ of \cite{harrisANT} we have $\prod\limits_{i=1}^{n}Q_{i}(M(\Pi),\imath)\sim_{E(\Pi)}((2\pi i)^{n(n-1)/2}\delta(M(\Pi),\imath))^{-2}$. Hence, 
$$\prod\limits_{i=0}^{n}Q_{i}(M(\Pi),\imath)\sim_{E(\Pi)}((2\pi i)^{n(n-1)/2}\delta(M(\Pi),\imath))^{-1}\sim_{E(\Pi)}  p(\widecheck{\xi_{\Pi}^{c}},\imath)^{-1}\sim_{E(\Pi)} p(\widecheck{\xi_{\Pi}},\imath),$$ 
which is equivalent to $P^{(n)}(\Pi,\imath)$ by \eqref{local end}.

\end{proof}

%\subsubsection{Explanation}\label{explanation}   Corollary \ref{main factorization} is the automorphic version of the Tate relation \eqref{eq:fundrel}.  The period $P_i(\Pi,\imath)$ is identified with the motivic period $Q_i(M(\Pi),\imath)$ attached to
%the system of realizations in Theorem \ref{thm:clozel}; the normalization as in Definition \ref{local period definition} compensates for the discrepancy between
%complex conjugation and the operator $F_\infty$ that should be attached to the system of realizations.

\subsection{Proof of Thm.\ \ref{auto-facto}}
Our proof will proceed in several steps.

\subsubsection*{Step 1:}
Let us start off with the following 

\begin{obsv}\label{q=0}
Recall that $\pi_{\lambda,q}$ is holomorphic when $q=0$, cf.\ Lem.\ \ref{knownknown}. It thus follows directly from the definition that we have $Q(\pi(0))\sim_{E(\Pi) E(\pi(0))} P^{(I_{0})}(\Pi)$. Hence, by Thm.\ \ref{thm:artihap}, and Lem.\ \ref{Lemma CM},
\begin{eqnarray}
Q(\pi(0))&\sim_{E(\Pi) E(\pi(0))}& \left( \prod\limits_{\imath_v\neq \imath_{v_0}}P^{(0)}(\Pi,\imath_v)\right) P^{(1)}(\Pi,\imath_{v_0}) \nonumber \\
&\sim_{E(\Pi) E(\pi(0)) E_F(\widecheck{\xi}_{\Pi})}& \left( \prod\limits_{\imath_v\neq \imath_{v_0}}p(\widecheck{\xi}_{\Pi},\imath_v)^{-1}\right) P^{(1)}(\Pi,\imath_{v_0}) \nonumber\\
&\sim_{E(\Pi) E(\pi(0)) E_F(\widecheck{\xi}_{\Pi})}&\left( \prod\limits_{\imath_v\in \Sigma}p(\widecheck{\xi}_{\Pi},\imath_v)^{-1}\right) P^{(1)}(\Pi,\imath_{v_0})p(\widecheck{\xi},\bar{\imath}_{v_0})^{-1}\nonumber\\
&\sim_{E(\Pi) E(\pi(0)) E_F(\widecheck{\xi}_{\Pi})}& p(\widecheck{\xi}_{\Pi},\Sigma)^{-1}\cfrac{P^{(1)}(\Pi,\imath_{v_0})}{P^{(0)}(\Pi,\imath_{v_0})}.
\end{eqnarray}
However, by Lem.\ 1.34 in \cite{grob_lin}, we may reduce this relation to the smallest field containing $F^{Gal}$, on which all the quantities on both sides depend and remain well-defined. But this field is $E_0(\Pi)=E(\Pi) E(\pi(0))$. Therefore, Thm.\ \ref{auto-facto} is true when $q=0$. This is going to be used as the first step in our inductive argument.
\end{obsv}
Now, let $q$ be arbitrary. Then, in the notation of \S\ref{reviewH14}, the infinity type of $\Pi$ at $v$ is $\{z^{a_{v,i}} \bar{z}^{-a_{v,i}}\}_{1\leq i\leq n}$ where $a_{v,i}=-A_{v,n+1-i}$ (recall that we descend from $\Pi^{\sf v}$ rather than from $\Pi$ now). Next, let $\pi'(q)$ be the cohomological tempered cuspidal automorphic representation of $H'(\A_{F^+})$, constrcuted in Thm.\ \ref{BS}. By a direct calculation one gets that the $(\q',K_{H',\infty})$-cohomology of $\pi'(q)$ is non-vanishing only in degree $q':=n-q-2$.\\\\ 
Let $\Pi'=BC(\pi'(q)^{{\sf v}})$ be the base change of the contragredient of $\pi'(q)$. Then, the infinity type of $\Pi'$ at $v\in S_\infty$ is $\{z^{b_{v,j}} \bar{z}^{-b_{v,j}}\}_{1\leq j\leq n-1}$, with $b_{v,j}=A_{v,j}-\tfrac{1}{2}$ if either $v\neq v_{0}$, or $v=v_{0}$ and $j\neq q+1$, whereas $b_{v_{0},q+1}=A_{v_0,q+1}+\tfrac{1}{2}$. Hence, if we calculate the automorphic split indices of the pair $(\Pi,\Pi')$, cf.\ Def.\ \ref{split automorphic}, then we obtain

$$sp(i,\Pi;\Pi',\imath_v) =     \left\{ \begin{array}{rcl}
         1 & \mbox{if}
         & 1\leq i\leq n-1
 \\ 0  & \mbox{if} & i=0\text{ or }n
                \end{array}\right.;
 sp(j,\Pi';\Pi,\imath_v) =     \begin{array}{rcl}
         1 & \mbox{if}
         & 0\leq j\leq n-1
                \end{array}.
 $$
 at $v\neq v_{0}$, and
 $$sp(i,\Pi;\Pi',\imath_{v_0}) =     \left\{ \begin{array}{rcl}
         1 & \mbox{if}
         & 1\leq i\leq n-1,i\neq q, i\neq q+1
 \\ 0  & \mbox{if} & i=0, q+1 \text{ or }n\\
 2 & \mbox{if} & i=q
                \end{array}\right.;$$
$$sp(j,\Pi';\Pi,\imath_{v_0}) =   \left\{  \begin{array}{rcl}
         1 & \mbox{if}
         & 0\leq j\leq n-1, j\neq n-q-1, j\neq n-q-2\\
2 & \mbox{if}
         & j=n-q-2\\
0 & \mbox{if}
         & j=n-q-1
                \end{array}.\right.
 $$
at $v=v_{0}$. \\\\
We want to insert them into the formula provided by Thm.\ \ref{automorphic Deligne central}: As a first and obvious observation, it is clear by construction that, since $\Pi_\infty$ is $(n-1)$-regular, $\Pi'_\infty$ is $(n-2)$-regular. Combining Thm.\ \ref{BS}.(1) with Rem.\ \ref{rmk:des}, we also see that $\Pi'$ is a cuspidal automorphic representation, which satisfies Hyp.\ \ref{descent}. Therefore, $\Pi'$ satisfies the assumptions of Thm.\ \ref{automorphic Deligne central}, whence, inserting the values of the automorphic split indices from above into \eqref{eq:mainthm1}, we obtain

\begin{eqnarray}\label{RS L-value}
\nonumber
L^S(\tfrac{1}{2},\Pi\otimes \Pi')& \sim_{E(\Pi)E(\Pi')} & (2\pi i)^{dn(n-1)/2} \prod\limits_{\imath_v\in \Sigma}\left(\prod\limits_{1\leq i\leq n-1}P^{(i)}(\Pi,\imath_v) \prod\limits_{0\leq j\leq n-1}P^{(j)}(\Pi',\imath_v)\right)\times   \\ &&\cfrac{  P^{(q)}(\Pi,\imath_{v_0})P^{(n-q-2)}(\Pi',\imath_{v_0})    }{  P^{(q+1)}(\Pi,\imath_{v_0}) P^{(n-q-1)}(\Pi',\imath_{v_0})}.
 \end{eqnarray}
 The following observation is crucial for what follows: 
 
\begin{obsv}\label{q1}
$L^S(\tfrac{1}{2},\Pi\otimes \Pi')\neq 0$.
\end{obsv}
In order to see this, recall that by Thm.\ \ref{BS}.(2) there are factorizable cuspidal automorphic forms $f \in \pi(q)$, $f' \in \pi'(q)$, whose attached GGP-period does not vanish $\CP(f,f')\neq 0$. Hence, as all the local pairings $I^*_v(f_v,f'_v)$, cf.\ \S\ref{sect:GGPunit}, are convergent by the temperedness of $\pi(q)_v$ and $\pi'(q)_v$, it follows from the Ichino-Ikeda-N.Harris formula, Thm.\ \ref{conjecture II}, that necessarily
$$L^S(\tfrac{1}{2},BC(\pi(q))\otimes BC(\pi'(q)))\neq 0.$$
But since both $\Pi$ and $\Pi'$ are conjugate self-dual, we have
\begin{equation}\label{eq:l}
L^S(\tfrac{1}{2},BC(\pi(q))\otimes BC(\pi'(q)))=L^S(\tfrac{1}{2},\Pi^{{\sf v}}\otimes \Pi'^{{\sf v}}) =L^S(\tfrac{1}{2},\Pi^{c}\otimes \Pi'^{c}) = L^S(\tfrac{1}{2},\Pi\otimes \Pi').
\end{equation}
Therefore, indeed
$$L^S(\tfrac{1}{2},\Pi\otimes \Pi')\neq 0.$$

\subsubsection*{Step 2:}\label{analysis} 
We resume the notation from Step 1. Recall from the discussion below Thm.\ \ref{BS} that the factorizable cuspidal automorphic forms $f \in \pi(q)$, $f' \in \pi'(q)$ may be chosen such that, for all $v\in S_\infty$, $f_v$ (resp. $f'_v$) belongs to the $E(\pi(q))$- (resp. $E(\pi'(q))$-) rational subspaces of the minimal $K_{H,v}$-type of $\pi(q)_v$ (resp. $K_{H',v}$-type of $\pi'(q)_v$).\\\\
As in \S\ref{sect:GGPunit}, let $\xi$ be the Hecke character of $U(V_{1})(\A_{F^+})$ given by $\xi=(\xi_{ \pi'(q)} \xi_{\pi(q)})^{-1}$ and write $\pi''(q)=\pi'(q)\otimes \xi$. Let $f_{0}$ be a deRham-rational element of $\xi$. We define $f''=f'\otimes f_{0}$, a deRham-rational element in $\pi''(q)$.  Then, by Lem.\ \ref{add U1}, the GGP-period 
$$\CP(f,f'')=\cfrac{|I^{can}(f,f'')|^2}{\<f,f\>\ \<f'',f''\>}$$
satisfies
$$\CP(f,f')=\CP(f,f'').$$
Furthermore, by Thm.\ \ref{BS} and Thm.\ \ref{GPknownknown}, $I^{can}(f,f'')$ is a non-zero element of $E(\pi(q))E( \pi'(q))$. So, by our choice of $f$ and $f'$, Thm.\ \ref{conjecture II} and the very definition of the automorphic $Q$-periods attached to $\pi(q)$ and $\pi''(q)$, cf.\ \S\ref{Qperiod}, imply that
\begin{equation}\label{II1}
\begin{aligned}
\frac{1}{Q(\pi(q))Q(\pi''(q))} &\sim_{E(\pi(q)) E(\pi'(q))}   \Delta_{H} \frac{L^S(\tfrac{1}{2},\Pi^{\sf v}\otimes \Pi'^{\sf v})}{L^S(1,\Pi^{\sf v},{\rm As}^{(-1)^n})L^S(1,\Pi'^{\sf v},{\rm As}^{(-1)^{n-1}})} \ \prod_{v \in S_\infty} I^*_v(f_v,f'_v)  \\
&\sim_{E(\pi(q)) E(\pi'(q))}   (2\pi i)^{dn(n+1)/2} \frac{L^S(\tfrac{1}{2},\Pi\otimes \Pi')}{L^S(1,\Pi,{\rm As}^{(-1)^n})L^S(1,\Pi',{\rm As}^{(-1)^{n-1}})} \ \prod_{v \in S_\infty} I^*_v(f_v,f'_v).
\end{aligned}
\end{equation}
Here, we could remove the contragredient in the second line, as both $\Pi$ and $\Pi'$ are conjugate self-dual, whereas the replacement of $\Delta_H$ by a power of $2 \pi i$ is a consequence of (1.37) and (1.38) in \cite{grob_lin}, and the elimination of the local factors $I^*_v(f_v,f'_v)$ at the non-archimedean places follows from Lem.\ \ref{localIv}. At the archimedean places we make the following observation:

\begin{prop}\label{arch}  
Under the hypotheses of Thm.\ \ref{BS}, the local factors $I^*_v(f_v,f'_v) \neq 0$ for $v \in S_{\infty}$.
\end{prop}
\begin{proof}  This is an immediate consequence of the non-vanishing of the global period $\CP(f,f')$.
\end{proof}
Hence, as we are admitting Conj.\ \ref{lvarch}, we obtain
\begin{equation}\label{II1b}
\begin{aligned}
\frac{1}{Q(\pi(q))Q(\pi''(q))} &\sim_{E(\pi(q)) E(\pi'(q))}   (2\pi i)^{dn(n+1)/2} \frac{L^S(\tfrac{1}{2},\Pi\otimes \Pi')}{L^S(1,\Pi,{\rm As}^{(-1)^n})L^S(1,\Pi',{\rm As}^{(-1)^{n-1}})}.\end{aligned}
\end{equation}

\subsubsection*{Step 3:}
We recall from Step 1 above that $\Pi'$ is a cohomological conjugate self-dual cuspidal automorphic representation of $G_{n-1}(\A_F)$, which satisfies Hyp.\ \ref{descent} and is $(n-2)$-regular. Hence, both $\Pi$ and $\Pi'$ satisfy the conditions of Thm.\ \ref{thm:artihap} and Thm.\ \ref{Asai thm}. As a consequence, combining the relations \eqref{local end 2}, \eqref{Asai L-value 1} and \eqref{RS L-value}, one gets 

\begin{equation}\label{II2}
(2\pi i)^{dn(n+1)/2}\frac{L^S(\tfrac{1}{2},\Pi\otimes \Pi')}{L^S(1,\Pi,{\rm As}^{(-1)^n})L^S(1,\Pi',{\rm As}^{(-1)^{n-1}})} \sim_{E(\Pi)E(\Pi')}\cfrac{  P^{(q)}(\Pi,\imath_{v_0})P^{(n-q-2)}(\Pi',\imath_{v_0})    }{  P^{(q+1)}(\Pi,\imath_{v_0}) P^{(n-q-1)}(\Pi',\imath_{v_0})}.
\end{equation}
Recall that $L^S(\tfrac{1}{2},\Pi\otimes \Pi')\neq 0$, cf.\ Obs.\ \ref{q1}. This allows us to combine \eqref{II1b} with \eqref{II2}, and so, using the fact that $Q(\pi''(q)) \sim_{E(\pi'(q))} Q(\pi'(q)) \cdot Q(\xi)$, we arrive at the following conclusion:

\begin{equation}\label{II5}
%\begin{aligned}
\frac{1}{Q(\pi(q))Q(\pi'(q))Q(\xi)} \sim_{E_q(\Pi) E_{q'}(\Pi')}   \  \cfrac{  P^{(q)}(\Pi,\imath_{v_0})P^{(n-q-2)}(\Pi',\imath_{v_0})    }{  P^{(q+1)}(\Pi,\imath_{v_0}) P^{(n-q-1)}(\Pi',\imath_{v_0})}
%\end{aligned}
\end{equation}

\subsubsection*{Step 4:}

We need one last ingredient before we can complete the proof of Thm.\ \ref{auto-facto} by induction on the $F$-rank $n$:

\begin{lem} The following relation $$Q(\xi) \sim_{E_F(\xi)E_F(\widecheck{\xi}_{\Pi})E_F(\widecheck{\xi}_{\Pi'})}    p(\widecheck{\xi}_{\Pi},\Sigma)p(\widecheck{\xi}_{\Pi'}, \Sigma)$$
holds. Interpreted as families of complex numbers it is equivariant under ${\rm Aut}(\C)$.
\end{lem}

\begin{proof}

Recall that $U(V_{1})$ is the one-dimensional unitary group of signature $(1,0)$ at each $\imath\in \Sigma$. Let $T_1:=R_{F/\Q}(U(V_1))$. By definition of the CM-periods we have $Q(\xi)\sim_{E_F(\xi)} p(\xi, (T_1,h_{1}))$ where $h_{1}: R_{\C/\R}(\mathbb{G}_{m,\C})\rightarrow T_{1,\R}$ is the map, which sends $z$ to $z/\overline{z}$ at each $\imath\in\Sigma$. \\\\
We define a map $h_{\tilde{\Sigma}}: R_{\C/\R}(\mathbb{G}_{m,\C})\rightarrow T_{F,\R}$, where $T_{F}=R_{F/\Q}(\mathbb{G}_{m})$, by sending $z$ to $z/\overline{z}$ at each $\imath\in \Sigma$. The pair $(T_{F},h_{\tilde{\Sigma}})$ is then a Shimura datum. We extend $\xi$ to a character of $\Acm^{\times}$, still denoted by $\xi$. The natural inclusion $T_1\hookrightarrow T_{F}$ induces a map from the Shimura datum $(T_1,h_1)$ to $(T_{F},h_{\tilde{\Sigma}})$. By Prop.\ \ref{relation CM-period}, we have $p(\xi, (T_1,h_1))\sim_{E_F(\xi)} p(\xi, (T_{F},h_{\tilde{\Sigma}}))$.\\\\
Let $(T_{F},h_{\Sigma})$ and $(T_{F},h_{\overline{\Sigma}})$ be as in \S\ref{CM-periods}. Multiplication defines a map from $(T_{F}, h_{\tilde{\Sigma}})\times (T_{F},h_{\overline{\Sigma}})$ to $(T_{F}, h_{\Sigma})$. It follows from Prop.\ \ref{relation CM-period} (see also Prop.\ $1.4$ and Cor.\ $1.5$ of \cite{Harris93}), that we have 
$$
p(\xi, (T_{F},h_{\tilde{\Sigma}}))\sim_{E_F(\xi)} p(\xi,(T_{F},h_{\Sigma}))p(\xi,(T_{F},h_{\overline{\Sigma}}))^{-1}=p(\xi,\Sigma)p(\xi, \overline{\Sigma})^{-1}.$$
By Lem.\ \ref{Lemma CM},
$p(\xi,\Sigma)p(\xi, \overline{\Sigma})^{-1}\sim_{E_F(\xi)} p(\xi,\Sigma)p(\xi^{c,-1}, \Sigma)\sim_{E_F(\xi)} p(\xi/\xi^{c},\Sigma)$. Note that $\xi/\xi^{c}$ is the base change of the original $\xi$. Recall that $\Pi^{c}\cong \Pi^{{\sf v}}$ is the base change of $\pi(q)$. Hence $\widecheck{\xi}_{\Pi}$ is the base change of $\xi_{\pi(q)}^{-1}$. Similarly, $\widecheck{\xi}_{\Pi'}$ is the base change of $\xi_{\pi'(q)}^{-1}$. Therefore, $\xi/\xi^{c}=\widecheck{\xi}_{\Pi}\widecheck{\xi}_{\Pi}$. Consequently, recollecting all relations from above and invoking Lem.\ \ref{Lemma CM} once more, we get
\begin{equation}
Q(\xi)\sim_{E_F(\xi)} p(\widecheck{\xi}_{\Pi}\widecheck{\xi}_{\Pi'},\Sigma)\sim_{E_F(\xi)E_F(\widecheck{\xi}_{\Pi})E_F(\widecheck{\xi}_{\Pi'})} p(\widecheck{\xi}_{\Pi},\Sigma)p(\widecheck{\xi}_{\Pi'},\Sigma).
\end{equation}
\end{proof}
The previous Lem.\ and equation (\ref{II5}) now implies

\begin{equation}\label{II6}
%\begin{aligned}
Q(\pi(q))Q(\pi'(q)) \sim_{E_q(\Pi) E_{q'}(\Pi')}   \   \left( p(\widecheck{\xi}_{\Pi},\Sigma)^{-1}\cfrac{ P^{(q+1)}(\Pi,\imath_{v_0})}{P^{(q)}(\Pi,\imath_{v_0})}\right) \times  \left( p(\widecheck{\xi}_{\Pi'},\Sigma)^{-1} \cfrac{P^{(n-q-1)}(\Pi',\imath_{v_0})    }{   P^{(n-q-2)}(\Pi',\imath_{v_0})}\right)
%\end{aligned}
\end{equation}
Here, we could remove the number field  $E_F(\xi)E_F(\widecheck{\xi}_{\Pi})E_F(\widecheck{\xi}_{\Pi'})$ from the relation using \cite{grob_lin}, Lem.\ 1.19.\\\\
We may finish the proof of Thm.\ \ref{auto-facto} by induction on $n$. When $n=2$, the integer $q$ is necessarily $0$. The theorem is then clear by Observation \ref{q=0}. We assume that the theorem is true for $n-1\geq 2$. Again, if $q=0$, then the theorem follows from Observation \ref{q=0}. So, let $1\leq q\leq n-2$. Recall that our representation $\pi'(q)$ from above is an element in $ \prod(H',\Pi'^{{\sf v}})$ whose $(\q',K_{H',\infty})$-cohomology in concentrated in degree $n-q-2\leq n-3$. Moreover, we have verified above that $\Pi'$ is $(n-2)$-regular and satisifes Hyp.\ \ref{descent}, whence $\Pi'$ and $\pi'(q)$ satisfy the conditions of Thm.\ \ref{auto-facto}. Hence $Q(\pi'(q)) \sim_{E_{q'}(\Pi')}   \    p(\widecheck{\xi}_{\Pi'},\Sigma)^{-1} \cfrac{P^{(n-q-1)}(\Pi',\imath_{v_0})    }{   P^{(n-q-2)}(\Pi',\imath_{v_0})}$. The theorem then follows from equation (\ref{II6}) and \cite{grob_lin}, Lem.\ 1.19.

\vskip 10pt

\footnotesize
{\sc Harald Grobner: Fakult\"at f\"ur Mathematik, University of Vienna\\ Oskar--Morgenstern--Platz 1\\ A-1090 Vienna, Austria.}
\\ {\it E-mail address:} {\tt harald.grobner@univie.ac.at}

\vskip 10pt

{\sc Michael Harris: Department of
Mathematics, Columbia University, New York, NY  10027, USA. }
\\ {\it E-mail address:} {\tt harris@math.columbia.edu}

\vskip 10pt

{\sc Jie Lin: Universit\"at Duisburg-Essen, Fakult\"at f\"ur Mathematik, \\ Mathematikcarr\'ee, Thea-Leymann-Strasse 9 \\
45127 Essen, Germany}
\\ {\it E-mail address:} {\tt jie.lin@uni-due.de}

\end{document}